\documentclass[11pt]{article}
\usepackage[margin=1in]{geometry}
\usepackage{graphicx,subfigure}
\usepackage{hyperref}
\usepackage{amsmath}
\usepackage{amsthm}
\usepackage{amssymb}
\usepackage{mathrsfs}
\usepackage{bbm}
\usepackage{algpseudocode}
\usepackage{algorithm}
\usepackage{xcolor}
\usepackage{tikz,caption}
\usetikzlibrary{arrows.meta,calc}
 \usepackage{enumitem}
  \usepackage{multirow}
  \usepackage{booktabs,longtable}
 
\newtheorem{theorem}{Theorem}

\newtheorem{lemma}{Lemma}

\newtheorem{proposition}{Proposition}

\newtheorem{definition}{Definition}
\newtheorem{fact}{Fact}

\newtheorem{assumption}{Assumption}

\numberwithin{equation}{section}

\numberwithin{lemma}{section}
\numberwithin{corollary}{section}
\numberwithin{definition}{section}

\newtheorem{remark}{Remark}
\numberwithin{remark}{section}

\usepackage{booktabs}

\newcommand{\M}{\mathcal{M}}
\newcommand{\N}{\mathrm{N}}

\newcommand{\rmN}{\mathrm{N}}
\newcommand{\rmT}{\mathrm{T}}
\newcommand{\rmD}{\mathrm{D}}

\newcommand{\U}{\mathbf{U}}

\newcommand{\R}{\mathbb{R}}
\newcommand{\E}{\mathcal{E}}

\newcommand{\rmI}{\mathrm{I}}
\newcommand{\oM}{\mathcal{M}}

\newcommand{\be}{\begin{equation}}
\newcommand{\ee}{\end{equation}}
\newcommand{\ba}{\begin{array}}
\newcommand{\ea}{\end{array}}
\newcommand{\bad}{\begin{aligned}}
\newcommand{\ead}{\end{aligned}}

\newcommand{\normfro}[1]{\| #1 \|}

\newcommand{\argmin}{\mathop{\rm argmin}}

\newcommand{\secondform}{\mathrm{I\!I}}

\newcommand{\rank}{\mathrm{rank}}
\newcommand{\range}{\mathrm{range}}

\newcommand{\diag}{\mathrm{diag}}
\newcommand{\Diag}{\mathrm{Diag}}

\newcommand{\Retr}{\mathrm{Retr}}

\newcommand{\T}{\mathrm{T}}
 
\newcommand{\st}{\mathrm{s.t. }}

 \newcommand{\grad}{\mathrm{grad}}

\newcommand{\retr}{\mathrm{Retr}}

\newcommand{\Proj}{\mathcal{P}}

\usepackage[normalem]{ulem}

\begin{document}
	
	\title{Retractions by Alternating Projections}
	

\author{%
Shixiang Chen%
\thanks{School of Mathematical Sciences, Key Laboratory of the Ministry of Education for Mathematical Foundations and Applications of Digital Technology, University of Science and Technology of China, Hefei, Anhui, China. Yixiao He: \texttt{hyxxxx@mail.ustc.edu.cn}. Corresponding author: Shixiang Chen (\texttt{shxchen@ustc.edu.cn}). }%
\and
Yixiao He\protect\footnotemark[\value{footnote}]%
\and
Wen Huang%
\thanks{School of Mathematical Sciences, Xiamen University, Xiamen, China (Wen Huang: \texttt{wen.huang@xmu.edu.cn}).}%
}

	\date{\today}

	\maketitle

    \begin{abstract}
    
Alternating projections and their variants are classical tools for computing points in intersections of sets. Existing analyses for smooth manifolds mainly focus on local convergence rates under transversality or related regularity conditions. In this work, we develop a unified framework for a broad class of (possibly inexact) alternating-projection-type methods on intersections of smooth manifolds. Specifically, under the assumption that two $C^{2,1}$ embedded submanifolds $\mathcal{M}_1, \mathcal{M}_2 \subset \mathbb{R}^n$ intersect cleanly, we show that the associated alternating mapping admits a well-defined local limiting map $\psi$ on the intersection manifold $\mathcal{M}=\mathcal{M}_1\cap \mathcal{M}_2$, and that $\psi$ is a retraction on $\mathcal{M}$. If, in addition, $\mathcal{M}_1$ and $\mathcal{M}_2$ are $C^{3,1}$, then $\psi$ is a second-order retraction. Furthermore, the standard NewtonSLRA scheme, which exhibits quadratic local
behavior under transversality, can be understood as inducing a second-order
retraction on \(\M\). This framework thus provides   new retraction-based optimization tools for problems constrained to the intersection manifold.
	\end{abstract}
\noindent\textbf{Mathematics Subject Classification (2020):} 65K10 $\cdot$   58C05 $\cdot$ 49M37 $\cdot$ 90C26.\\
\noindent\textbf{Keywords:} retraction; second-order retraction; alternating projections; alternating-projection-type methods; manifold intersections; clean intersection; intrinsic transversality.
	
	\section{Introduction}
Alternating projections trace back to Schwartz~\cite{schwarz1869uber} and von Neumann~\cite{von1950geometry}, originally proposed to find a point in the intersection of two subspaces,
and were later generalized to convex sets~\cite{bregman1965method,gubin1967method,bauschke1993convergence,bauschke1996projection}.
When projecting onto the intersection is computationally prohibitive, alternating projections provide an efficient alternative.
Recent works establish local convergence rates   on nonconvex sets and manifolds under geometric regularity assumptions such as (intrinsic) transversality~\cite{lewis2009local,lewis2008alternating,drusvyatskiy2015transversality}, non-tangential intersections~\cite{andersson2013alternating,bauschke2013restricted,bauschke2013restricted-theory}, and separable intersection \cite{noll2016local}. Moreover, their inexact variants~\cite{kruger2016regularity,drusvyatskiy2019local} have also been proposed. In addition, Newton-type refinements can lead to local quadratic convergence, particularly in the structured low-rank approximation problem
\cite{schost2016quadratically,nagasaka2021relaxed} and for general nonconvex sets~\cite{xiao2025quadratically}.

While the above works primarily focus on local convergence rates of alternating projections,
a complementary question is how these updates can be used as  building blocks  for optimization algorithms on manifold intersections. To this end, a key ingredient in Riemannian optimization is a \emph{retraction},
which serves as a computationally tractable surrogate for the exponential map.
This replacement is particularly important on intersection manifolds, where the exponential map is
typically intractable, since it would require solving a nonlinear geodesic differential--algebraic system
arising from the coupled constraints. Algorithmically, retractions serve as the basic tool for mapping tangent directions back to the manifold,
and their smoothness is crucial for fast local convergence of Riemannian optimization algorithms. Projection-like retractions and their properties have been systematically developed in~\cite{absil2012projection}. 
Nevertheless, on intersection manifolds, the coupled constraints make it considerably more delicate to construct tractable retractions with the desired smoothness.
 Some existing work~\cite{schost2016quadratically,andersson2013alternating} shows that alternating-type algorithms provide sufficiently accurate approximations to the orthogonal projection, but does not establish that the induced limit maps form retractions, especially with the smoothness condition.
This motivates us to revisit alternating-projections schemes through the lens of retractions.

\paragraph{Problem setting.}
In this paper, we consider two embedded $C^{p,1}$ submanifolds $\M_1,\M_2 \subset \R^n$ with $p\ge 2$.
Here $C^{p,1}$ means that the manifolds admit $C^{p,1}$ atlases, namely, local coordinate maps that are $p$ times continuously differentiable and whose $p$th derivatives are locally Lipschitz.
While the local convergence analysis of alternating projections typically only requires $C^p$ regularity, the additional Lipschitz continuity is needed to establish the $C^{p-1}$ regularity of the induced retraction.

We assume that $\M_1$ and $\M_2$ intersect cleanly \cite{drusvyatskiy2015transversality}, meaning that the intersection
\[
\M := \M_1\cap \M_2
\]
is an embedded submanifold and, for any $\bar x\in \M$, one has
\[
\T_{\bar x}\M \;=\; \T_{\bar x}\M_1 \cap \T_{\bar x}\M_2 .
\]
Within this setting, we study a unified alternating-type iteration of the form
\[
\varphi := \phi_1\circ \phi_2,
\]
where $\phi_1:\R^n\to \M_1$ represents an inexact orthogonal projection map,
and $\phi_2:\R^n\to \R^n$ is     an   inexact projection mapping associated with $\M_2$. Note that $\phi_1$ takes values in $\M_1$. 
Given a reference point $\bar x\in \M$ and a tangent vector $\eta\in \T_{\bar x}\M$,
we associate the one-step map $\varphi$ with the limiting map
\begin{equation}\label{eq:limiting_alter_proj}\tag{Alt-R}
	\psi(\bar x,\eta)
	:= \lim_{k\to\infty}\varphi^{k}(\bar x+\eta),
\end{equation}
whenever the above limit is locally well-defined.

Our main contribution is a unified framework showing that a broad class of alternating-projection-type methods induces a retraction $\psi$ on the intersection manifold.
\begin{theorem}[Informal]
Let $\M_1,\M_2\subset \R^n$ be two embedded $C^{p,1}$ submanifolds ($p\ge 2$) that intersect cleanly at
$\bar x\in \M$.
\begin{enumerate}[leftmargin=*,label=(\roman*)]
\item
Any method from a broad class of locally $R$-linearly convergent methods---including the alternating projection method (APM)~\cite{lewis2008alternating,andersson2013alternating,drusvyatskiy2015transversality,noll2016local}
and inexact alternating projection methods (IAPM)~\cite{kruger2016regularity,drusvyatskiy2019local}---induces  a retraction  $\psi$ on $\M$ (in the sense of Definition~\ref{def:retraction}).
Furthermore, such a retraction is of second order if $p \ge 3$.

\item
Assume in addition that a transversality condition holds and $p\ge 3$.
Then the quadratically convergent scheme NewtonSLRA~\cite{schost2016quadratically}
induces a second-order retraction on $\M$. 

\end{enumerate}
\end{theorem}
 This theorem provides a unified characterization of alternating projections, inexact variants, and Newton-type local refinements through the induced map $\psi$. It also provides principled tools for  optimization over intersecting manifolds, enabling the use of second-order Riemannian algorithms such as the Riemannian Newton method \cite{smith1994optimization, Absil2009} and the Riemannian trust-region method \cite{absil2007trust}. Although first-order retractions already suffice for smooth optimization, second-order retractions provide more accurate pullback models and can substantially simplify the analysis of superlinear or quadratic local convergence \cite{Absil2009, boumal2023introduction}. In particular, recent work on nonsmooth Riemannian optimization \cite{wang2024adaptive} explicitly requires second-order retractions to achieve fast local convergence. Motivated by this perspective, we further develop a hybrid three-stage algorithm that preserves the retraction property, and numerical experiments on quadratic problems \cite{hou2025low} illustrate the practical benefit of the proposed retractions.

The main technical difficulty is that alternating-type schemes are generally not contractions in the ambient space, so a direct Banach fixed-point argument is unavailable for establishing the regularity of the limiting map. Our analysis relies on an orthogonal tangent--normal decomposition of the one-step error along the iterates. The normal component is contractive because of the clean-intersection geometry, whereas the tangential component is summable due to the $R$-linear decay of the iterates. This mechanism ensures that the limiting map $\psi$ is well defined and sufficiently smooth; in particular, it yields a retraction, and under stronger regularity, a second-order retraction.


\paragraph{Outline of this paper.}
Section~\ref{sec:prelim} collects notation and basic concepts on clean intersections, (intrinsic) transversality,  retraction, and metric regularity, etc.
Section~\ref{sec:alternating_projection} reviews existing alternating projection algorithms and their variants in our settings. Section~\ref{sec:opt_on_intersection} presents the main contribution of the work. We provide a unified framework to show that all related alternating-projections-type algorithms induce a retraction under certain regularity conditions.  Section~\ref{sec:new_alg} provides a hybrid algorithm that combines the strengths of alternating projections and Newton-type refinements. Section~\ref{sec:numerical_exp} presents numerical experiments illustrating the practical impact of the proposed retractions.

	\section{Notations and Background}\label{sec:prelim}
 \textbf{Notations.}
We work in a Euclidean space $\mathcal E$ (identified with $\R^n$ when the dimension needs to be specified). The Euclidean norm in $\E$ is denoted by $\|\,\cdot\,\|$.
Given a closed set $D\subset \E$, we write $\Proj_{D}$ for the orthogonal projection onto $D$, i.e.,
\[
\Proj_{D}(x) := \argmin_{y\in D}\|y-x\|,
\]
and
\[
d_{D}(x):=\|x-\Proj_{D}(x)\|
\]
for the distance from $x$ to $D$. Though the projection $\Proj_D$ could be set-valued, the projection onto a manifold $\Proj_{\M}$ is single-valued locally (see Proposition~\ref{prop:well-define_proj}). This paper mainly considers local behavior of alternating projections, hence we treat $\Proj_{\M}$ as a single-valued map by default.
The open ball in $\E$ centered at $x$ with radius $r$ is denoted by
\[
\mathbb{B}(x,r):=\{\,y\in\E:\ \|y-x\|<r\,\}.
\]
For a linear operator $\mathcal A$, we denote by $\mathcal A^*$ its adjoint and by $\|\mathcal A\|$ its operator norm.

 For a $C^p$ mapping $F$ defined in a Euclidean space $\mathcal E$, we denote its differential by $\mathrm D F$ and write
\[
\mathrm D F(x)[u] := \lim_{t\to 0}\frac{F(x+tu)-F(x)}{t},
\qquad x\in U\subset \mathcal E,\ u\in\mathcal E,
\]
whenever the limit exists.

\subsection{Riemannian manifold}
Let us review some standard concepts of  Riemannian geometry and Riemannian optimization. These can also be found in \cite{lee2012smooth,Absil2009,boumal2023introduction,lewis2008alternating}.


A mapping $F:U\to \mathcal E$, defined on an open set $U\subset \mathcal E$, is said to be of class $C^{p,1}$ if it is $p$ times continuously differentiable and its $p$-th derivative is locally Lipschitz.

\begin{definition}\label{def:local_def} 
Let $\mathcal{E}$ be a linear space of dimension $n$. A non-empty subset $\M$ of $\mathcal E$ is a $C^{p,1}$ embedded submanifold of $\mathcal E$ of dimension $d$ if either
	\begin{enumerate}
		\item $d=n$ and $\M$ is open in $\E$; or 
		\item $d=n-m$ for some $m\ge 1$ and, for each $x\in\M,$ there exists a neighborhood $U$ of $x$ in $\E$ and a $C^{p,1}$ function $h:U\to \R^{m}$ such that
        \begin{enumerate}
            \item If $y\in U$, then $h(y)=0$ if and only if $y\in \M$; and
            \item $\mathrm{rank} \rmD h(x)=m.$
        \end{enumerate}
	\end{enumerate}
    Such a function $h$ is called a local defining function for $\M$ at $x.$
\end{definition}
The above definition requires that the local defining function satisfies the linear independence constraint qualification(LICQ). 
Proposition~\ref{prop:submanifold} provides another approach to verify whether a subset is an embedded submanifold, which only requires the constant rank constraint qualification(CRCQ). 
\begin{definition}[{\cite[Theorem 5.12]{lee2012smooth}}]\label{prop:submanifold} 
    Let $\M$ be a smooth manifold, and let $h:\M\rightarrow \mathcal{E}$ be a smooth map with constant rank $c$, i.e., $\mathrm{rank}(\mathrm{D} h(x)) = c, \forall x \in \mathcal{M}$. Each level set of $h$ is a properly embedded submanifold of codimension $c$ in~$\M$. 
\end{definition}

By \cite[Prop.~5.38]{lee2012smooth}, if $h$ is a local defining function for $\mathcal{M}$, then for every $x \in \mathcal{M}  \cap U$,
\begin{equation}\label{eq:tangent-kernel}
	\T_x\mathcal{M}  =  \ker\big( \mathrm{D} h(x)\big), ~\text{where }~  \mathrm{D}h(x): \T_x \M \to \mathbb{R}^m .
\end{equation}
Meanwhile, one has
\[
\N_x\M = (\T_x\M)^\perp = \range(\rmD h(x)^*).
\]
A Riemannian manifold is a smooth manifold endowed with a Riemannian metric.
In this paper, we equip all embedded submanifolds with the Riemannian metric induced by the ambient Euclidean inner product.

For a $C^p$ map $F:\mathcal M\to \mathcal N$ between manifolds, the differential at $x\in\mathcal M$ is the linear map
$\mathrm D F(x):\T_x\mathcal M\to \T_{F(x)}\mathcal N$ characterized by
\[
\mathrm D F(x)[\xi] \;=\; \frac{\mathrm d}{\mathrm dt}\,F(\gamma(t))\Big|_{t=0},
\]
where $\xi\in\T_x\mathcal M$ and $\gamma:(-\varepsilon,\varepsilon)\to\mathcal M$ is any smooth curve with
$\gamma(0)=x$ and $\dot\gamma(0)=\xi$.

For $x\in \M$, the tangent space and the normal space of $\M$ at $x$ are denoted by $\T_x\M$ and $\N_x\M$, respectively.
For convenience, we set
\begin{equation}
	P_x:=\Proj_{\T_x\M},\qquad Q_x:=\Proj_{\N_x\M}.
\end{equation}

\subsection{Transversality and Intrinsic transversality}
Given two embedded submanifolds $\mathcal{M}_1,\mathcal{M}_2\subset \R^n$, their intersection $\mathcal{M}_1\cap \mathcal{M}_2$ need not be an embedded submanifold of $\R^n$.
A sufficient condition ensuring that $\mathcal{M}_1\cap \mathcal{M}_2$ is an embedded submanifold is \emph{transversality}; see Definition~\ref{def:tranverse},
which is adapted from \cite[Theorem~6.30 and Problem~6-10]{lee2012smooth}.

\begin{definition}\label{def:tranverse}
	Let $\M_1$ and $\M_2$ be two embedded submanifolds of $\R^n$. We say that $\M_1$ and $\M_2$ intersect transversely if at each  $x\in \M_1\cap\M_2,$  $\T_x\M_1 + \T_x\M_2 = \R^n,$ or equivalently, $\rmN_x\M_1\cap \rmN_x\M_2=\{0\}$. In this case, the intersection $\M=\M_1\cap \M_2$ is an embedded submanifold of $\R^n$ whose codimension is the sum of the codimension of $\M_1$ and $\M_2$. Moreover, it holds that $\T_x\M=\T_x\M_1\cap \T_x\M_2.$
\end{definition}
In view of Definition~\ref{def:local_def}, transversality amounts to the LICQ  condition for the joint equality-constrained representation of the intersection manifold. The transversality is strong, since it is necessary that $\dim(\T_x\M_1)+\dim(\T_x\M_2)\ge n.$  In contrast, in many applications one  has $\dim(\T_x\M_1)+\dim(\T_x\M_2)< n.$ This paper only assumes clean intersection, which corresponds to the CRCQ condition in Definition~\ref{prop:submanifold}. Formally, we assume that the following conditions hold throughout this paper.
\begin{assumption}[Clean intersection]\label{assumpt:M1andM2_intersect_cleanly}
	Let $\M_1,\M_2\subset \mathbb{R}^n$ be $C^{p,1}$ ($p\ge 2$) embedded submanifolds, and fix
	$\bar x\in \M_1\cap \M_2$.
	We assume that $\M_1$ and $\M_2$ intersect \emph{cleanly} at $\bar x$, namely,
	$\M:=\M_1\cap \M_2$ is a $C^{p,1}$ embedded submanifold in a neighborhood of $\bar x$ and
	there exists a neighborhood $\U_{\bar x}\subset \M$ of $\bar x$ such that
	\[
	\T_{x}\M=\T_{x}\M_1\cap \T_{x}\M_2,
	\qquad \forall x\in \U_{\bar x}.
	\]
\end{assumption}

This clean-intersection condition is often verifiable in structured matrix
problems. For instance, a recent variational analysis of determinantal varieties
derives tangent intersection rules for low-rank matrix sets intersected with
additional smooth constraints, including affine and orthogonally invariant
constraints; see, e.g., \cite{yang2025variational}. These examples illustrate
that the condition \(\T_x(\M_1\cap\M_2)=\T_x\M_1\cap \T_x\M_2\) naturally appears
in structured low-rank optimization.

The clean intersection coincides with the \emph{nontangential intersection} in the sense of \cite{andersson2013alternating}.
It is obvious that transversality implies a clean intersection. We also have the following regularity conditions. 

\begin{definition}\cite[Definition~3.1]{drusvyatskiy2015transversality}
	We say two manifolds $\M_1$ and $\M_2$ are intrinsically transverse at $\bar x$ if  there exist a neighborhood $U$ of $\bar x$
	and a constant $\kappa\in(0,1]$ such that for all
	$x\in (\M_1\cap U)\setminus \M_2$ and $y\in (\M_2\cap U)\setminus \M_1$, with
	$u:=\frac{x-y}{\|x-y\|}$, one has
	\[
	\max\big\{d_{\rmN_{y}\M_2}\big(u\big),\ d_{\rmN_{x}\M_1}\big(u\big)\big\}\ \ge\ \kappa.
	\]
\end{definition} 
 
 \begin{definition}[0-separability]\cite{noll2016local}
     We say $\M_1$ intersects $\M_2$   0-separably at $\bar x$ with  an angle  $\alpha>0$ if   there exists a neighborhood $U'$ of $\bar x$ such that,  	the building block		$z\in \M_1\setminus \M_2 \ \rightarrow\ x\in \Proj_{\M_2}(z)\setminus\M_1\ \rightarrow\ z^+\in \Proj_{\M_1}(x)$ satisfies that the angle between $z-x$ and $z'-x$ is at least $\alpha.$ 
 \end{definition}
  
 These two definitions arise in the analysis of alternating projections for general nonconvex sets. Note that 0-separability  is not symmetric for general settings. If $\M_2$ also intersects 0-separably with $\M_1$, we obtain a symmetric condition.  In this case, we say that $\M_1,\M_2$ intersect 0-separably.
The clean intersection implies {intrinsic transversality} at any $\bar x\in \M_1\cap \M_2$; see \cite[Section~3]{drusvyatskiy2015transversality}.  In turn, intrinsic transversality implies $\M_1,\M_2$ intersects 0-separably \cite[Proposition~2]{noll2016local}. In particular, if $\M_1$ is  super-regular (\cite[Def. 4.3]{drusvyatskiy2015transversality}) at $\bar x$ and $\M_1$ intersects $\M_2$ 0-separable  at $\bar x$,
then the inexact alternating projections  \cite[Algorithms 4\& 5]{drusvyatskiy2019local}
converges $R$-linearly to a point in $\M:=\M_1\cap\M_2$.
Note that both $\M_1,\M_2$ are a $C^2$ embedded submanifold of $\R^n$ and hence are super-regular at~$\bar x$~\cite{drusvyatskiy2019local}; in fact, it is prox-regular \cite{poliquin2000local}.

\subsection{Retraction} 
A retraction provides a computationally tractable approximation of the exponential map, which can be expensive to evaluate even on complete manifolds.
It is a fundamental tool in Riemannian optimization algorithms.
We first need to present the definition of the acceleration of a smooth curve on a manifold.
\begin{definition}
	Let $\M$ be a $C^p$ Riemannian submanifold embedded in $\R^n$ and  $c \colon I \to \mathcal{M}$ be a $C^{p-1}$ curve, where $I$ is an open interval in $\R$. Its velocity is the vector field $c' \in \mathfrak{X}(c)$, where  $\mathfrak{X}(c)$ denotes the set of smooth vector fields on the curve $c.$   When $p\ge 3$, the acceleration of $c$ is the   vector field $c'' \in \mathfrak{X}(c)$ defined by:
	\[
	c'' = \frac{\mathrm{D}}{\mathrm{d}t} c',
	\]
	where $\frac{\mathrm{D}}{\mathrm{d}t}: \mathfrak{X}(c)\rightarrow\mathfrak{X}(c)$ is  the  covariant derivative on $\M$.
	We also call $c''$ the intrinsic acceleration of $c$.
\end{definition}
 We write the classical (extrinsic) acceleration of $c$ in the embedding space as follows:
\[
\ddot{c} = \frac{\mathrm{d}^2}{\mathrm{d}t^2} c.
\]
Let $\nabla:\T\M\times \mathfrak{X}(\M)\rightarrow \T\M$ be the Riemannian connection defined by 
\begin{equation}\label{def:Levi-Civita}
	\nabla_\eta V = \Proj_{\T_X\M}(\mathrm{D}\bar{V}(X)[\eta]),
\end{equation}
where $V\in \mathfrak{X}(\M)$ is a smooth vector field on $\M$, $\bar{V}$ is the smooth extension of $V$ in $\mathcal E$. 
When $ \frac{\mathrm{D}}{\mathrm{d}t}$ is the  covariant derivative induced by the Riemannian connection \eqref{def:Levi-Civita}, one has (see \cite[(5.23)]{boumal2023introduction}):
\begin{equation}\label{eq:intrinsic_acc_and_ex_acc}
	c''(t)= \Proj_{\T_{c(t)}\M}(\ddot{c}(t) ).
\end{equation}
Since the derivative $c':I\to \T_{c(t)}\M$, in that spirit, we use notations $\dot{c}$ and $c'$ interchangeably for velocity since the two notions coincide.

\paragraph{Tangent bundle.}
The  tangent bundle  of $\M$ is defined by
\[
\T\M:=\{(x,\eta)\in \E\times\E:\ x\in\M,\ \eta\in \T_x\M\}.
\]
It is well known that $\T\M$ is itself a $C^{p-1}$ embedded submanifold of $\E\times\E$.
Hence, for a mapping $F:\mathcal D\to \E$ defined on an open set $\mathcal D\subset \T\M$,
its derivative at $(x,\eta)\in\mathcal D$ is understood in the standard sense of
embedded submanifolds: for any smooth curve $(x(t),\eta(t))\subset \T\M$ with
$(x(0),\eta(0))=(x,\eta)$,
\begin{equation}\label{def:differential_tangent_bundle}
\mathrm D F(x,\eta)\big[\dot x(0),\dot\eta(0)\big]
=\frac{\mathrm d}{\mathrm dt}\Big|_{t=0}F\big(x(t),\eta(t)\big).
\end{equation}
In particular, $\mathrm D F(x,\eta)$ is a linear map from the tangent space
$\T_{(x,\eta)}(\T\M)$ to $\E$.


\paragraph{Geodesic.} Since $\M=\M_1\cap\M_2\subset\E$ is a $C^2$ embedded submanifold.
For any $(x,\eta)\in \T\M$, there exists a (locally unique) geodesic $c(t), t\in I$ on $\M$ with $c(0)=x$ and $\dot c(0)=\eta$; this is a standard consequence of the local well-posedness of the geodesic ODE on a $C^2$ Riemannian manifold.
 Moreover, $c(t)$ is a geodesic if and only if its covariant acceleration vanishes,
$\nabla_{\dot c(t)}\dot c(t)=0, \forall t\in I$, equivalently $\Proj_{\T_{c(t)}\M}\ddot c(t)=0$.
Therefore,   if $\M$ admits a local defining function   $\M\cap U=\{x:\,H(x)=0\}$ with LICQ, then
there exists a multiplier $\lambda(t)$ such that
\begin{equation}\label{eq:geodesic_normal_form}
\ddot c(t)=\rmD H(c(t))^\top \lambda(t).
\end{equation}
On the other hand, the feasibility condition $H(c(t))\equiv 0$ yields, by differentiation,
\[
\rmD H(c(t))\,\dot c(t)=0,
\qquad
\rmD H(c(t))\,\ddot c(t)+\rmD^2 H(c(t))[\dot c(t),\dot c(t)]=0,
\]
where $\rmD^2H(c)[\dot c,\dot c]$ is the ambient second-order direction derivative.
Substituting \eqref{eq:geodesic_normal_form} into the second identity gives
\[
\big(\rmD H\,\rmD H^\top\big)(c(t))\,\lambda(t)
=-\,\rmD^2 H(c(t))[\dot c(t),\dot c(t)],
\]
and therefore the geodesic equation can be written as the explicit second-order ODE
\[
\ddot c(t)
=-\,\rmD H(c(t))^\top\Big(\rmD H(c(t))\,\rmD H\big(c(t)\big)^\top\Big)^{-1}
\rmD^2 H(c(t))[\dot c(t),\dot c(t)].
\]
This nonlinear differential system typically has no closed-form solution on intersection manifolds and is expensive to solve numerically, which motivates using computable retractions instead of the exponential map.

\begin{definition}[Retraction \cite{absil2012projection}]\label{def:retraction}
Let $\M$ be a $C^p(p\ge 2)$\footnote{Since $\T\M$ is $C^{p-1}$, any retraction on $\M$ is at most $C^{p-1}$ manifold.  Hence we require $p\ge 2$ \cite{absil2012projection}.} manifold.
	\begin{enumerate}
		\item A retraction on $\M$ is a $C^{p-1}$ map
		\[
		\mathrm{Retr} \colon \T\mathcal{M} \to \mathcal{M} \colon (x, \eta) \mapsto \mathrm{Retr}_x(\eta)
		\]
		{such that each curve $c(t): = \mathrm{Retr}_x(t\eta)$ satisfies $c(0) = x$ and $c'(0) = \eta$.}
		\item   A second-order retraction $\mathrm{Retr}$ on a  $C^p (p\ge 3)$ manifold $\mathcal{M}$ is a $C^{p-1}$-retraction such that, for all $x \in \mathcal{M}$ and all $\eta \in \T_x \mathcal{M}$, the curve $c(t) = \mathrm{Retr}_x(t\eta)$ has zero acceleration at $t = 0$, that is, $c''(0) = 0$.
	\end{enumerate}
\end{definition} 
  For the first-order condition, it is equivalent to that $\Retr$ is $C^1$ and for any $x\in \M$
\begin{equation}\label{eq:first-order-retraction}\tag{R1}
c(t) = x+t\eta+ o(t), \quad \text{as } t\to 0.
\end{equation}
This implies that the linearization along the tangent space approximates the curve well.
The second-order property is relaxed from the exponential map requirement to the weaker condition \(c''(0)=0\), rather than \(c''(t)=0\) for all \(t\).
 By \eqref{eq:intrinsic_acc_and_ex_acc}, it is equivalent to that $\Retr$ is $C^2$  and for any $x\in \M$ both \eqref{eq:first-order-retraction} and the following hold
\begin{equation}\label{eq:second-order-retraction}\tag{R2}
\Proj_{\T_x\M}(c(t)-(x+t\eta))=o(t^2),  \quad \text{as } t\to 0.
\end{equation}
Since the alternating projection method is well defined on an open neighborhood of the manifold,
we consider a retraction as a mapping defined on an open subset of the tangent bundle $\T\M$.
In particular, we will construct $\psi$ on a bundle neighborhood of $(\bar x,0)$.
More precisely, we seek an open set of the form
\begin{equation}\label{def:retr_smooth_neighborhood}
	\mathcal{D}(\bar x)
	:=\{(x,\eta)\in\T\M:\ x\in \Omega_{\bar x},\ \|\eta\|<\rho_{\bar x}\},
\end{equation}
where $\Omega_{\bar x}\subset\M$ is a neighborhood of $\bar x$ that is open in $\M$, and $\rho_{\bar x}>0$ is a constant.

\subsection{Metric regularity}

Let $\Phi:\mathcal Y \rightrightarrows \mathcal Z$ be a set-valued mapping between Euclidean spaces.
We say that $\Phi$ is \emph{metrically regular}  \cite{Rockafellar2009},\cite[Section~5.1]{lewis2008alternating} at $\bar y$ for $\bar z\in \Phi(\bar y)$ if there exist
a constant $\kappa>0$ and neighborhoods of $\bar y$ and $\bar z$ such that
\[
d_{\Phi^{-1}(z)}(y)\ \le\ \kappa\, d_{\Phi(y)}(z)
\quad\text{for all $y$ near $\bar y$ and $z$ near $\bar z$.}
\]

\paragraph{Regular intersection of two sets.}
Given two closed sets $M,N\subset \mathcal E$, define the multifunction
$\Phi:\mathcal E^2 \rightrightarrows \mathcal E$ by
\[
\Phi(x,y)=
\begin{cases}
	\{x-y\}, & x\in M,\ y\in N,\\
	\emptyset, & \text{otherwise}.
\end{cases}
\]
Then $0\in \Phi(x,y)\ \Longleftrightarrow\ x=y\in M\cap N$.
Following \cite[Section~5.1]{lewis2008alternating}, we say that $M$ and $N$ have a \emph{regular intersection}
at $\bar x\in M\cap N$ if $\Phi$ is metrically regular at $(\bar x,\bar x)$ for $0$.
   We get the following lemma, which will be used in our later analysis; see also \cite[Eq.(1.2)]{drusvyatskiy2015transversality}.

\begin{lemma}\label{lem:EB_on_M}
	Let $\M=\M_1\cap \M_2$  intersect cleanly at $\bar x$.
	Then there exist $r_1>0$  and  $\kappa>0$ such that for all $x\in \M_1\cap \mathbb B(\bar x, r_1)$,
	 	\begin{equation}\label{eq:bound_M_2}
	d_{\M}(x)\ \le\  \kappa\ d_{\M_2}(x).
	 	\end{equation}
\end{lemma}
 \begin{proof}
 	Intrinsic transversality implies subtransversality of $(\M_1,\M_2)$ at $\bar x$
 	\cite[Theorem~6.2]{drusvyatskiy2015transversality}. Moreover, clean intersection implies intrinsic transversality. Hence, there exist $r_1>0$ and $\kappa>0$ such that \eqref{eq:bound_M_2} holds.
 \end{proof}

\section{Revisit of alternating projections type algorithms}\label{sec:alternating_projection}
We begin by recalling projective retractions and the local regularity of the orthogonal projection, which will be instrumental in establishing the retraction properties of alternating-projections-type schemes.
\subsection{Projective retraction}
A  well-known result is that the orthogonal projection operator is a second-order retraction on a  $C^3$ embedded submanifold of Euclidean space; see \cite{absil2012projection}.   Given any $z\in\mathcal E$, we define the projection problem  as follows:
\begin{equation}\label{prob:intersection}
	\Proj_{\M}(z)=  \argmin_{ x\in\M} \frac{1}{2}\normfro{x-z }^2. \tag{Proj}
\end{equation}
It follows from \cite[Lemma 3.1]{absil2012projection} that the projection is always well-defined locally for embedded submanifolds of Euclidean space.  We present the result below.

\begin{proposition}\label{prop:well-define_proj}
For any $C^{p,1}(p\ge 2)$ submanifold $\M$ embedded in $\mathcal E$. Given  any $x\in\M$, there exists a   constant $\bar \delta>0$ such that   for any  $z\in \mathbb{B}(x, \bar \delta)$,
	problem \eqref{prob:intersection} has a unique optimal point.
Furthermore,  we get   a   retraction operator, which is $C^{p-1,1}$  on $\M$ as follows:
\begin{equation}\label{prob:retraction}
	\retr_{\bar{x}}^{\Proj}(\eta) := \argmin_{ y\in\M} \frac{1}{2}\normfro{y- (\bar{x}+\eta) }^2, \quad \text{with}\quad \|\eta\|\leq \bar{\delta}.\tag{Proj-R}
\end{equation}
Finally, if $p\geq 3$, $	\retr^{\Proj}$ is also a second-order retraction.
\end{proposition}
\begin{remark}
Proposition~\ref{prop:well-define_proj} is classical: the projection retraction
$\Retr^{\Proj}$ is locally well-defined under $C^{p}$ regularity of $\M$, and this
does not require Lipschitz continuity of higher-order derivatives.
For alternating-type schemes, the existence of the limit map
$\psi_{\bar x}(\eta)=\lim_{k\to\infty}\phi^{k}(\bar x+\eta)$ at a fixed base point
$\bar x\in\M$ also only requires $C^{p}$ regularity.
The stronger $C^{p,1}$ condition is invoked only later to derive uniform
estimates and joint $C^{j}$ regularity of the limit map \eqref{eq:limiting_alter_proj} on $\T\M$ for $j=1,2$.
\end{remark}

Importantly, we observe that the proof of Proposition~\ref{prop:well-define_proj} in \cite[Lemma 3.1]{absil2012projection} implies that any first-order critical point in a neighborhood of $\bar x$ is the unique optimal solution. That is, if $\eta$ in \eqref{prob:retraction} is sufficiently small, the retraction can be computed by any algorithm that converges to a first-order critical point. However, designing an efficient method to solve this problem remains a challenging task due to the complex constraints.  This motivates to further study the alternating-projections-type algorithms.

\subsection{Alternating-projections-type  methods}\label{subsection:APM}

Different from the projection problem \eqref{prob:intersection}, alternating projection--type algorithms seek
a computationally tractable approximation to it.
We categorize the related approaches into three types: the classical Alternating
Projections Method (APM)~\cite{lewis2008alternating,andersson2013alternating,drusvyatskiy2015transversality,noll2016local}, a broader family of
{i}nexact {A}lternating {P}rojections {m}ethods (IAPM), and Newton-type schemes.
The IAPM class includes inexact alternating projections (iAP)~\cite{kruger2016regularity}\cite{drusvyatskiy2019local}.
We also discuss related quadratically convergent alternating-type variants,
such as APHL~\cite{xiao2025quadratically}.
In addition, Newton-type schemes such as NewtonSLRA~\cite{schost2016quadratically,nagasaka2021relaxed}
achieve quadratic local convergence by exploiting higher-order models.
Under appropriate regularity assumptions, e.g., (intrinsic) transversality or
 non-degeneracy (see \cite[eq.(1.3)]{xiao2025quadratically}), these methods enjoy local linear or quadratic convergence. Note
in the smooth manifold setting, this nondegeneracy reduces to transversality since   
		 $\T_x\M_2=\ker(\mathrm{D} H(x))$, so $\mathrm{D} H(x)\T_x\M_1=\R^m$ is equivalent to $\T_x\M_1+\T_x\M_2=\R^n$.
Table~\ref{tab:manifold_intersection} summarizes representative methods related
to our setting. Throughout the paper, we consider two $C^{p,1}$ ($p\ge 2$) manifolds with clean intersection. Note that clean intersection implies {intrinsic transversality} \cite{drusvyatskiy2015transversality}, and intrinsic transversality implies 0-separability \cite[Proposition 2]{noll2016local}. 
On the other hand, the quadratically convergent schemes NewtonSLRA and APHL
typically require stronger regularity conditions to guarantee that the iterates
are well-defined and convergent. In the retraction theory below, we only prove
the second-order retraction property for the standard NewtonSLRA map, while
APHL is treated as a related quadratically convergent method; see Appendix~\ref{app:APHL_first_order}
for its first-order expansion property.

\begin{table}[htbp]
	\centering
	\caption{Summary of methods for finding intersection points of two nonconvex sets.
The smoothness condition is listed when the methods are specified to two
manifolds. For equality-constrained manifolds, \(C^{1,1}\) regularity suffices
for APM and IAPM methods, whereas \(C^{2,1}\) regularity suffices for
NewtonSLRA. APHL is included as a related quadratically convergent method; its
projective mapping \(Q\) generally induces an oblique or variable-metric
correction, rather than the Euclidean projection-type maps covered by the main
framework.}
	\label{tab:manifold_intersection}
	\begin{tabular}{llccc}
		\toprule
		\textbf{Type} & \textbf{Algorithm} & \textbf{Rate} & \textbf{Regularity}  & \textbf{Smoothness}\\
		\midrule
		\multirow{4}{*}{APM}  
		& \multirow{4}{*}{APM}  & \multirow{4}{*}{Linear}   & Transversality\cite{lewis2008alternating}& $C^2$ \\
          & &  & Nontangential \cite{andersson2013alternating} & $C^2$\\
       & &  & Intrinsic transversality \cite{drusvyatskiy2015transversality} & $C^2$\\
      & &   & 0-separability \cite{noll2016local} & $C^2$\\
        \hline
		\multirow{3}{*}{IAPM}
		& iAP~\cite{kruger2016regularity}& Linear & Intrinsic transversality & $C^2$\\
		& iAP~\cite{drusvyatskiy2019local}& Linear & 0-separability& $C^2$\\
		& APHL~\cite{xiao2025quadratically}& Quadratic & Non-degeneracy & $C^{1,1}$\\
		\hline
		Newton
		& NewtonSLRA\cite{schost2016quadratically,nagasaka2021relaxed}
		& Quadratic & Transversality & $C^3$\\
		\bottomrule
	\end{tabular}
\end{table}

\subsubsection{Alternating projections method}
Lewis and Malick  \cite{lewis2008alternating} studied the alternating projections method for finding a point in the intersection of two transverse manifolds. 
Specifically, let
$ 
\Proj_{\M_1}$ \text{and}  $\Proj_{\M_2 }$ 
denote the (orthogonal) projection operators onto \(\M_1\) and \(\M_2 \) respectively. They are both well-defined locally. Given $\bar{x}\in\M_1\cap \M_2,$ there exist $\delta_1>0$ and $\delta_2>0$ such that the projections 
\begin{equation}\label{eq:well-defined-PM-PS}
	\left\{ \begin{array}{cc}
		\Proj_{\M_1}(\bar{x}+\eta)\quad \text{is smooth},&\quad \forall \eta\in\R^{n\times r}, \|\eta\|\leq\delta_{1},\\
		\Proj_{\M_2}(\bar{x}+\xi)\quad \text{is smooth},& \quad \forall \xi\in\R^{n\times r},  \|\xi\|\leq \delta_{2}. 
	\end{array}\right.
\end{equation}  
Define the composite mapping
\begin{equation}\label{eq:alter_projection}\tag{Alt-P}
\varphi(x)= \Proj_{\M_1}\circ \Proj_{\M_2 }(x).
\end{equation}
Combined with Proposition \ref{prop:well-define_proj}, 
we know that near the point $\bar{x}$, the map $\Proj_{\M}$, $\Proj_{\M_1}$, $\Proj_{\M_2}$ and $\eta\mapsto\varphi(\bar{x}+\eta)$ are   well defined on $\mathbb{B}(\bar{x}, \delta_s)$, where 
\begin{equation}\label{def:local_alter_region_well_def}
	\delta_s:= \min\{\bar{\delta},\delta_{1}, \delta_{2}\}.
\end{equation}
Given the initial point $x_0=\bar{x}+\eta$ for  $\bar{x}\in \M$ and $\eta\in\T_{\bar{x}}\M$, the alternating projections method iterates as
\[
x_{k+1} = \varphi(x_k).
\]
It converges to an intersection point at a locally linear rate if the two manifolds are transverse; see \cite[Theorem 4.3]{lewis2008alternating}. The convergence rate is closely related to the angle between two manifolds.
\begin{definition}[Angle between two manifolds \cite{lewis2008alternating}]\label{def:angle}
	Let $x\in\M,$ the angle between  $\M_1$ and $\M_2$ at $x$ is the angle between  $\T_x\M_1$ and $\T_x\M_2.$ That is, it is the angle  between $0$ and $\pi/2$ whose cosine is 
	\[
	\begin{aligned}
		c(\M_1, \M_2, x) : = \max\{ \langle \xi, \zeta\rangle\mid \xi&\in  \T_x\M_1\cap (\T_x\M)^\perp, \normfro{\xi} \leq 1, \\ \zeta&\in  \T_x\M_2\cap (\T_x\M)^\perp, \normfro{\zeta}\le 1 \}.   
	\end{aligned}
	\]
\end{definition}

	Regarding the linear subspaces, one has the following result.
	\begin{lemma}\label{lem:contraction_projection}\cite[Lemma 9.5]{deutsch2001best} 
		Given two linear subspaces $M$ and $N$ in $\mathcal E$, one has
		\[
		c(M,N)= \|\Proj_{N}\Proj_{M} - \Proj_{M\cap N}\|_2.
		\]
	\end{lemma}
		Since  $\M_1,\M_2$    intersect cleanly at $\bar x$, we have $\T_{\bar{x}}\M_1\cap \T_{\bar{x}}\M_2 = \T_{\bar x}\M$.  Then, it holds  \cite[Proposition 4.6]{andersson2013alternating} that
	\begin{equation}\label{def:angle_M1_M2_old}
		c(\M_1,\M_2,\bar x)=\|\Proj_{\T_{\bar{x}}\M_1}\Proj_{\T_{
				\bar{x}}{\M_2 }}-\Proj_{\T_{\bar{x}}\M}\|_2<1.
	\end{equation}
	A key idea in analyzing the convergence rate of the alternating projections method is that one linearized step contracts the distance to the target set by a constant factor  {$c \in [c(\M_1, \M_2, \bar{x}), 1)$ for current iterate $x$ sufficiently close to $\bar{x}$}. This contraction property, stated in equation \eqref{def:angle_M1_M2_old},  underlies the local linear convergence of the method.

Moreover,  the mapping
$x\mapsto c(\M_1,\M_2,x)$ is continuous on $\M$ (in fact $C^{p-1}$; see \cite[Prop.4.7]{andersson2013alternating}).  It follows that for any constant $c$ satisfying $ c(\M, \M_2, \bar{x})<c<1,$ there exists $\delta_c\in(0,\delta_s],$ where $\delta_s$ is given in \eqref{def:local_alter_region_well_def}, such that  {for any $x \in \mathbb{B}(\bar{x}, \delta_c) \cap \M$,}  
	\begin{equation}\label{eq:upper_bound_cosine}
c(\M_1,\M_2,x)
=\bigl\|\Proj_{\T_{x}\M_1}\Proj_{\T_{x}\M_2}-\Proj_{\T_{x}\M}\bigr\|_2\leq c.
	\end{equation}
	

It follows from \cite{lewis2008alternating,andersson2013alternating} that the APM converges $R-$linearly to a limiting point $x^*\in\M$. Specifically, there exists a radius $\delta'>0$ such that  the limiting map  $\psi_{\bar{x}}:\T_{\bar{x}}\M\cap \mathbb{B}(0, \delta')\rightarrow\M$ is well-defined. Later, we will show that $\psi_{\bar{x}}$ is a (second-order) retraction, provided the $C^{p,1}$ regularity of the manifolds.

\subsubsection{Inexact alternating projections and NewtonSLRA}\label{section:revist-inexact-alt}
In this section, we revisit the inexact alternating projection schemes
underlying IAPM and NewtonSLRA. In contrast to the exact alternating projection map
\eqref{eq:limiting_alter_proj}, the orthogonal projection onto $\M_i$ is replaced by a
computationally cheaper surrogate. 

Specifically, the mapping $\phi_1:\mathcal{E}\to\M_1$ and $\phi_2:\E\to \E$  
approximates the orthogonal projector  {$\Proj_{\M_1}$ and $\Proj_{\M_2}$ respectively}. Note that   $\phi_1$ needs to be   $\M_1$-valued. The resulting inexact alternating map is
\begin{equation}\label{eq:inexact_alter_proj}\tag{Alt-IP}
	\varphi_{\mathrm{IAPM}} \;:=\; \phi_1\circ \phi_2.
\end{equation}
To ensure that the limiting mapping
$\psi_{\bar x}(\eta)=\lim_{k\to\infty}\varphi_{\mathrm{IAPM}}^{\,k}(\bar x+\eta)$ converges, there are three  existing types of algorithms:
\begin{enumerate}
\item[(a)] $\phi_1$ is an exact projection and $\phi_2$ is an inexact one: $\phi_1=\Proj_{\M_1}$ and $\phi_2$ satisfies
\begin{equation}\label{eq:second_order-inexact_M2}
	\big\|\phi_2(x)-\Proj_{\M_2}(x)\big\|
	\;=\;o \big(d_{\M_2}(x)\big),
	\qquad x\to \bar x.
\end{equation}
\item[(b)] $\phi_2$ is an exact projection and $\phi_1$ is an inexact one: $\phi_2=\Proj_{\M_2}$ and $\phi_1$ is so-called \emph{faithful approximation} to $\Proj_{\M_1}$ \cite{drusvyatskiy2019local}:  if $z\in\M_1$ is sufficiently close to $\bar x$, $y=\Proj_{\M_2}(z),$ for any $\epsilon>0$,  one has 
\begin{equation}\label{eq:faithful}
    \| \Proj_{\M_1} (y) - \phi_1(y) \|\leq \epsilon \|y-z\|.
\end{equation}
 
\item[(c)] Both $\phi_1$ and $\phi_2$ are inexact, and $\phi_i:\E\to \M_i$ satisfy so-called   $(\tau,\sigma)$-projection \cite{kruger2016regularity} (see Appendix~\ref{sec:append-iap}),
	denoted by $\Proj_{\M_i}^{\tau,\sigma}$, as an inexact substitute for $\Proj_{\M_i}$.
\end{enumerate}

Despite the use of inexact projections, IAPM and NewtonSLRA can still attain linear or even
quadratic local convergence under appropriate regularity conditions in Table~\ref{tab:manifold_intersection}. Below, we briefly review
representative schemes, emphasizing (i) the update rule, (ii) the main computational cost,
and (iii) the local convergence rate. Further conditions for the  retraction property will be stated later in
Assumptions~\ref{assump:alt_framework} and \ref{assump:DP}. 

\begin{enumerate}
	\item \textbf{Inexact alternating projections (iAP)}  Kruger and Thao \cite{kruger2016regularity} propose the $(\tau,\sigma)$-projection notion. Under intrinsic transversality assumptions for general
	nonconvex sets, $R$-linear convergence of the corresponding inexact alternating projection scheme
	was established.
	
	In a related direction, Drusvyatskiy and Lewis \cite{drusvyatskiy2019local}
studied two types of approximate projection steps under the $0$-separability
and super-regularity assumptions, namely the two schemes (a) and (b) described
above.
	
	\emph{Computation.}
	A typical choice of inexact map $\phi_i$ is based on a local model or linearization. For instance, 
	if $\M_2=\{z:H(z)=0\}$ with a $C^{1,1}$ mapping $H$, one may define $\phi_2(x)$ as the
	projection onto the linearized constraint:
	\begin{equation}\label{eq:projection_to_linearization}
		\phi_2(x)= \argmin_{z}\ \|z-x\|^2
		\quad \text{s.t.}\quad
		H(x)+\mathrm{D} H(x)(z-x)=0.
	\end{equation}
	When $\mathrm{D} H(x)\mathrm{D} H(x)^{\top}$ is nonsingular, this yields the closed form
	\[
	\phi_2(x)
	= x-\mathrm{D} H(x)^{\top}\big(\mathrm{D} H(x)\mathrm{D} H(x)^{\top}\big)^{-1}H(x),
	\]
	and it satisfies $\|\phi_2(y)-\Proj_{\M_2}(y)\|
=
O\big(d_{\M_2}(y)^2\big),
 y\to \M_2$; see
	\cite[Theorem~2]{drusvyatskiy2019local}.
	The faithful approximation can also be obtained as shown in \cite[Section 7]{drusvyatskiy2019local}.

	\emph{Convergence.}
	Under the corresponding regularity assumptions, the iterates converge $R$-linearly to the
	intersection; see \cite[Theorem~31]{kruger2016regularity} and \cite[Theorems~1\&3]{drusvyatskiy2019local}.
		
		\item \textbf{NewtonSLRA \cite{schost2016quadratically,nagasaka2021relaxed}.}  
		This method was originally developed to approximate  the projection onto the intersection of   an affine subspace $\M_1$ with a
		fixed rank manifold $\M_2=\mathcal{D}_r=\{X\in\R^{q_1\times q_2}: \rank(X)=r\}$ in \cite{schost2016quadratically}.
		More generally, it applies to the intersection of an affine subspace $\M_1$ with any
		$C^3$ submanifold $\M_2$, provided that the transversality condition holds at
		some intersection point $\bar x$.
		
		NewtonSLRA keeps an accurate projection onto $\M_2$ in order to exploit
		second-order geometry, and then performs a Newton-like correction by intersecting
		$\M_1$ with the affine tangent space of $\M_2$ at the projected point.
		Given a current iterate $x\in \M_1$ close to $\bar x$, compute
		\[
		\tilde x = \Proj_{\M_2}(x).
		\]
		Under LICQ, the tangent and normal spaces of $\M_2$ at $\tilde x$ admit the
		representations
		\[
		\T_{\tilde x}\M_2=\ker(\mathrm{D} H(\tilde x)),\qquad 
		\rmN_{\tilde x}\M_2=\mathrm{range}\big(\mathrm{D} H(\tilde x)^{\top}\big).
		\]
		The next iterate is then defined as the orthogonal projection of $x$ onto the
		linearized intersection $\M_1\cap(\tilde x+\T_{\tilde x}\M_2)$:
		\begin{equation}\label{eq:projection_Onto_linear_intersection}
		x^{+}
		= \argmin_{y}\ \|x-y\|^2,
		\quad \st\quad y\in \M_1,\ \ \mathrm{D} H(\tilde x)\,[y-\tilde x]=0,
		\end{equation}
		i.e., \begin{equation*}x^{+}=\Proj_{\M_1\cap(\tilde x+\T_{\tilde x}\M_2)}(x).\end{equation*}
        Since $x-\tilde x\in \N_{\tilde x}\M_2,$ one has $\Proj_{\M_1\cap(\tilde x+\T_{\tilde x}\M_2)}(x) = \Proj_{\M_1\cap(\tilde x+\T_{\tilde x}\M_2)}(\tilde x).$ We rewrite the iteration as 
        \begin{equation}\label{eq:NewtonSLRA}
         \varphi_{\mathrm{NewtonSLRA}}(x) : =   \Proj_{\M_1\cap(\tilde x+\T_{\tilde x}\M_2)}(\Proj_{\M_2}(x)).
        \end{equation}
		The transversality condition ensures that the intersection   $\M_1\cap(\tilde x+\T_{\tilde x}\M_2)$ is non-empty near $\bar x$ \cite[Lemma 2.4]{schost2016quadratically}, while the clean intersection condition does not guarantee this. Therefore, we need to assume transversality for NewtonSLRA. In \cite{nagasaka2021relaxed}, a relaxed version was proposed. The main difference is that the update in \eqref{eq:NewtonSLRA} is replaced by
        \begin{equation}\label{eq:relax_SLRA}
        x^{+}=\Proj_{\M_1\cap(\tilde x+\T_{x,\tilde x}\M_2)}(\Proj_{\M_2}(x)),  
        \end{equation}
        where $\T_{x,\tilde x}\M_2= \big(\mathrm{span}\{\tilde x - x\}\big)^\perp$. Since $\mathrm{span}\{\tilde x - x\}\subset \N_{\tilde x}\M_2$, $\T_{x,\tilde x}\M_2$ is  called the relaxed tangent space. This update is cheaper than \eqref{eq:NewtonSLRA}  due to the smaller number of constraints. 
		
Although the iteration of the standard NewtonSLRA map \eqref{eq:NewtonSLRA}
has a similar form as $\varphi_1\circ\varphi_2$, $\varphi_1$ does not satisfy
the condition \eqref{eq:phi_inexact_proj_first_order} assumed later.
Nevertheless, its limiting map can still be analyzed from the same retraction
perspective. We therefore treat the standard NewtonSLRA map separately in
Subsection~\ref{subsect:NewtonSLRA}.

The relaxed update \eqref{eq:relax_SLRA} is also known to enjoy strong
local convergence properties in the SLRA setting. However, the relaxed tangent
space \(T_{x,\tilde x}\M_2=\operatorname{span}\{\tilde x-x\}^{\perp}\) need not
depend smoothly on \(x\) near the intersection, since \(\tilde x-x\to0\).
Thus the \(C^2\)-regularity required for the retraction result is not automatic
for the relaxed variant. It will be used only as a computational variant in the
numerical experiments.

		\emph{Computation:} (i) one projection onto $\M_2$;
		(ii) one projection given by \eqref{eq:NewtonSLRA} or \eqref{eq:relax_SLRA}. 
		
		\emph{Convergence:} under standard transversality assumptions, (relaxed) NewtonSLRA converges to an intersection point
		 at a local quadratic rate; see \cite[Theorem~4.1]{schost2016quadratically} and \cite[Theorem 2]{nagasaka2021relaxed}.
		
	\item 	\textbf{APHL~\cite{xiao2025quadratically}.}
    APHL is a quadratically convergent alternating-type method for solving
intersections between a general   set \( \M_1 \) and an equality-constrained manifold
\[
\M_2=\{x\in\mathbb R^n:H(x)=0\},
\qquad
H:\mathbb R^n\to\mathbb R^m.
\]
Its update is based on a projective mapping \(Q\) and a constraint-dissolving
correction \(A(x)=x-Q(x)d(x)\), followed by a projection back onto \(\M_1\).
In general, this correction corresponds to a variable-metric, or oblique,
projection onto the tangent space of \(\M_2\), rather than the Euclidean
orthogonal projection used in the present framework. When \(Q=I\) and the regularization is inactive, the correction
reduces to the linearized projection map \eqref{eq:projection_to_linearization}. Thus the general APHL map
is not covered by the Euclidean retraction theorem proved below. Nevertheless,
viewed as an overall solver, its limiting point satisfies the first-order
expansion \eqref{eq:first-order-retraction}; see
Appendix~\ref{app:APHL_first_order}. Establishing a full \(C^1\)-retraction theory for the general APHL map, or a
more general variable-metric extension of the present framework, is left for
future work.

		\emph{Computation:} (i) evaluation of $Q(x)$ and $\mathrm{D} H(x)$; (ii) solution of an
		$m\times m$ linear system to obtain $d(x)$; (iii) and one projection onto $\M_1$.
		
		\emph{Convergence:} under a suitable nondegeneracy regularity condition, see \cite[eq.(1.3)]{xiao2025quadratically}, the iterates converge locally
		 {quadratically} to $\M_1\cap\M_2$.  
 
	\end{enumerate}

\section{A Unified Framework for Limiting Retractions}\label{sec:opt_on_intersection}

In this section, we present the main contributions of the paper.
The organization is as follows.
In Subsection~\ref{subsect:revisit}, we use a unified framework to summarize the most convergence rates of the algorithms listed in Table~\ref{tab:manifold_intersection} and obtain the first-order retraction property \eqref{eq:first-order-retraction}.
Subsection~\ref{subsection:second-order} introduces the regularity assumptions and establishes the second-order approximation property~\eqref{eq:second-order-retraction}.
In Subsections~\ref{subsect:C^1} and \ref{subsec:C2_limit}, we prove the $C^{p-1}$ regularity of the limiting map $\psi$, thereby showing that $\psi$ defines a (second-order) retraction.
Subsections~\ref{subsection:very_assump_for_proj} and \ref{subsection:correctionness} verify that the required assumptions are naturally satisfied in our setting.
Finally, Subsection~\ref{subsect:NewtonSLRA} treats NewtonSLRA separately and shows that it can be analyzed within the same limiting-map framework.

To be clear, we   write
\begin{equation}\label{def:xstar_series}
\psi(x,\eta)=x^*=x_0+\sum_{k=0}^{\infty}\Delta_k(x,\eta),
	\qquad \Delta_k:=x_{k+1}-x_k.
\end{equation}
To establish the $C^{p-1}$ regularity of $\psi$ in some neighborhood $\mathcal{D}(\bar x)$ of $(\bar x,0)$ in the tangent bundle $\T\M$,
we will prove that for each $j=0,1,\ldots,p-1$,
\[
\sum_{k=0}^\infty \bigl\| \mathrm{D}^j \Delta_k(x,\eta)\bigr\|
\]
converges  absolutely and uniformly for all $(x,\eta)\in\mathcal{D}(\bar x)$. The neighborhood $\mathcal{D}(\bar x)$ will be defined later.

\subsection{Linear convergence: a revisit}\label{subsect:revisit}
Starting from $\bar x\in\oM$, if the alternating-projections-type methods in Table~\ref{tab:manifold_intersection} converge at least linearly, then we can show that the
corresponding limiting mappings induce
  retractions on $\M$. We emphasize that the limit point
produced by alternating projections generally does not coincide with the exact
projection point defined in \eqref{prob:retraction}. For the retraction  property \eqref{eq:first-order-retraction}, we show it by bounding the distance $\|\varphi(\bar x,\eta)-\Proj_{\overline\M}(\bar x+\eta)\|$, which was also obtained in existing literatures \cite{andersson2013alternating,schost2016quadratically}. However, for the second-order property \eqref{eq:second-order-retraction}, it is more subtle to control the tangential error.  More importantly, it would be more tricky to show that $\psi(\cdot, \cdot)$ is $C^{p-1}$ on the open set $\mathcal D(\bar x)\subset\T\oM$ in \eqref{def:retr_smooth_neighborhood}.

We first assume that the iterative map $\varphi$ in
\eqref{eq:limiting_alter_proj} admits the decomposition described below, which
covers APM and the IAPM variants in Table~\ref{tab:manifold_intersection}. 


\begin{definition}
For $i\in\{1,2\}$, we say that $\phi_i$ is first-order consistent with
$\Proj_{\M_i}$ if
\begin{equation}\label{eq:phi_inexact_proj_first_order}
    \|\phi_i(y)-\Proj_{\M_i}(y)\|
    =
    o\big(d_{\M_i}(y)\big),
    \quad
    \text{as } y\to x,\ \ x\in \mathcal U_{\bar x}\cap \M_i .
\end{equation}
\end{definition}

\begin{assumption}\label{assump:alt_framework}
Under Assumption~\ref{assumpt:M1andM2_intersect_cleanly}, fix $\bar x\in \M$.
There exists an open neighborhood $\mathcal U_{\bar x}\subset \mathcal E$ such that
$\phi_1:\mathcal U_{\bar x}\to \M_1$ and
$\phi_2:\mathcal U_{\bar x}\to \mathcal E$ are $C^{p-1}$ mappings $(p\ge 2)$.
Set $\varphi:=\phi_1\circ\phi_2$ on $\mathcal U_{\bar x}$.

Assume that one of the following conditions holds:
\begin{enumerate}
  \item[(i)] $\phi_1=\Proj_{\M_1}$ on $\mathcal U_{\bar x}$, and $\phi_2$ is
  first-order consistent with $\Proj_{\M_2}$.

  \item[(ii)] $\phi_2=\Proj_{\M_2}$ on $\mathcal U_{\bar x}$, and $\phi_1$ is
  first-order consistent with $\Proj_{\M_1}$.

  \item[(iii)] For each $i\in\{1,2\}$, $\phi_i$ is first-order consistent with
  $\Proj_{\M_i}$ and satisfies
  \[
      \phi_i(\mathcal U_{\bar x})\subset \M_i .
  \]
\end{enumerate}
\end{assumption}
The first-order consistency condition
\eqref{eq:phi_inexact_proj_first_order} requires that the update map $\phi_i$
approximates the orthogonal projection onto $\M_i$ up to a little-$o$ error
relative to the distance to $\M_i$. This condition is sufficient to obtain the
differential property needed for the first-order retraction result.

As discussed in Section~\ref{subsection:APM}, the first-order consistency
condition is satisfied by the inexact projection constructions considered
there. The condition
\eqref{eq:phi_inexact_proj_first_order} is stronger than the faithful
approximation condition~\eqref{eq:faithful}, because
$d_{\M_i}(y)\le \|y-z\|$ for all $z\in \M_i$. Moreover,
Assumption~\ref{assump:alt_framework}(iii) is stronger than the
$(\tau,\sigma)$-projection notion in \cite{kruger2016regularity}. We adopt these stronger assumptions because they yield the differential
property needed in the subsequent first-order analysis. The second-order
residual estimates will be imposed separately in Assumption~\ref{assump:DP}.

\begin{lemma}\label{lem:Dphi_is_tangent_proj}
Suppose Assumption~\ref{assump:alt_framework} holds. Then, for each
$i\in\{1,2\}$ and any $x\in \M\cap\mathcal U_{\bar x}$, one has
\[
    \rmD \phi_i(x)=\Proj_{\T_x\M_i}.
\]
\end{lemma}

\begin{proof}
Fix $i\in\{1,2\}$ and $x\in \M\cap\mathcal U_{\bar x}$. If
$\phi_i=\Proj_{\M_i}$ on a neighborhood of $x$, then the conclusion follows
from Proposition~\ref{prop:well-define_proj}. We therefore consider the
inexact case.

Let $P_i:=\Proj_{\T_x\M_i}$. For any $u\in\E$, set $y(t):=x+tu$. Since
$x\in\M\subset\M_i$, the first-order consistency of $\phi_i$ with
$\Proj_{\M_i}$ gives
\begin{equation}\label{eq:phi_vs_proj}
    \|\phi_i(y(t))-\Proj_{\M_i}(y(t))\|
    =
    o\big(d_{\M_i}(y(t))\big),
    \qquad t\to0 .
\end{equation}
Moreover, since $x\in\M_i$,
\[
    d_{\M_i}(y(t))\le \|y(t)-x\|=|t|\|u\|.
\]
Combining this estimate with \eqref{eq:phi_vs_proj}, we obtain
\[
    \|\phi_i(y(t))-\Proj_{\M_i}(y(t))\|=o(|t|).
\]
By Proposition~\ref{prop:well-define_proj},
\[
    \Proj_{\M_i}(x+tu)=x+tP_i u+o(t).
\]
Hence
\[
    \phi_i(x+tu)=x+tP_i u+o(t),\qquad t\to0.
\]
Since $u\in\E$ is arbitrary, this proves
\[
    \rmD\phi_i(x)[u]=P_i u,\qquad \forall u\in\E,
\]
and therefore $\rmD\phi_i(x)=\Proj_{\T_x\M_i}$.
\end{proof}

\begin{remark}\label{rem:local_sets}
	Fix $\bar x\in\M$.
	By Assumptions~\ref{assumpt:M1andM2_intersect_cleanly}--\ref{assump:alt_framework},
	there exist open neighborhoods $\mathbf U_{\bar x}\subset \M$ and $\mathcal U_{\bar x}\subset \E$ of $\bar x$
	such that:
	\begin{itemize}[leftmargin=*,itemsep=2pt,topsep=2pt]
		\item the clean-intersection property in Assumption~\ref{assumpt:M1andM2_intersect_cleanly}
		(and hence the angle bound \eqref{eq:upper_bound_cosine}) holds for all $x\in \mathbf U_{\bar x}$;
		\item the mappings $\Proj_{\M_1},\Proj_{\M_2},\Proj_{\M}$ and $\phi_i$ are well-defined and $C^{p-1,1}$ on $\mathcal U_{\bar x}$,
		and the metric regularity    in Lemma~\ref{lem:EB_on_M} holds on $\mathcal U_{\bar x}$.
	\end{itemize}
	
	Choose an open neighborhood $\Omega_{\bar x}\subset \M$ of $\bar x$ such that
	\[
	\Omega_{\bar x}\subset \mathbf U_{\bar x}\cap \M
	\qquad\text{and}\qquad
	\Omega_{\bar x}\subset  \M\cap \mathcal U_{\bar x}.
	\]
	 Define the closure of the neighborhood $\Omega_{\bar x}$ in $\oM$ as
	\begin{equation}\label{eq:closure}
		K:=\mathrm{cl}_{\oM}{\Omega}_{\bar x},
	\end{equation} 
	and  open $\delta$-neighborhood  of $K$ in $\mathcal E$ with radius $\delta>0$ by
	\[
	\mathcal N_\delta(K):=\{\,y\in\mathcal E:\ d_K(y)<\delta\,\}.
	\]
	Then $K$ is compact and satisfies $K\subset \mathbf U_{\bar x}\cap \mathcal U_{\bar x}$.
	
	Throughout the sequel, all statements are understood to hold after possibly shrinking
	$\Omega_{\bar x}$ and $\mathcal U_{\bar x}$ while keeping the above inclusions.
\end{remark}

We have the following results from existing literatures, which are the key to establishing the linear convergence rate.
\begin{fact}\label{fact:alt_linear}
	For any $ x\in \M$.
	Suppose Assumptions~\ref{assumpt:M1andM2_intersect_cleanly} and \ref{assump:alt_framework} hold at $x$.
	Then there exist constants $\delta_{ x}>0$ and $\tau_{ x}\in(0,1)$ such that
	$\mathbb B( x,\delta_{ x})\subset \mathcal U_{ x}$ and, for any
	\[
	y\in \M_1\cap \mathbb B( x,\delta_{ x}),
	\qquad y^+:=\varphi(y),
	\]
	the point $y^+$ is well-defined and satisfies
	\begin{equation}\label{eq:linear_decay_dM2}
		d_{\M_2}(y^+)\le \tau_{ x}\, d_{\M_2}(y).
	\end{equation}
	Consequently, if $x_{k+1}=\varphi(x_k)$ and $x_k\in \M_1\cap \mathbb B( x,\delta_{ x})$ for all $k$,
	then $d_{\M_2}(x_k)\le \tau_{ x}^{\,k}\, d_{\M_2}(x_0)$.
\end{fact}

In particular, \eqref{eq:linear_decay_dM2} 
is ensured by   \cite[Theorems 1\&3]{drusvyatskiy2019local} and  \cite[Theorem~31]{kruger2016regularity}\footnote{To obtain this property, the parameters $\tau,\sigma$ should be modified as explained in Appendix~\ref{sec:append-iap}.}  for inexact alternating
projections methods. Hence, it also holds for APM. 

 Since the retraction should be defined on the open set $\mathcal D(\bar x)$ in \eqref{def:retr_smooth_neighborhood} of the tangent bundle $\T\M$. That is, $\psi$ should be well-defined on it. We get the uniform version of Fact~\ref{fact:alt_linear} on a compact set in Lemma~\ref{lem:uniform_fact2}. 

 \begin{lemma}\label{lem:uniform_fact2}
 	Fix $\bar x\in \M$ and take neighborhoods $\Omega_{\bar x}\subset \M$
 	and $ \mathcal U_{\bar x}\subset \E$ as in Remark~\ref{rem:local_sets}.
 	Set  $K$ as defined in \eqref{eq:closure} and suppose Assumptions~\ref{assumpt:M1andM2_intersect_cleanly} and \ref{assump:alt_framework} hold at any $x\in K$.
 	Then there exist constants $\delta_K>0$ and $\tau_K\in(0,1)$ such that
 	$\mathcal N_{\delta_K}(K)\subset \mathcal  U_{\bar x}$ and for any
 	$x\in \M_1\cap \mathcal N_{\delta_K}(K)$, the iterate $x^+=\varphi(x)$
 	is well-defined and satisfies
 	\[
 	d_{\M_2}(x^+)\le \tau_K\, d_{\M_2}(x).
 	\]
 \end{lemma}
 
 \begin{proof}
 	Since $K\subset \oM\cap \mathcal U_{\bar x}$ and $\mathcal U_{\bar x}$ is open in $\mathcal E$,
 	there exists $\delta_U>0$ such that
 	\[
 	\mathcal N_{\delta_U}(K)\subset \mathcal U_{\bar x}.
 	\]
 	
 	For each $u\in K$, using Fact~\ref{fact:alt_linear}, there exist
 	constants $\delta_u>0$ and $\tau_u\in[0,1)$ such that for all
 	$x\in \M_1\cap \mathbb B(u,\delta_u)$, the iterate $x^+=\varphi(x)$ is well-defined and satisfies
 	\[
 	d_{\M_2}(x^+)\le \tau_u\, d_{\M_2}(x).
 	\]
 	The collection $\{\mathbb B(u,\delta_u/2)\}_{u\in K}$ forms an open cover of $K$.
 	By compactness of $K$, there exists finite subcover
 	$K\subset \bigcup_{j=1}^J \mathbb B(u_j,\delta_{u_j}/2)$.
 	Set
 	\[
 	\delta_0:=\min_{1\le j\le J}\frac{\delta_{u_j}}{2},
 	\qquad
 	\tau_K:=\max_{1\le j\le J}\tau_{u_j}\in(0,1),
 	\qquad
 	\delta_K:=\min\{\delta_U,\delta_0\}.
 	\]
 	Then $\mathcal N_{\delta_K}(K)\subset \mathcal U_{\bar x}$.
 	
 	Now take any $x\in \M_1\cap \mathcal N_{\delta_K}(K)$.
 	Choose $u\in K$ with $\|x-u\|<\delta_K\le \delta_0$.
 	Pick an index $j$ such that $u\in \mathbb B(u_j,\delta_{u_j}/2)$.
 	Then
 	\[
 	\|x-u_j\|\le \|x-u\|+\|u-u_j\|<\delta_0+\frac{\delta_{u_j}}{2}
 	\le \frac{\delta_{u_j}}{2}+\frac{\delta_{u_j}}{2}
 	=\delta_{u_j},
 	\]
 	so $x\in \M_1\cap \mathbb B(u_j,\delta_{u_j})$.
 	Applying Fact~\ref{fact:alt_linear} at $u_j$ yields that $x^+=\varphi(x)$ is well-defined and
 	\[
 	d_{\M_2}(x^+)\le \tau_{u_j}\, d_{\M_2}(x)\le \tau_K\, d_{\M_2}(x),
 	\]
 	which completes the proof.   
 \end{proof}

Note that existing results usually need $x_0\in\M_1$ as a prior, such as \cite{drusvyatskiy2019local} and \cite{kruger2016regularity}. However, in our framework $x_0=\bar x+\eta\notin\M_1$. We first get the following bound.

 \begin{lemma}\label{lem:one_step_into_N}
 	Under the same conditions  of
 	Lemma~\ref{lem:uniform_fact2},   there exist constants $\delta_V>0$ and $\rho_{\rm in}>0$ such that
 	for any $x\in K$ and any $\eta\in \T_x\M$ with $\|\eta\|\le \rho_{\rm in}$,
 	setting $x_0:=x+\eta$, the point $x_0$ belongs to $\mathcal N_{\delta_V}(K),$
 	the iterate $x_1:=\varphi(x_0)$ is well-defined, and
 	\[
 	\|x_1-x_0\|=o(\|\eta\|) \quad \text{as } \eta \to 0,
 	\]
 	\[
\phi_2(x_0)\in   \mathcal N_{\delta_K}(K), \quad \text{and}\quad 	x_1\in \M_1\cap \mathcal N_{\delta_K}(K).
 	\]
 \end{lemma}
 \begin{proof}
 	Fix $\bar x\in \M$ and take neighborhoods $\Omega_{\bar x}\subset \M$ and
 	$\mathcal U_{\bar x}\subset \E$ as in Remark~\ref{rem:local_sets}.

 	\textbf{Step 1}.  
 	Since $K$ is compact and $\mathcal U_{\bar x}$ is open with $K\subset \mathcal U_{\bar x}$,
 	there exists $\delta_V>0$ such that
 	\[
 	\mathcal V_K:=\mathcal N_{\delta_V}(K)\subset \mathcal U_{\bar x}.
 	\]
 	Moreover,  by Assumption~\ref{assump:alt_framework}, for each $u\in K\subset \M_2$ we have $\phi_2(u)=u$,
 	so by continuity of $\phi_2$ there exists $\delta_u>0$ such that
 	$\phi_2\big(\mathbb B(u,\delta_u)\big)\subset \mathcal U_{\bar x}$.
 	Using compactness of $K$ and a finite subcover argument, we may shrink $\delta_V$ if necessary so that
 	$	\phi_2(\mathcal V_K)\subset \mathcal U_{\bar x}.$

 	 \textbf{Step 2}. 
 	Fix any $x\in K$ and $\eta\in \T_x\M$ with $\eta\to 0$, and let $x_0=x+\eta$.
 	Since $\eta\in \T_x\M\subset \T_x\M_2$ and $\M_2$ is at least $C^2$,  it follows from Proposition~\ref{prop:well-define_proj} that
 	\begin{equation}\label{eq:tangent_distance_small_o_K}
 		d_{\M_2}(x_0) = \|x_0 - \Proj_{\M_2}(x_0)\| =o(\|\eta\|),
 		\qquad \eta\to 0,
 	\end{equation}
 	uniformly with respect to $x\in K$.
 
 	Using \eqref{eq:phi_inexact_proj_first_order} together with \eqref{eq:tangent_distance_small_o_K}, we obtain
 	\[
 	\|\phi_2(x_0)-\Proj_{\M_2}(x_0)\|
 	=o\big(d_{\M_2}(x_0)\big)
 	=o(\|\eta\|),
 	\]
 	uniformly for $x\in K$. Therefore,
 	\begin{align*}
 		\|\phi_2(x_0)-x_0\|
 		&\le \|\phi_2(x_0)-\Proj_{\M_2}(x_0)\|+\|\Proj_{\M_2}(x_0)-x_0\| \\
 		&= o(\|\eta\|)+d_{\M_2}(x_0)
 		= o(\|\eta\|),
 	\end{align*}
 	uniformly with respect to $x\in K$.
 	
 	\textbf{Step 3}. 
 	It follows from Lemma~\ref{lem:Dphi_is_tangent_proj}: for $v\to 0$,
 	\begin{equation}\label{eq:projM1_first_order_uniform}
 		\phi_1(x+v)=x+\Proj_{\T_x\M_1}(v)+o(\|v\|),
 	\end{equation}
 	uniformly with respect to $x\in K$.
 	Apply \eqref{eq:projM1_first_order_uniform} with $v=\phi_2(x_0)-x$.
 	Then
 	\[
 	v=\eta+(\phi_2(x_0)-x_0),
 	\qquad
 	\|v\|=\|\eta\|+o(\|\eta\|),
 	\]
 	and using $\eta\in \T_x\M\subset \T_x\M_1$ together with $\|\phi_2(x_0)-x_0\|=o(\|\eta\|)$, we get
 	\[
 	\Proj_{\T_x\M_1}(v)=\eta+\Proj_{\T_x\M_1}(\phi_2(x_0)-x_0)=\eta+o(\|\eta\|).
 	\]
 	Consequently,  one has
 	\[
 	x_1=\varphi(x_0)= \phi_1(\phi_2(x_0))=x+\eta+o(\|\eta\|),
 	\]
 	uniformly for $x\in K$, which yields
 	$\|x_1-(x+\eta)\|=o(\|\eta\|)$.
  
\textbf{Final step}.
 	Since $d_K(x_0)\le \|x_0-x\|=\|\eta\|$, choosing $\rho_{\rm in}\le \delta_V$ ensures $x_0\in \mathcal V_K$.
 	Furthermore, by the uniform $o(\|\eta\|)$ bound above, we may shrink $\rho_{\rm in}>0$ so that
 $\|v\|=\|\phi_2(x_0)-x\|\leq \delta_K $ and	$\|x_1-(x+\eta)\|\le \frac12\|\eta\|$ whenever $\|\eta\|\le \rho_{\rm in}$.
 	Then
 	\begin{equation}\label{ineq:bound_x1_x}
 	d_K(x_1)\le \|x_1-x\|
 	\le \|x_1-(x+\eta)\|+\|\eta\|
 	\le \frac32\|\eta\|.
 	\end{equation}
 Furthermore, assume $\frac32\rho_{\rm in}\le \delta_K$, then $x_1\in \mathcal N_{\delta_K}(K)$.
 \end{proof}

We can ensure that all iterates stay in the neighborhood $\mathcal N_{\delta_K}(K)$ for $x_0=x+\eta$ with $x\in K$ and  $\eta\in \T_x\M$ is small enough, as stated in the following lemma.

\begin{lemma}\label{lem:uniform_invariance_summable}
	Fix $\bar x\in\M$ and let $\Omega_{\bar x}\subset\M$ be as in Remark~\ref{rem:local_sets}.
Under the same conditions  of
 	Lemma~\ref{lem:uniform_fact2}, there exist constants $C_K>1$ and $\rho_K>0$ with
	\[
	0<\rho_K\le \min\Big\{\frac{\delta_K}{3C_K},\,\delta_K\Big\},
	\]
	such that for any     $x\in\Omega_{\bar x}$ and any $\eta\in \T_{x}\M$ with $\|\eta\|\le \rho_K$,
	defining $x_0:=x+\eta$ and $x_{k+1}:=\varphi(x_k)$, the sequence $\{x_k\}_{k\ge 0}$ is well-defined and satisfies, $x_0\in \mathcal N_{\delta_K}(K),
\varphi_2(x_0)\in \mathcal N_{\delta_K}(K),$ and 
	for all $k\ge 1$,
	\begin{equation}\label{eq:invariance_uniform}
		x_k\in \M_1\cap \mathcal N_{\delta_K}(K),
		\qquad
		\phi_2(x_k)\in \mathcal N_{\delta_K}(K),
	\end{equation}
	and moreover
	\begin{equation}\label{eq:bound_for_x_k-barx_uniform}
		\|x_k-x\|\le C_K\,\|\eta\|.
	\end{equation}
	In particular, for all $k\ge 1$,
	\begin{equation}\label{eq:bound_of_sequential_uniform}
	\|\Delta_k\|=	\|x_{k+1}-x_k\|
		\le 2L_K\, d_{\M_2}(x_k)\leq 2L_K \tau_K^{k-1}d_{\M_2}(x_1),
	\end{equation}
	where $L_K$ is a Lipschitz constant of $\phi_1$ on $\mathcal N_{\delta_K}(K)$.
\end{lemma}

\begin{proof}
	\textbf{Step 1.}
    
 The case \(k=1\) follows from Lemma~\ref{lem:one_step_into_N}, while the assertions for
\(x_0\) are exactly the initialization bounds established there.
 
	\textbf{Step 2.}
	By Lemma~\ref{lem:uniform_fact2}, for any $y\in \M_1\cap \mathcal N_{\delta_K}(K)$, the iterate
	$y^+=\varphi(y)$ is well-defined and satisfies
	\begin{equation}\label{eq:linear_decay_dM2_uniform_again}
		d_{\M_2}(y^+)\le \tau_K\, d_{\M_2}(y).
	\end{equation}
	By Assumption~\ref{assump:alt_framework}  and compactness of $K$, after possibly shrinking $\delta_K$
	we may assume that
	\begin{equation}\label{eq:phi2_vs_projM2_uniform_again}
		\|\phi_2(y)-\Proj_{\M_2}(y)\|\le d_{\M_2}(y),
		\qquad \forall\,y\in\mathcal N_{\delta_K}(K).
	\end{equation}
	In particular,
	\begin{equation}\label{eq:phi2_step_bound_again}
		\|\phi_2(y)-y\|
		\le \|\phi_2(y)-\Proj_{\M_2}(y)\|+\|\Proj_{\M_2}(y)-y\|
		\le 2d_{\M_2}(y),
		\qquad \forall\,y\in\mathcal N_{\delta_K}(K).
	\end{equation}
	Let $L_K$ be a Lipschitz constant of $\phi_1$ on $\mathcal N_{\delta_K}(K)$; such an $L_K$ exists since
	$\phi_1$ is $C^1$ near $K$.

	\textbf{Step 3.}
	Fix $x\in K$ and $\eta\in \T_x\M$ with $\|\eta\|\le \min\{\rho_{\rm in},\delta_K\}$, and define
	$x_{k+1}=\varphi(x_k)$.
	We prove by induction that \eqref{eq:invariance_uniform} holds and that \eqref{eq:bound_for_x_k-barx_uniform}
	holds with
	\[
	C_K:=\frac32 L_K\Big(1+\frac{2L_K}{1-\tau_K}\Big).
	\]
	Note that $L_K\ge 1$ since $\rmD \phi_1(x)=\Proj_{\T_x\M_1}$. This yields that $C_K>1$.
	The case $k=1$ follows from Step~1.
	
	Assume $x_k\in \M_1\cap \mathcal N_{\delta_K}(K)$ for some $k\ge 1$.
	Then Step 2 implies that $x_{k+1}=\varphi(x_k)$ is well-defined and satisfies
	\[
	d_{\M_2}(x_{k+1})\le \tau_K\, d_{\M_2}(x_k).
	\]
	Moreover, since \(x_k\in\M_1\) and \(\varphi_1(x_k)=x_k\), the Lipschitz continuity of
\(\varphi_1\) gives  
	\[
	\|x_{k+1}-x_k\|
	=\|\phi_1(\phi_2(x_k))-\Proj_{\M_1}(x_k)\|
	\le L_K\|\phi_2(x_k)-x_k\|
	\le 2L_K\, d_{\M_2}(x_k),
	\]
	which proves \eqref{eq:bound_of_sequential_uniform}.
	
We also have  $d_{\M_2}(x_j)\le \tau_K^{j-1}d_{\M_2}(x_1)$ for any $j\leq k$. Hence
	\begin{equation}\label{eq:summable}
	\sum_{j=1}^{k}\|x_{j+1}-x_j\|
	\le \sum_{j=1}^{k}2L_K\, d_{\M_2}(x_j)
	\le 2L_K\, d_{\M_2}(x_1)\sum_{j=0}^{\infty}\tau_K^j
	= \frac{2L_K}{1-\tau_K}\, d_{\M_2}(x_1).
	\end{equation}
	Therefore, it follows that
	\[
	\|x_{k+1}-x\|
	\le \|x_1-x\|+\sum_{j=1}^{k}\|x_{j+1}-x_j\|
	\le \frac32 L_K\|\eta\|+\frac{2L_K}{1-\tau_K}\cdot \frac32 L_K\|\eta\|
	= C_K\|\eta\|.
	\]
	Now choose
	\[
	\rho_K:=\min\Big\{\rho_{\rm in},\,\frac{\delta_K}{3C_K}\Big\}.
	\]
	Then $\|x_{k+1}-x\|\le \delta_K/3$, which implies $x_{k+1}\in  \mathcal N_{\delta_K}(K).$
	
	Finally, using $x\in \M_2$ and \eqref{eq:phi2_vs_projM2_uniform_again}, we have 
	\[
	\begin{aligned}
\|\phi_2(x_{k+1})-x\|&	\le \|\phi_2(x_{k+1})-\Proj_{\M_2}(x_{k+1})\|+\|\Proj_{\M_2}(x_{k+1})-x_{k+1}\|+\|x_{k+1}-x\|\\
&\le 2d_{\M_2}(x_{k+1})+\|x_{k+1}-x\| \\
&\le 3\|x_{k+1}-x\|
\le \delta_K,
	\end{aligned}
	\]
	so $\phi_2(x_{k+1})\in \mathcal N_{\delta_K}(K)$.
	This proves \eqref{eq:invariance_uniform} for $k+1$.  
\end{proof}
With the   properties established in Lemma~\ref{lem:uniform_invariance_summable},
we can further show that the alternating mapping admits a unique limit in $\M$,
and this limit coincides with the orthogonal projection onto $\M$ up to a first-order error.

\begin{proposition}\label{prop:limit_and_projection_error}
	Fix $\bar x\in\M$, set  $K$ as defined in \eqref{eq:closure}. 
Suppose Assumptions~\ref{assumpt:M1andM2_intersect_cleanly} and \ref{assump:alt_framework} hold at any $x\in K$. Then, there exist constants $\delta_K,\rho_K>0$ such that
	  for any $x\in K$ and any $\eta\in \T_x\M$ with $\|\eta\|\leq \rho_K$,
	the sequence $\{x_k\}$ generated by $x_{k+1}=\varphi(x_k)$ with $x_0=x+\eta$
	converges to a unique limit $x^*\in \M\cap \mathcal{N}_{\delta_K}(K)$. Moreover, we have 	\begin{equation}\label{eq:bound_to_M}
		\|x^*-\Proj_{\M}(x_0)\|= o(\|\eta\|)\quad \text{as } \eta\to 0.
	\end{equation}
	and there exists $\widehat C_K>0$ such that
	\begin{equation}\label{ineq:R-linear_x_k}
		\|x_k-x^*\|\leq  \widehat C_K\,\tau_K^{k}\|\eta\|.
	\end{equation}
\end{proposition}

\begin{proof}
By Lemma~\ref{lem:uniform_invariance_summable}, \eqref{eq:summable} holds for any $k\ge 1$. Hence, $\{x_k\}$ is  Cauchy and has a single limit $x^*$. Moreover, we have
$d_{\M_2}(x_k)$ converges to 0. 
By the closedness of $\M_2$ and $x_k\in \M_1$, we get $x^*\in \M$.  
 
Since $x\in\M_2$, we have $d_{\M_2}(x_1)\le \|x_1-x\|$.
By \eqref{eq:bound_for_x_k-barx_uniform} with $k=1$, $d_{\M_2}(x_1)\le C_K\|\eta\|$.
Then, by summing \eqref{eq:summable},
\[
\|x_k-x^*\|
\le \sum_{j=k}^{\infty}\|\Delta_j\|
\le \frac{2L_K}{1-\tau_K}\,\tau_K^{k-1}d_{\M_2}(x_1)
\le \widehat C_K\,\tau_K^{k}\|\eta\|,
\]
where $\widehat C_K:=\frac{2L_KC_K}{(1-\tau_K)\tau_K}$.

Since $d_{\M_2}(x_1)\leq d_{\M_2}(x_0)+\|x_1-x_0\|$, combined with  \eqref{eq:tangent_distance_small_o_K} and Lemma~\ref{lem:one_step_into_N} gives $d_{\M_2}(x_1) = o(\|\eta\|)\  \text{as } \eta\to 0.$ Furthermore, we have
  \[
 \begin{aligned}
 	\|x^*-\Proj_{\M}(x_0)\|& \leq \|x^*-x_0\|+\|x_0- \Proj_{\M}(x_0)\|\\
 	&\leq \sum_{k=1}^\infty \|\Delta_k\|+\|x_0- \Proj_{\M}(x_0)\|\\
 	&\leq \frac{2L_K}{1-\tau_K}d_{\M_2}(x_1)+\|x_0- \Proj_{\M}(x_0)\|.
 \end{aligned}
 \]
 Note that $\|x_0- \Proj_{\M}(x_0)\|=o(\|\eta\|)$ since $\M$ is  $C^2$. We proved \eqref{eq:bound_to_M}.
  
\end{proof}
Note that \eqref{eq:bound_to_M} implies that the limiting map $\psi$ satisfies
the first-order expansion condition \eqref{eq:first-order-retraction}.
Proposition~\ref{prop:limit_and_projection_error} applies to the
projection-type methods reviewed in Section~\ref{section:revist-inexact-alt}
whenever Assumptions~\ref{assumpt:M1andM2_intersect_cleanly} and
\ref{assump:alt_framework} hold; this includes APM and the linearized
projection map \eqref{eq:projection_to_linearization}. The general APHL
correction is not covered by this proposition because its projective mapping
$Q$ typically induces an oblique or variable-metric correction. Nevertheless,
APHL satisfies the first-order expansion when viewed as an overall solver
initialized at $P_{\M_1}(x+\eta)$; see Appendix~\ref{app:APHL_first_order}.

 {We remark that  \eqref{eq:first-order-retraction} is the essential condition for establishing the global convergence of first-order \cite{boumal2019global,chen2020proximal} and stochastic methods \cite{davis2025stochastic}. Indeed, recent works have explicitly formulated this requirement as a uniform second-order bound (e.g., Definition 5.1 in \cite{davis2025stochastic}), which does not strictly require the retraction to be $C^1$ smooth with respect to $x$. Nevertheless, such smoothness is a standard assumption in the classical Riemannian optimization algorithms \cite{Absil2009,boumal2023introduction} to guarantee the uniformity of such bounds and to facilitate the analysis of higher-order local convergence. Therefore, to align with these established settings and ensure the robustness of the algorithmic mapping, it remains to show that $\psi$ is $C^{p-1}$ in a neighborhood of $(\bar x,0)$ in $\T\M$.}{}

\subsection{Regularity assumptions and second-order approximation}\label{subsection:second-order}
In this section, we show that the limiting map satisfies the second-order property \eqref{eq:second-order-retraction}.
We also introduce additional assumptions to establish the $C^{p-1}$ regularity of $\psi$.
Specifically, we require Assumption~\ref{assump:DP}, which allows us to decompose the approximation error
into tangential and normal components, thereby enabling a refined analysis.
This assumption is natural: the orthogonal projection onto $\M$ is a second-order retraction
and admits exactly such a decomposition. Under Assumption~\ref{assump:alt_framework}, the first-order residual structure follows automatically,
as shown in Sections~\ref{subsection:very_assump_for_proj} and \ref{subsection:correctionness}. The second-order residual condition required when
\(p\ge3\) is imposed separately and verified for specific projection-like
maps, including exact projection and the linearized projection map.

Assumption~\ref{assump:DP} plays a key role in our analysis.
Roughly speaking, the normal component leads to a contractive behavior of the alternating map,
whereas the tangential component is summable due to the $R$-linear convergence of $\|x_k-x^*\|$.
Together, these two ingredients yield a summable error bound for the iterates,
which allows us to establish the retraction properties.
In addition, the $C^{p-1,1}$ regularity of $\varphi$ in Assumption~\ref{assump:DP} guarantees that the residual terms
depend smoothly on $(x,\eta)$, which is essential for proving that $\psi$ is $C^{p-1}$ near $\bar x$.

\begin{assumption}\label{assump:DP}
	Let $\Omega_{\bar x}\subset \M$, $K$, $\mathcal U_{\bar x}\subset \E$ and $\delta_K>0$
	be given in Lemma~\ref{lem:uniform_invariance_summable}.
	We assume the following properties hold for both $i=1,2$.
	
	For each $i\in\{1,2\}$, there exist residual mappings 
	\[
	R_{{\rmN}_i}: K\times \mathbb B(0,\delta_K)\to  \rmN_x\M_i,
	\qquad  (x, u) \mapsto R_{{\rmN}_i}(x, u) \in \rmN_x\M_i
      \]
    \[
	R_{{\rmT}_i}: K\times \mathbb B(0,\delta_K)\to \T_x\M_i, \qquad  (x, u) \mapsto R_{{\T}_i}(x, u) \in \T_x\M_i
	\]
	such that for any $x\in K$ and any $u\in \mathbb B(0,\delta_K)$, the mapping $\phi_i$ admits the decomposition
\begin{equation}\label{eq:DP_decomp_on_K}
		\phi_i(x+u)
		= x+\Proj_{\T_x \M_i}(u)+R_{{\rmN}_i}(x,u)+R_{{\rmT}_i}(x,u).
	\end{equation}
	Moreover, the maps $(x,u)\mapsto R_{{\rmN}_i}(x,u)$ and $(x,u)\mapsto R_{{\rmT}_i}(x,u)$ are $C^{p-1}$
	on $K\times \mathbb B(0,\delta_K)$.
	
	\begin{enumerate}
		\item[(i)] If $p= 2$, then as $u\to 0$,
		\[
		R_{{\rmN}_i}(x,u)=o(\|u\|),
		\qquad
		R_{{\rmT}_i}(x,u)=o(\|u\|),
		\]
		uniformly with respect to $x\in K$.
	  We additionally assume $\varphi=\phi_1\circ \phi_2\in C^{1,1}$, i.e.,    the Jacobian $\rmD\varphi$ is Lipschitz on
		$\mathcal N_{\delta_K}(K)$.  There exists $\Lambda_K>0$ such that
		\begin{equation}\label{eq:Dvarphi_Lip_on_NK}
			\|\rmD\varphi(y)-\rmD\varphi(z)\|
			\le \Lambda_K\|y-z\|,
			\qquad \forall\,y,z\in \mathcal N_{\delta_K}(K).
		\end{equation}
		
		\item[(ii)] If $p\ge 3$, then for any $x\in K$ and any $u$ with $x+u\in \mathcal N_{\delta_K}(K)$ and
		$\Proj_{\rmN_x\M_i}(u)=o(\|u\|)$, one has, as $u\to 0$,
		\begin{equation}\label{eq:condition_small_o_normal}
			R_{{\rmN}_i}(x,u)=O(\|u\|^2),
			\qquad
			R_{{\rmT}_i}(x,u)=o(\|u\|^2),
		\end{equation}
		uniformly with respect to $x\in K$. Moreover, we assume $\varphi \in C^{2,1}$, i.e.,  $\rmD^2\varphi$ is Lipschitz on
		$\mathcal N_{\delta_K}(K)$.  There exists $\Lambda_{K,2}>0$ such that
		\begin{equation}\label{eq:DDvarphi_Lip_on_NK}
			\|\rmD^2\varphi(y)-\rmD^2\varphi(z)\|
			\le \Lambda_{K,2}\|y-z\|,
			\qquad \forall\,y,z\in \mathcal N_{\delta_K}(K).
		\end{equation}
	\end{enumerate}
\end{assumption}

We remark that the condition \eqref{eq:condition_small_o_normal} is crucial for establishing second-order retraction property \eqref{eq:second-order-retraction}.
Based on Assumption~\ref{assump:DP}, we get the Lipschitz constants of the tangent residuals $R_{{\rmT}_i}$.
It will be useful in later error analysis.
\begin{lemma}\label{lem:uniform_o_r_RT}
Assume Assumption~\ref{assump:DP}(ii) holds with $p\ge3$. Let
$\omega:(0,r_K]\to\mathbb R_+$ satisfy $\omega(r)\downarrow0$ as $r\downarrow0$.
Then there exist a function $L_{i,K}:(0,r_K]\to\mathbb R_+$ with
$L_{i,K}(r)=o(r)$ and a constant $C_{i,K}>0$ such that the following holds.

For every $x\in K$ and every $u_1,u_2\in\E$ with
\[
    \|u_j\|\le r,\qquad
    \|\Proj_{N_x\M_i}u_j\|\le \omega(r)\|u_j\|,
    \qquad j=1,2,
\]
one has
\[
\begin{aligned}
    \|R_{T_i}(x,u_1)-R_{T_i}(x,u_2)\|
    &\le
    L_{i,K}(r)
    \|\Proj_{\T_x\M_i}(u_1-u_2)\|  
    + C_{i,K}r
    \|\Proj_{N_x\M_i}(u_1-u_2)\|.
\end{aligned}
\]
\end{lemma}
\begin{proof}
Fix $i\in\{1,2\}$ and write
\[
    u_j=t_j+n_j,\qquad
    t_j:=\Proj_{\T_x\M_i}u_j,\quad
    n_j:=\Proj_{N_x\M_i}u_j .
\]
Since \(R_{T_i}(x,0)=0\) and \(\rmD_uR_{T_i}(x,0)=0\), and since
Assumption~\ref{assump:DP}(ii) gives
\[
    R_{T_i}(x,t)=o(\|t\|^2),\qquad t\in \T_x\M_i,
\]
we have
\(
    \rmD^2_{uu}R_{T_i}(x,0)[v,w]=0,
    \
    \forall v,w\in \T_x\M_i,
\)
uniformly for \(x\in K\).

By the continuity of \(\rmD^2_{uu}R_{T_i}\) and the compactness of \(K\), there
exists a function \(\alpha_{i,K}(r)\downarrow0\) such that, for all
\(\|u\|\le r\),
\[
    \|\rmD^2_{uu}R_{T_i}(x,u)[v,w]\|
    \le \alpha_{i,K}(r)\|v\|\|w\|,
    \qquad
    v,w\in \T_x\M_i .
\]
Moreover, since \(\rmD^2_{uu}R_{T_i}\) is locally bounded, there is a constant
\(C_{i,K}>0\) such that
\[
    \|\rmD^2_{uu}R_{T_i}(x,u)\|\le C_{i,K},
    \qquad x\in K,\ \|u\|\le r_K .
\]

For \(u=t+n\) satisfying the near-tangency condition, \(\|n\|\le \omega(r)r\).
Using the fundamental theorem of calculus first in the tangent variables and
then in the normal variables gives
\[
\begin{aligned}
    \|R_{T_i}(x,t_1+n_1)-R_{T_i}(x,t_2+n_1)\|
    &\le
    \big(\alpha_{i,K}(r)r+C_{i,K}\omega(r)r\big)\|t_1-t_2\|,\\
    \|R_{T_i}(x,t_2+n_1)-R_{T_i}(x,t_2+n_2)\|
    &\le
    C_{i,K}r\|n_1-n_2\|.
\end{aligned}
\]
Therefore,
\[
\begin{aligned}
    \|R_{T_i}(x,u_1)-R_{T_i}(x,u_2)\|
    &\le
    L_{i,K}(r)\|t_1-t_2\|
    +
    C_{i,K}r\|n_1-n_2\|,
\end{aligned}
\]
where
\[
    L_{i,K}(r):=\big(\alpha_{i,K}(r)+C_{i,K}\omega(r)\big)r=o(r).
\]
This proves the claim.
\end{proof}

We also need the following identities associated with the  
tangent/normal decomposition at a point $x\in \M$, which is powerful in later analysis.

\begin{lemma}\label{lem:frozen_splitting_fixed_point}
	Assume  Assumption~\ref{assump:DP} holds with $p\ge 2$.
	Let $x\in \M\cap \mathcal N_{\delta_K}(K)$.
	Then, we have
	\begin{align}
		\rmD\varphi(x)
		&=\Proj_{\T_{x}\M_1}\Proj_{\T_{x}\M_2}, 
		\tag{a}\label{eq:Dphi_on_Mbar}\\
		(\rmD\varphi(x)-I)P_x=0,
		\qquad &
		Q_x\,\rmD\varphi(x)\,P_x=0, 
		\qquad 
		P_x\, \rmD\varphi(x)\, Q_x = 0.
		\tag{b}\label{eq:tangent_identity_frozen}
	\end{align}
	Moreover, there exists $c_K\in[0,1)$ such that
	\[
	\bigl\|\Proj_{\T_{x}\M_1}\Proj_{\T_{x}\M_2}-\Proj_{\T_{x}\M}\bigr\|\le c_K,
	\]
	and consequently
	\begin{equation}\tag{c}\label{eq:normal_contr_frozen}
		\|Q_x\,\rmD\varphi(x)\,Q_x\|\le c_K.
	\end{equation}
\end{lemma}

\begin{proof}
	{(a)}
	Since $x\in \M\subset \M_i$, Assumption~\ref{assump:DP}(i) gives
	$R_{{\rmN}_i}(x,u)=o(\|u\|)$ and $R_{{\rmT}_i}(x,u)=o(\|u\|)$ as $u\to 0$, hence
	\[
	\phi_i(x+u)=x+\Proj_{\T_x\M_i}(u)+o(\|u\|),
	\quad\text{so}\quad
	\rmD\phi_i(x)=\Proj_{\T_x\M_i}.
	\]
	Therefore,
	\[
	\rmD\varphi(x)=\rmD\phi_1(x)\rmD\phi_2(x)
	=\Proj_{\T_x\M_1}\Proj_{\T_x\M_2},
	\]
	which proves \eqref{eq:Dphi_on_Mbar}.
	
	{(b)}
	For any $v\in \T_x\M=\T_x\M_1\cap \T_x\M_2$, we have
	$\Proj_{\T_x\M_2}v=v$ and $\Proj_{\T_x\M_1}v=v$, hence $\rmD\varphi(x)v=v$.
	Equivalently, $(\rmD\varphi(x)-I)P_x=0$.
	Moreover, $\rmD\varphi(x)P_x$ maps into $\T_x\M$, thus applying $Q_x$ yields
	$Q_x\,\rmD\varphi(x)\,P_x=0$. Since $\rmD\varphi(x)$ maps $\big(\T_{x}\M\big)^\perp$ into itself, we get
	$P_x\rmD\varphi(x)Q_x=0.$
	
	{(c)}
	By \eqref{eq:upper_bound_cosine} and the fact that $x\in \mathcal N_{\delta_K}(K)$,
	there exists $c_K\in[0,1)$ such that
	\[
	\bigl\|\Proj_{\T_x\M_1}\Proj_{\T_x\M_2}-\Proj_{\T_x\M}\bigr\|\le c_K.
	\]
	Let $u\in\mathcal E$ and set $u_N:=Q_xu\in \N_x\M$. Since $P_xu_N=0$, we get
	\[
	\rmD\varphi(x)u_N
	=\bigl(\Proj_{\T_x\M_1}\Proj_{\T_x\M_2}-P_x\bigr)u_N.
	\]
	Thus $\|\rmD\varphi(x)u_N\|\le c_K\|u_N\|$.
	Finally,
	\[
	\|Q_x\,\rmD\varphi(x)\,Q_xu\|
	\le \|\rmD\varphi(x)u_N\|
	\le c_K\|u_N\|
	\le c_K\|u\|,
	\]
	which gives \eqref{eq:normal_contr_frozen}.  
\end{proof}
 
 To get the retraction property \eqref{eq:second-order-retraction}, we decompose the recursive errors in the tangent and normal spaces.

\begin{lemma}\label{lem:EN_ET_recursions}
	Fix $\bar x\in\M$ and suppose the assumptions of
	Proposition~\ref{prop:limit_and_projection_error} hold.
	Let $x\in K$ and $\eta\in \T_x\M$ with $\|\eta\|\le\rho_K$.
	Define $x_0=x+\eta$ and $x_{k+1}=\varphi(x_k)$ for $k\ge0$.
	
	Let
	\[
	w^k:=x_k-x,\qquad v^k:=\phi_2(x_k)-x,
	\qquad E_k:=x_k-(x+\eta)=w^k-\eta.
	\]
	Then for all $k\ge0$,
	\begin{equation}\label{eq:E_total_recursion}
		E_{k+1}=\rmD\varphi(x)\,E_k+\mathcal E_k,
	\end{equation}
	where
	\begin{equation}\label{eq:def_Ecal_k}
		\mathcal E_k
		:=R_{\rmT_1}(x,v^k)+R_{\rmN_1}(x,v^k)
		+\Proj_{\T_x\M_1}\!\big(R_{\rmT_2}(x,w^k)+R_{\rmN_2}(x,w^k)\big).
	\end{equation}
	Moreover, one has the  recursions along normal and tangent spaces as follows:
	\begin{align}
		E_{k+1}^{\rm N}
		&:=Q_xE_{k+1}
		=Q_x\rmD\varphi(x)Q_xE_k^{\N}+Q_x\mathcal E_k,
		\tag{N}\label{eq:EN_recursion}\\[2mm]
		E_{k+1}^{\rm T}
		&:=P_xE_{k+1}
		=E_k^{\T}+P_x\big(R_{\rmT_1}(x,v^k)+R_{\rmT_2}(x,w^k)\big).
		\tag{T}\label{eq:ET_recursion}
	\end{align} 
\end{lemma}

\begin{proof}
	By Lemma~\ref{lem:uniform_invariance_summable}, $x_k\in\mathcal N_{\delta_K}(K)$
	and $\phi_2(x_k)\in\mathcal N_{\delta_K}(K)$ for all $k\ge0$.
	Hence $\|w^k\|\le\delta_K$ and $\|v^k\|\le\delta_K$.
	
	Apply Assumption~\ref{assump:DP} at the base point $x\in K$:
	\[
	\phi_2(x+w^k)
	=x+\Proj_{\T_x\M_2}(w^k)+R_{\rmN_2}(x,w^k)+R_{\rmT_2}(x,w^k),
	\]
	so $v^k=\Proj_{\T_x\M_2}(w^k)+R_{\rmN_2}(x,w^k)+R_{\rmT_2}(x,w^k)$.
	Using this in the decomposition of $\phi_1$ gives
	\[
	x_{k+1}
	=\phi_1(x+v^k)
	=x+\Proj_{\T_x\M_1}(v^k)+R_{\rmN_1}(x,v^k)+R_{\rmT_1}(x,v^k).
	\]
	Therefore, it follows that
	\[
	w^{k+1}
	=\Proj_{\T_x\M_1}\Proj_{\T_x\M_2}(w^k)+\mathcal E_k,
	\]
	where $\mathcal E_k$ is exactly \eqref{eq:def_Ecal_k}.
	Subtracting $\eta$ from both sides and using $\eta\in\T_x\M$ together with
	Lemma~\ref{lem:frozen_splitting_fixed_point}\eqref{eq:tangent_identity_frozen},
	$\rmD\varphi(x)\eta=\eta$, yields \eqref{eq:E_total_recursion}.
	
	Applying $Q_x$ to \eqref{eq:E_total_recursion} and using
	Lemma~\ref{lem:frozen_splitting_fixed_point}\eqref{eq:tangent_identity_frozen},
	$Q_x\rmD\varphi(x)P_x=0$, we obtain \eqref{eq:EN_recursion}.
	Applying $P_x$ and using
	$P_x\rmD\varphi(x)P_x=P_x$ and $P_x\rmD\varphi(x)Q_x=0$ from
	Lemma~\ref{lem:frozen_splitting_fixed_point}\eqref{eq:tangent_identity_frozen},
	we get
	\[
	P_xE_{k+1}=P_xE_k+P_x\mathcal E_k.
	\]
	Finally, since $R_{\rmN_1}(x,v^k)\in\N_x\M_1\perp\T_x\M$ and
	$R_{\rmN_2}(x,w^k)\in\N_x\M_2\perp\T_x\M$, we have
	$P_xR_{\rmN_1}(x,v^k)=0$ and $P_xR_{\rmN_2}(x,w^k)=0$,
	which gives \eqref{eq:ET_recursion}. 
\end{proof}

We get the following result, which certificates that $\psi(x, \cdot)$ satisfies the second-order retraction property~\eqref{eq:second-order-retraction}. The main idea is to show that, in the tangent space, the errors are summable thanks to the $R$-linear convergence rate \eqref{ineq:R-linear_x_k}, while in the normal space the error is contractive due to the angle condition \eqref{eq:normal_contr_frozen}.
\begin{proposition}\label{prop:tangent_error_o_eta2}
	Fix $\bar x\in\M$ and suppose the assumptions of Proposition~\ref{prop:limit_and_projection_error}  hold. 	Assume moreover that Assumption~\ref{assump:DP} holds with $p\ge 3$. Let $x\in K$ and $\eta\in \T_x\M$ with $\|\eta\|\le\rho_K$.
	Let $x_k$ be generated by $x_{k+1}=\varphi(x_k)$ from $x_0=x+\eta$,
	and let $x^*\in \M\cap\mathcal N_{\delta_K}(K)$ be its limit point.
	Then,  one has
	\begin{equation} \label{eq:normal_error_o_eta2_final}
			\Proj_{\N_x\M}\bigl(x^*-(x+\eta)\bigr)=O(\|\eta\|^2),
	\end{equation}
	\begin{equation}\label{eq:tangent_error_o_eta2_final}
		\Proj_{\T_x\M}\bigl(x^*-(x+\eta)\bigr)=o(\|\eta\|^2),
		\qquad \eta\to 0,
	\end{equation}
	uniformly with respect to $x\in K$.
\end{proposition}

\begin{proof}
By Lemma~\ref{lem:EN_ET_recursions}, for all $k\ge0$,
\begin{align*}
    E_{k+1}^{\rm N}
    &=Q_x\rmD\varphi(x)Q_xE_k^{\rm N}+Q_x\mathcal E_k,\\
    E_{k+1}^{\rm T}
    &=E_k^{\rm T}
    +P_x\big(R_{\rmT_1}(x,v^k)+R_{\rmT_2}(x,w^k)\big),
\end{align*}
where $w^k=x_k-x$ and $v^k=\phi_2(x_k)-x$.

\medskip
\noindent\textbf{Step 1: Uniform second-order bound for the normal error.}
By Lemma~\ref{lem:frozen_splitting_fixed_point}\eqref{eq:normal_contr_frozen},
$\|Q_x\rmD\varphi(x)Q_x\|\le c_K$ for some $c_K\in[0,1)$ depending only on $K$.
Moreover, Assumption~\ref{assump:DP} gives
$\|R_{\rmN_i}(x,u)\|\le C\|u\|^2$ for $\|u\|\le\delta_K$, uniformly for
$x\in K$. Since \eqref{eq:bound_for_x_k-barx_uniform} gives
$\|w^k\|\le C_K\|\eta\|$ for all $k\ge0$, and $\phi_2$ is locally Lipschitz,
we also have $\|v^k\|=O(\|\eta\|)$ uniformly in $k$. Hence, by the definition
of $\mathcal E_k$, there exists $C'>0$ such that
$\|Q_x\mathcal E_k\|\le C'\|\eta\|^2$ for all $k\ge0$. Therefore,
\[
    \|E_{k+1}^{\rm N}\|
    \le c_K\|E_k^{\rm N}\|+C'\|\eta\|^2.
\]
Since $E_0^{\rm N}=0$, this yields
$\sup_{k\ge0}\|E_k^{\rm N}\|\le C_{\rm N}\|\eta\|^2$, proving
\eqref{eq:normal_error_o_eta2_final}.

\medskip
\noindent\textbf{Step 2: Tail estimate for the normal component.}
Since $E_k\to E_*:=x^*-(x+\eta)$, the tangential recursion gives
\[
    E_*^{\rm T}
    =
    \sum_{k\ge0}
    P_x\big(R_{\rmT_1}(x,v^k)+R_{\rmT_2}(x,w^k)\big).
\]
Let $w^*:=x^*-x$ and $v^*:=\phi_2(x^*)-x$. Since $x^*$ is a fixed point of
$\varphi$, the same recursion at the limit gives
\[
    P_x\big(R_{\rmT_1}(x,v^*)+R_{\rmT_2}(x,w^*)\big)=0.
\]

We next prove a normal tail estimate. At the limit point,
\[
    E_*^{\rm N}
    =
    Q_x\rmD\varphi(x)Q_xE_*^{\rm N}
    +Q_x\mathcal E_*,
\]
where $\mathcal E_*$ is obtained from $\mathcal E_k$ by replacing
$(w^k,v^k)$ with $(w^*,v^*)$. Hence
\[
  \|  E_{k+1}^{\rm N}-E_*^{\rm N}\|
    =
    \| Q_x\rmD\varphi(x)Q_x(E_k^{\rm N}-E_*^{\rm N})
    +
    Q_x(\mathcal E_k-\mathcal E_*)\|\leq c_K\|  E_{k}^{\rm N}-E_*^{\rm N}\|+  \|Q_x(\mathcal E_k-\mathcal E_*)\|.
\]
The residual maps in $\mathcal E_k$ satisfy $R(x,0)=0$ and
$\rmD_uR(x,0)=0$. Thus their derivatives are $O(r)$ on $\|u\|\le r$.
Here \(r=O(\|\eta\|)\), because
$\|w^k\|+\|w^*\|+\|v^k\|+\|v^*\|=O(\|\eta\|)$. The mean-value theorem gives
\[
    \|\mathcal E_k-\mathcal E_*\|
    =
    O(\|\eta\|)
    \big(\|w^k-w^*\|+\|v^k-v^*\|\big).
\]
Using the local Lipschitz continuity of $\phi_2$ and
Proposition~\ref{prop:limit_and_projection_error}, we have
\begin{equation}\label{eq:vk-wk-star}
    \|v^k-v^*\|
    =
    O(\|w^k-w^*\|)
    =
    O(\|x_k-x^*\|)
    =
    O(\tau_K^{\,k}\|\eta\|).
\end{equation}
Therefore, $\|Q_x(\mathcal E_k-\mathcal E_*)\|=
O(\tau_K^{\,k}\|\eta\|^2)$. Let
$q_K:=\max\{\tau_K, c_K\}\in(0,1)$. We obtain
\begin{equation}\label{eq:E^Nk}
    \|E_k^{\rm N}-E_*^{\rm N}\|
    =
    O(q_K^{\,k}\|\eta\|^2).
\end{equation}

\medskip
\noindent\textbf{Step 3: Estimate of the tangential component.}
By Step~1, after possibly shrinking the neighborhood, there exists
$r=r(\eta)=O(\|\eta\|)$ such that, uniformly for $x\in K$,
\[
    \|w^k\|,\ \|w^*\|,\ \|v^k\|,\ \|v^*\|\le r,
    \qquad k\ge0.
\]
Moreover, these vectors are uniformly nearly tangent to the corresponding
manifolds. Hence Lemma~\ref{lem:uniform_o_r_RT}, with $p\ge3$, yields a
function $L_K(r)=o(r)$ and a constant $C_K^{\rm N}>0$ such that, for all
admissible $u_1,u_2$ with $\|u_1\|,\|u_2\|\le r$,
\begin{equation}\label{eq:RTx}
    \begin{aligned}
\|R_{\rmT_i}(x,u_1)-R_{\rmT_i}(x,u_2)\|
&\le
L_K(r)\|\Proj_{\T_x\M_i}(u_1-u_2)\|
+ C_K^{\rm N} r\|\Proj_{\N_x\M_i}(u_1-u_2)\|.
\end{aligned}
\end{equation}

We first estimate the normal parts of $w^k-w^*$ and $v^k-v^*$. Since
$w^k-w^*=E_k-E_*$ and $\N_x\M_i\subset \N_x\M$, \eqref{eq:E^Nk} gives
\[
    \|\Proj_{\N_x\M_2}(w^k-w^*)\|
    \le
    \|Q_x(E_k-E_*)\|
    =
    O(q_K^{\,k}\|\eta\|^2).
\]
For $v^k-v^*$, using $\rmD\phi_2(x)=\Proj_{\T_x\M_2}$ and the $C^2$ regularity
of $\phi_2$, we have
\[
    v^k-v^*
    =
    \Proj_{\T_x\M_2}(w^k-w^*)
    +O(\|\eta\|\|w^k-w^*\|)
    =
    \Proj_{\T_x\M_2}(w^k-w^*)+O(q_K^k\|\eta\|^2).
\]
Applying $\Proj_{\N_x\M_1}$, and using
$w^k-w^*=P_x(E_k-E_*)+Q_x(E_k-E_*)$ with
$P_x(E_k-E_*)\in \T_x\M=\T_x\M_1\cap\T_x\M_2$, we obtain
\[
\begin{aligned}
\|\Proj_{\N_x\M_1}(v^k-v^*)\|
&\le
\|\Proj_{\N_x\M_1}\Proj_{\T_x\M_2}Q_x(E_k-E_*)\|
+
O(q_K^k\|\eta\|^2)  =
O(q_K^k\|\eta\|^2).
\end{aligned}
\]
Consequently, we get
\[
\|\Proj_{\N_x\M_1}(v^k-v^*)\|
+
\|\Proj_{\N_x\M_2}(w^k-w^*)\|
=
O(q_K^{\,k}\|\eta\|^2).
\]

Using \eqref{eq:RTx}, together
with \eqref{eq:vk-wk-star} and the preceding normal estimates, we obtain, for
some $\tilde C>0$,
\[
\begin{aligned}
&\|R_{\rmT_1}(x,v^k)-R_{\rmT_1}(x,v^*)\|
+\|R_{\rmT_2}(x,w^k)-R_{\rmT_2}(x,w^*)\| \le
\tilde C\,L_K(r)\,\tau_K^{\,k}\|\eta\|
+
\tilde C\,r\,q_K^{\,k}\|\eta\|^2 .
\end{aligned}
\]
Summing over $k\ge0$ yields
\[
\begin{aligned}
\|E_*^{\rm T}\|
&=
\left\|
\sum_{k\ge0}P_x\Big(
R_{\rmT_1}(x,v^k)+R_{\rmT_2}(x,w^k)
-\big(R_{\rmT_1}(x,v^*)+R_{\rmT_2}(x,w^*)\big)
\Big)
\right\| \\
&\le
\tilde C\,L_K(r)\|\eta\|\sum_{k\ge0}\tau_K^{\,k}
+
\tilde C\,r\|\eta\|^2\sum_{k\ge0}q_K^{\,k} \\
&\le
C''\,L_K(r)\|\eta\|
+
C''\,r\|\eta\|^2 .
\end{aligned}
\]
Since $L_K(r)=o(r)$ and $r=O(\|\eta\|)$, the first term is
$o(\|\eta\|^2)$, while the second term is $O(\|\eta\|^3)$. Hence
\[
    \|E_*^{\rm T}\|=o(\|\eta\|^2),
\]
which proves \eqref{eq:tangent_error_o_eta2_final}.
\end{proof}
\begin{remark}
    For any finite iteration $k,$ we also have the second-order approximation: as $\eta\to 0,$
\begin{equation}\label{eq:tangent_error_o_eta2_finite}
		x_k = x+\eta+\mathcal R^\T_k(\eta)+\mathcal R^\N_k(\eta),\qquad \mathcal R^\T_k(\eta) = o(\|\eta\|^2),\qquad \mathcal R^\N_k(\eta) = O(\|\eta\|^2).
	\end{equation}
    This means that all the iterates approximate the orthogonal projection $\Proj_{\M}(x+\eta)$ by a second-order error.
\end{remark}

\subsection{\texorpdfstring{$C^1$}{C1} regularity of the limiting map
\texorpdfstring{$\psi$}{psi} on \texorpdfstring{$\mathrm T\M$}{TM}}
\label{subsect:C^1}
We next prove that $\psi$ is $C^1$ on a neighborhood of $(\bar x,0)$ in the tangent bundle $\T\M$.  Different from the $C^2$ regularity conditions in Proposition~\ref{prop:limit_and_projection_error}, we further need the $C^{2,1}$ regularity in Assumption~\ref{assump:DP} since we require a Lipschitz continuity condition for $\rmD \varphi$. 

Recall the series defined in~\eqref{def:xstar_series}. Our goal is to show that the differential $\rmD \Delta_k$ is also absolutely summable. For each $k\ge0$, define $F_k:\mathcal D(\bar x)\to \E$ by $F_k(x,\eta):=x_k$, and set
\begin{equation}\label{def:x_k_deriv_eta}\tag{J}
	J_k:=\mathrm D F_k(x,\eta),
\end{equation}
in the sense of \eqref{def:differential_tangent_bundle}. Since $F_{k+1}(x,\eta)=\varphi(F_k(x,\eta)),$ one has
\[
J_{k+1} = \mathrm{D}\varphi(x_k)J_k.
\]
Then
$
\mathrm D\Delta_k(x,\eta)=J_{k+1}-J_k,
$
and therefore differentiability of $\psi$ will follow once we show that
$\sum_{k\ge0}\|J_{k+1}-J_k\|$ converges uniformly. For simplicity, we will treat $J_k$ as the Jacobian of $x_k$ with respect to $(x,\eta)$,
without explicitly displaying a local trivialization of $\T\M$.

The following lemma provides uniform bounds for the derivatives $J_k$ via a tangent–normal decomposition, in the spirit of Proposition~\ref{prop:tangent_error_o_eta2}, but centered at the limit point $x^*$. We denote the orthogonal splitting at $x^*$ by
 \[
 P^*=\Proj_{\T_{x^*}\M},\qquad Q^*=\Proj_{\N_{x^*}\M}.
 \]
 \begin{lemma} \label{lem:bound_Jk_uniform}
	Under the same conditions of Proposition~\ref{prop:limit_and_projection_error},
	let $x^*\in \M\cap \mathcal N_{\delta_K}(K)$ be the limit point of $\{x_k\}$. Assume moreover that Assumption~\ref{assump:DP} holds with $p\ge 2$. Let $J_k$ be given by \eqref{def:x_k_deriv_eta}.
 Then, after possibly shrinking $\rho_K$,
 there exists a constant $q_K\in [\tau_K, 1)$ such that
\begin{equation}\label{eq:a_decay_final}
	\|Q^*J_k\|\le C_{N,K}\,q_K^{\,k}\,\|\eta\|,\qquad \forall k\ge0.
\end{equation} and   $C_{J,K}>0$
such that
\begin{equation}\label{eq:Jk_uniform_bound}
	\sup_{k\ge 0}\ \sup_{(x,\eta)\in\mathcal D_K}\ \|J_k\|
	\le C_{J,K}.
\end{equation}
\end{lemma}

\begin{proof}
By \eqref{eq:invariance_uniform} and Proposition~\ref{prop:limit_and_projection_error},
we have $x_k\in \mathcal N_{\delta_K}(K)$, $x^*\in \M\cap \mathcal N_{\delta_K}(K)$ and
\[
\|x_k-x^*\|\le \widehat C_K\,\tau_K^{\,k}\,\|\eta\|\qquad \forall k\ge0.
\]
Denote
\[
t_k:=\|P^*J_k\|,\qquad a_k:=\|Q^*J_k\|.
\]

\textbf{Step 1: recursion for the normal block.}
Since $J_{k+1}=\rmD\varphi(x_k)J_k$ and $I=P^*+Q^*$, we have
\[
Q^*J_{k+1}
=Q^*\rmD\varphi(x_k)Q^*J_k+Q^*\rmD\varphi(x_k)P^*J_k.
\]
Lemma~\ref{lem:frozen_splitting_fixed_point}\eqref{eq:normal_contr_frozen} gives
$\|Q^*\,\rmD\varphi(x^*)\,Q^*\|\le c_K$ with some $c_K\in[0,1)$.
By Lipschitz continuity of $\rmD\varphi$ and $\|x_k-x^*\|\le \widehat C_K\tau_K^{\,k}\|\eta\|$,
\[
\|Q^*(\rmD\varphi(x_k)-\rmD\varphi(x^*))Q^*\|
\le \Lambda_K\|x_k-x^*\|
\le \Lambda_K\widehat C_K\,\tau_K^{\,k}\|\eta\|
\le \Lambda_K\widehat C_K\rho_K\,\tau_K^{\,k}.
\]
Shrink $\rho_K$ so that $\Lambda_K\widehat C_K\rho_K\le (1-c_K)/2$.
Then there exists $\tilde c_K\in(c_K,1)$ such that
\begin{equation}\label{eq:normal_block_contr_along_iterates_Jk}
	\|Q^*\,\rmD\varphi(x_k)\,Q^*\|\le \tilde c_K<1,\qquad \forall k\ge0.
\end{equation}
Moreover, by Lemma~\ref{lem:frozen_splitting_fixed_point}\eqref{eq:tangent_identity_frozen},
$Q^*\rmD\varphi(x^*)P^*=0$, hence
\[
\|Q^*\,\rmD\varphi(x_k)\,P^*\|
\le \|\rmD\varphi(x_k)-\rmD\varphi(x^*)\|
\le \Lambda_K\widehat C_K\,\tau_K^{\,k}\|\eta\|.
\]
Therefore, there exists $C_1>0$ such that
\begin{equation}\label{eq:a_recursion_Jk}
	a_{k+1}\le \tilde c_K\,a_k + C_1\,\tau_K^{\,k}\,\|\eta\|\,t_k,
	\qquad \forall k\ge0.
\end{equation}

\textbf{Step 2: recursion for the tangential block.}
Similarly,
\[
P^*J_{k+1}
=P^*\rmD\varphi(x_k)P^*J_k+P^*\rmD\varphi(x_k)Q^*J_k,
\]
and thus
\begin{equation}\label{eq:t_recursion_raw_Jk}
	t_{k+1}\le \|P^*\rmD\varphi(x_k)P^*\|\,t_k+\|P^*\rmD\varphi(x_k)Q^*\|\,a_k.
\end{equation}
Since $P^*\rmD\varphi(x^*)P^*=P^*$ by Lemma~\ref{lem:frozen_splitting_fixed_point}\eqref{eq:tangent_identity_frozen},
the Lipschitz property of $\rmD\varphi$ implies
\[
\|P^*\rmD\varphi(x_k)P^*\|
\le 1+\|\rmD\varphi(x_k)-\rmD\varphi(x^*)\|
\le 1+C_2\,\tau_K^{\,k}\,\|\eta\|
\]
for some $C_2>0$.
Moreover, $\|P^*\rmD\varphi(x_k)Q^*\|$ is uniformly bounded on $\mathcal N_{\delta_K}(K)$; denote
\[
M_{\varphi,K}:=\sup_{y\in\mathcal N_{\delta_K}(K)}\|\rmD\varphi(y)\|<\infty.
\]
Then \eqref{eq:t_recursion_raw_Jk} yields
\begin{equation}\label{eq:t_recursion_Jk}
	t_{k+1}\le \bigl(1+C_2\,\tau_K^{\,k}\,\|\eta\|\bigr)t_k + M_{\varphi,K}\,a_k,
	\qquad \forall k\ge0.
\end{equation}

 Applying Lemma~\ref{lem:summable_coupled_recursion} to \eqref{eq:a_recursion_Jk} and \eqref{eq:t_recursion_Jk}, we get
\begin{equation}\label{eq:t_uniform_bound_final_Jk}
	\sup_{k\ge0}t_k\le C_{T,K},
\end{equation}
for some $C_{T,K}>0$.

Moreover, substituting \eqref{eq:t_uniform_bound_final_Jk} into \eqref{eq:a_recursion_Jk} gives
\[
a_{k+1}\le \tilde c_K a_k + C_1 C_{T,K}\,\tau_K^{\,k}\|\eta\|.\]
Let $q_K:=\max\{\tilde c_K,\tau_K\}\in(0,1)$. Then there exists $C_{N,K}>0$ such that
\begin{equation*}
	\|Q^*J_k\|=a_k\le C_{N,K}\,q_K^{\,k}\,\|\eta\|,\qquad \forall k\ge0.
\end{equation*}
Finally, since $\|J_k\|\le t_k+a_k$ and $\|\eta\|\le \rho_K\le 1$, we have
\[
\sup_{k\ge0}\|J_k\|
\le \sup_{k\ge0}t_k + \sum_{k\ge0}a_k
\le C_{T,K}+\frac{C_{N,K}}{1-q_K}
=:C_{J,K},
\]
which proves \eqref{eq:Jk_uniform_bound}.  
\end{proof}
As a result, we obtain the geometric decay of $\rmD\Delta_k$.
\begin{lemma} \label{lem:uniform_DDelta_summable}
	Under the same conditions of Lemma~\ref{lem:bound_Jk_uniform},
	there exist constants $C_{\Delta,K}>0$ and $ q_K\in(0,1)$ such that for all $k\ge0$,
	\begin{equation}\label{eq:DDelta_decay}
		\|\rmD\Delta_k\|\le C_{\Delta,K}\, q_K^{\,k}\,\|\eta\|.
	\end{equation}
	In particular, $\sum_{k\ge0}\sup_{\mathcal D_K}\|\rmD\Delta_k\|<\infty$.
\end{lemma}

\begin{proof}

    By \eqref{eq:invariance_uniform} and Proposition~\ref{prop:limit_and_projection_error},
	all iterates satisfy $x_k\in \mathcal N_{\delta_K}(K)$, $x^*\in \M\cap \mathcal N_{\delta_K}(K)$ and
	\[
	\|x_k-x^*\|\le \widehat C_K\,\tau_K^{\,k}\,\|\eta\|\qquad \forall k\ge0.
	\]
 
		By $\Delta_k=\varphi(x_k)-x_k$, we have
	\begin{equation}\label{eq:DDelta_formula_again}
		\rmD\Delta_k=\bigl(\rmD\varphi(x_k)-I\bigr)J_k,
	\end{equation}
	where  $J_k$ is given by  \eqref{def:x_k_deriv_eta}.
	Substituting $I=P^*+Q^*$   into \eqref{eq:DDelta_formula_again} yields:
	\[
    \begin{aligned}
        	\|\rmD\Delta_k\|&=\|P^* \rmD\Delta_k+Q^*\rmD\Delta_k\|\\
	&\le \|P^*(\rmD\varphi(x_k)-I) J_k\|\,
	+\|Q^*\rmD\Delta_k\|.
    \end{aligned}
	\]
	For the  tangent part, it follows from   Lemma~\ref{lem:frozen_splitting_fixed_point}\eqref{eq:tangent_identity_frozen} that
	 $(\rmD\varphi(x^*)-I)P^*=0$, hence
	\[
    \begin{aligned}
     &\quad   	\|P^*(\rmD\varphi(x_k)-I)J_k\|\\
	&\le \|P^*(\rmD\varphi(x_k)-\rmD\varphi(x^*))P^*J_k\|+\|P^*(\rmD\varphi(x_k)-\rmD\varphi(x^*))Q^*J_k\|\\
	&\le \Lambda_K\|x_k-x^*\|(\|P^*J_k\|+\|Q^*J_k\|)\\
	&\le \Lambda_K\widehat C_{P,J}\,\tau_K^{\,k}\|\eta\|,
    \end{aligned}
	\]
    where we use the the uniform bounds of $\|J_k\|$ and $\|Q^*J_k\|$ in Lemma~\ref{lem:bound_Jk_uniform} in the last inequality.
  
	For the normal contribution, one has
\[ \begin{aligned}
     \|Q^*\rmD\Delta_k\| = \|Q^*(J_{k+1}-J_k)\|
\leq \|Q^* J_{k+1}\|+\|Q^* J_k\|.
    \end{aligned}
\]
 
Therefore, there exists $C_{\Delta,K}>0$ such that
\[
\|\rmD\Delta_k\|\le C_{\Delta,K}\, q_K^{\,k}\,\|\eta\|,
\qquad \forall k\ge0,
\]
which is the desired summable bound.  
\end{proof}

Consequently, we show that $\psi$ is a first-order retraction.

\begin{theorem}\label{thm:retr-1st}
	 Fix $\bar x\in \M$ and let $\Omega_{\bar x}\subset \M$ and
	$K=\mathrm{cl}_{\M}\Omega_{\bar x}$ be defined as in Remark~\ref{rem:local_sets}. Suppose Assumptions~\ref{assumpt:M1andM2_intersect_cleanly}, \ref{assump:alt_framework}  and Assumption~\ref{assump:DP} hold with $p\ge 2$  at any $x\in K$. Then there exists    $\rho_K>0$,
such that  for each $ (x,\eta)\in \mathcal{D}_K:=\{(x,\eta)\in \T\M: x\in \Omega_{\bar x},  \|\eta\|< \rho_K \}$,   the iterates
	\[
	x_0=x+\eta,\qquad x_{k+1}=\varphi(x_k)\quad (k\ge0)
	\]
    is well-defined. Moreover,
	  the limit
	\[
	\psi(x,\eta)=\lim_{k\to\infty}x_k
	\]
	defines a $C^1$ mapping $\psi:\mathcal D_K\to \M$.
	Moreover, $\psi(x,0)=x$ and $\rmD_\eta\psi(x,0)=I$ for all $x\in K$.
	Hence, $\psi$ is a (first-order) retraction on $\M$ over $\mathcal D_K$.
\end{theorem}

\begin{proof}
By Proposition~\ref{prop:limit_and_projection_error},   the series in \eqref{def:xstar_series}
 converges absolutely. 
The $C^1$ property of $\psi$ follows from Lemma~\ref{lem:uniform_DDelta_summable}.  Moreover, it follows from \eqref{eq:bound_to_M} that $\rmD_\eta\psi(x,0)=I$.
\end{proof}

\subsection{\texorpdfstring{$C^2$}{C2} regularity of the limiting map \texorpdfstring{$\psi$}{psi}}\label{subsec:C2_limit}
In this subsection, assume Assumption~\ref{assump:DP} with $p\ge 3$, we prove that $\psi$ is $C^2$ on $\mathcal D_K$, where $\mathcal D_K$ is given by Theorem~\ref{thm:retr-1st}.
  Consider the sequence $\{x_k\}$ generated by \eqref{eq:limiting_alter_proj} with $x_0=x+\eta$, and let
 $x^*=\psi(x,\eta)\in \M\cap\mathcal N_{\delta_K}(K)$ be the limit point in
 Proposition~\ref{prop:limit_and_projection_error}.
For each $k\ge0$, Let
$\Delta_k=F_{k+1}-F_k$ be given by    \eqref{def:xstar_series}.
Recall $J_k=\rmD F_k(x,\eta)$ and  write
\begin{equation}\label{eq:H_k}\tag{H}
    H_k:=\rmD^2 F_k(x,\eta),
\end{equation}
in the sense of \eqref{def:differential_tangent_bundle}.
Then,  	we have \[
 \rmD^2\Delta_k=H_{k+1}-H_k.\] 
 Our goal is show $\sum_{k\ge 0} \|\rmD^2\Delta_k\|<\infty$. 
 The idea follows similarly as Lemma~\ref{lem:uniform_DDelta_summable}, which decomposes $ \rmD^2\Delta_k$ into the tangent and normal spaces. 
 Since $F_{k+1}=\varphi\circ F_k$, $J_{k+1}=\rmD\varphi(x_k)\,J_k$ and $\varphi$ is $C^2$ on $\mathcal N_{\delta_K}(K)$,
 the chain rule on embedded submanifolds yields
 \begin{equation}\label{eq:chain_rule_JH}
 	H_{k+1}=\rmD\varphi(x_k)\,H_k+\rmD^2\varphi(x_k)[J_k,J_k].
 \end{equation}
 Then, we need to bound $H_{k}$ by $P^*H_k$ and $Q^* H_k$.
 	Denote
 \[
 B_k:=\rmD^2\varphi(x_k)[J_k,J_k].
 \]
 Multiplying \eqref{eq:chain_rule_JH} by $P^*$ and $Q^*$ gives the    following recursion:
 \begin{equation}\label{eq:Hk_block_recursion}
 	\begin{aligned}
 		Q^*H_{k+1}
 		&=Q^*\rmD\varphi(x_k)Q^*H_k + Q^*\rmD\varphi(x_k)P^*H_k + Q^*B_k,\\
 		P^*H_{k+1}
 		&=P^*\rmD\varphi(x_k)P^*H_k + P^*\rmD\varphi(x_k)Q^*H_k + P^*B_k.
 	\end{aligned}
 \end{equation}

We start by deriving   bounds for $B_k$.

\begin{lemma}\label{lem:bound_of_Bk}	Under the same conditions of Lemma~\ref{lem:bound_Jk_uniform}, suppose Assumption~\ref{assump:DP} with $p\ge 3$ holds. Then there exist constants $\bar B_K>0$, $C_{B,\Delta} > 0$, and $C_{B,K}>0$ such that
	for all $(x,\eta)\in \mathcal D_K$ and $k\ge0$,
	\begin{align}
		\|Q^*B_k\| &\le \bar B_K, \label{eq:QBk_uniform}\\
\|B_{k+1}-B_k\|
&\le C_{B,\Delta}\,q_K^{\,k}\|\eta\|,\label{eq:QBk_increment_decay}\\
		\|P^*B_k\| &\le C_{B,K}\,q_K^{\,k}\|\eta\|, \label{eq:PBk_decay}
	\end{align}
    where $q_K\in [\tau_K, 1)$ is given by Lemma~\ref{lem:bound_Jk_uniform}.
\end{lemma}

\begin{proof}
	\textbf{Step 1.}
	Since $\rmD^2\varphi$ is bounded on $\mathcal N_{\delta_K}(K)$, define
	\[
	M_{2,K}:=\sup_{y\in \mathcal N_{\delta_K}(K)}\|\rmD^2\varphi(y)\|<\infty.
	\]
	Moreover, \eqref{eq:Jk_uniform_bound} yields $\sup_{k\ge0}\|J_k\|\le C_{J,K}$ on $\mathcal D_K$.
	Hence,
	\[
	\|Q^*B_k\|\le \|B_k\|
	\le M_{2,K}\|J_k\|^2
	\le M_{2,K}C_{J,K}^2
	=: \bar B_K,
	\]
	which proves \eqref{eq:QBk_uniform}.

	\textbf{Step 2.} 
We have
\[
\begin{aligned}
B_{k+1}-B_k
={}&\big(\rmD^2\varphi(x_{k+1})-\rmD^2\varphi(x_k)\big)[J_k,J_k]\\
&+\rmD^2\varphi(x_{k+1})[J_{k+1}-J_k,J_k]
+\rmD^2\varphi(x_{k+1})[J_{k+1},J_{k+1}-J_k].
\end{aligned}
\]
Since $\rmD^2\varphi$ is Lipschitz on $\mathcal N_{\delta_K}(K)$,
together with $\|J_k\|\le C_{J,K}$ and $J_{k+1}-J_k=\rmD\Delta_k$, we obtain
\begin{equation}\label{eq:QBk_increment}
\|B_{k+1}-B_k\|
\le C_{B,D}(\|\Delta_k\| +  \|\rmD\Delta_k\|),
\end{equation}
for some constant $C_{B,D}>0$ depending only on $K$.
Using the $R$-linear rate $\|\Delta_k\|=O(\tau^k_K)$ in \eqref{eq:bound_of_sequential_uniform} and
\eqref{eq:DDelta_decay}  for $\|\rmD\Delta_k\|$,
we conclude that there exists $C_{B,\Delta}>0$ such that \eqref{eq:QBk_increment_decay} holds.

	\textbf{Step 3.}
	Decompose $J_k=P^*J_k+Q^*J_k$ to obtain
	\[
	\begin{aligned}
		P^*B_k
		=&\,P^*\rmD^2\varphi(x_k)[P^*J_k,P^*J_k]
		+2P^*\rmD^2\varphi(x_k)[P^*J_k,Q^*J_k]\\
		&\,+P^*\rmD^2\varphi(x_k)[Q^*J_k,Q^*J_k].
	\end{aligned}
	\]
	Since $\rmD^2\varphi$ is bounded and $\|P^*J_k\|\le \|J_k\|\le C_{J,K}$, 
	$\|Q^*J_k\|$ decays geometrically by \eqref{eq:a_decay_final}, there exists $C_1>0$
	such that
	\begin{equation}\label{eq:PkBk_cross_terms_bound}
		\big\|2P^*\rmD^2\varphi(x_k)[P^*J_k,Q^*J_k]
		+P^*\rmD^2\varphi(x_k)[Q^*J_k,Q^*J_k]\big\|
		\le C_1\,q_K^{\,k}\|\eta\|.
	\end{equation}
	
	\textbf{Step 4.}
	We claim that for any $u\in \T_{x^*}{\M}$,
	\begin{equation}\label{eq:D2phi_TT_normal}
		P^*\rmD^2\varphi(x^*)[u,u]=0.
	\end{equation}

Indeed,	fix $u\in \T_{x^*}{\M}$ and take a $C^2$ curve $c:(-t_0,t_0)\to \overline{\M}$
	such that $c(0)=x^*$ and $c'(0)=u$.
	Since $\varphi$ is the identity on ${\M}$, we have $\varphi(c(t))=c(t)$.
	Differentiating twice at $t=0$ gives
	\[
	\rmD^2\varphi(x^*)[u,u]+\rmD\varphi(x^*)c''(0)=c''(0).
	\]
	Applying $P^*$ and using $P^*\rmD\varphi(x^*)=P^*$ yields
	$P^*\rmD^2\varphi(x^*)[u,u]=0$.

	By \eqref{eq:D2phi_TT_normal} with $u=P^*J_k\in \T_{x^*}{\M}$, we get
	\[
	P^*\rmD^2\varphi(x^*)[P^*J_k,P^*J_k]=0,
	\]
	and therefore
	\[
	\begin{aligned}
		\big\|P^*\rmD^2\varphi(x_k)[P^*J_k,P^*J_k]\big\|
		&=\big\|P^*\big(\rmD^2\varphi(x_k)-\rmD^2\varphi(x^*)\big)[P^*J_k,P^*J_k]\big\|\\
		&\le \|\rmD^2\varphi(x_k)-\rmD^2\varphi(x^*)\|\cdot \|P^*J_k\|^2.
	\end{aligned}
	\]
	Since $\varphi$ is $C^{2,1}$ on $\mathcal N_{\delta_K}(K)$, $\rmD^2\varphi$ is   Lipschitz.
	Combining these estimates with $\|P^*J_k\|\le C_{J,K}$ gives
	\begin{equation}\label{eq:PkBk_TT_bound}
		\big\|P^*\rmD^2\varphi(x_k)[P^*J_k,P^*J_k]\big\|
		\le C_{B,s}\,\tau_K^{\,k}\|\eta\|
	\end{equation}
	for some $C_{B,s}>0$.
	
	Set $C_{B,K}:=C_1+C_{B,s}$.
	Combining \eqref{eq:PkBk_cross_terms_bound} and \eqref{eq:PkBk_TT_bound} yields
	\[
	\|P^*B_k\|\le C_{B,K}\,q_K^{\,k}\|\eta\|,
	\]
	which is \eqref{eq:PBk_decay}.  
\end{proof}

Next, we establish a uniform bound for $H_k$ by using a strategy similar to that
in Lemma~\ref{lem:bound_Jk_uniform} for bounding $J_k$.

\begin{lemma}\label{lem:Hk_uniform_bound}
	Under the same conditions of Lemma~\ref{lem:bound_of_Bk},   there exists $C_{H,K}>0$ depending only on $K$ such that
	\[
	\sup_{k\ge0}\ \|H_k\|\ \le\ C_{H,K}.
	\]
\end{lemma}

\begin{proof}
\textbf{Step 1.}
	Since $p\ge3$, the maps $\phi_i$ are $C^2$ on $\mathcal N_{\delta_K}(K)$ by Assumption~\ref{assump:DP},
	hence $\varphi=\phi_1\circ\phi_2$ is $C^2$ on $\mathcal N_{\delta_K}(K)$. 	Then $\rmD\varphi$ is Lipschitz on $\mathcal N_{\delta_K}(K)$:
	\begin{equation}\label{eq:Dvarphi_Lip_from_D2_again}
		\|\rmD\varphi(y)-\rmD\varphi(z)\|\le M_{2,K}\|y-z\|,
		\qquad \forall y,z\in\mathcal N_{\delta_K}(K).
	\end{equation}
	Denote
	\[
	M_{2,K}:=\sup_{y\in\mathcal N_{\delta_K}(K)}\|\rmD^2\varphi(y)\|<\infty.
	\]
By Lemma~\ref{lem:frozen_splitting_fixed_point}, \eqref{eq:Dvarphi_Lip_from_D2_again} and $\|x_k-x^*\|\le \widehat C_K\tau_K^{k}\|\eta\|$,
	we obtain the following summable bounds:
	there exists $C_0>0$ such that for all $k\ge0$,
	\begin{equation}\label{eq:block_perturbations_epsk}
		\begin{aligned}
			\|Q^*(\rmD\varphi(x_k)-\rmD\varphi(x^*))Q^*\|
			&\le C_0\,\tau_K^{k}\|\eta\|,\\
			\|Q^*\rmD\varphi(x_k)P^*\|
			=\|Q^*(\rmD\varphi(x_k)-\rmD\varphi(x^*))P^*\|
			&\le C_0\,\tau_K^{k}\|\eta\|,\\
			\|P^*\rmD\varphi(x_k)Q^*\|
			=\|P^*(\rmD\varphi(x_k)-\rmD\varphi(x^*))Q^*\|
			&\le C_0\,\tau_K^{k}\|\eta\|,\\
			\|P^*(\rmD\varphi(x_k)-\rmI)P^*\|
			=\|P^*(\rmD\varphi(x_k)-\rmD\varphi(x^*))P^*\|
			&\le C_0\,\tau_K^{k}\|\eta\|.
		\end{aligned}
	\end{equation}
It follows from \eqref{eq:normal_block_contr_along_iterates_Jk} that
\begin{equation*}\label{eq:Q_block_contraction_along_iterates}
		\|Q^*\rmD\varphi(x_k)Q^*\|\le \tilde c_K<1,\qquad \forall k\ge0,
	\end{equation*}
	for some $\tilde c_K\in(c_K,1)$.

	\textbf{Step 2: bound $P^*H_k$.}
	Let $t_k:=\|P^*H_k\|$ and $a_k:=\|Q^*H_k\|$.
	From \eqref{eq:Hk_block_recursion} and \eqref{eq:block_perturbations_epsk}, we get
	\begin{equation}\label{eq:tk_recursion}
		t_{k+1}
		\le (1+\varepsilon_k)t_k+\varepsilon_k a_k+\|P^*B_k\|,
		\qquad
		\varepsilon_k:=C_0\tau_K^{\,k}\|\eta\|,
	\end{equation}
	where $\sum_{k\ge0}\varepsilon_k<\infty$.
	Moreover, \eqref{eq:PBk_decay} gives $\sum_{k\ge0}\|P^*B_k\|<\infty$.

	From \eqref{eq:Hk_block_recursion}, \eqref{eq:normal_block_contr_along_iterates_Jk},
	\eqref{eq:block_perturbations_epsk}, and \eqref{eq:QBk_uniform}, we have
		\begin{equation}\label{eq:ak_recursion}
	a_{k+1}\le \tilde c_K a_k+\varepsilon_k t_k+\bar B_K.
		\end{equation}

	\textbf{Step 3.}
	Applying Lemma~\ref{lem:summable_coupled_recursion} to \eqref{eq:tk_recursion} and \eqref{eq:ak_recursion}, we get 
	$\sup_{k\ge0}a_k\le C_{N,K}$ for some constant $C_{N,K}>0$.
	Therefore,
	\[
	\|H_k\|\le \|P^*H_k\|+\|Q^*H_k\|\le C_{T,K}+C_{N,K}=:C_{H,K},
	\qquad \forall k\ge0,
	\]
	which completes the proof.  
\end{proof}

Finally, we show that $\{\|\rmD^2\Delta_k\|\}_{k\ge0}$ is summable.

\begin{lemma} 
	\label{lem:uniform_D2Delta_summable}
	Suppose the assumptions of Lemma~\ref{lem:bound_of_Bk} hold.
	Then, after possibly shrinking $\rho_K$, there exist constants $C_{2,K}>0$ and $q_K\in(0,1)$ such that
	for all $k\ge0$,
	\begin{equation}\label{eq:D2Delta_decay}
		\sup_{(x,\eta)\in\mathcal D_K}\ \|\rmD^2\Delta_k(x,\eta)\|
		\le C_{2,K}\,q_K^{\,k}.
	\end{equation}
	Consequently, $\sum_{k\ge0}\sup_{\mathcal D_K}\|\rmD^2\Delta_k\|<\infty$ and
	$\{\rmD^2F_k\}_{k\ge0}$ converges uniformly on $\mathcal D_K$.
\end{lemma}

\begin{proof}
It follows from \eqref{eq:chain_rule_JH} that \begin{equation}\label{eq:D2Delta_identity}
\rmD^2\Delta_k
=H_{k+1}-H_k
=(\rmD\varphi(x_k)-I)H_k + B_k.
\end{equation}

\textbf{Step 1.}
Applying $P^*$ to \eqref{eq:D2Delta_identity} yields
\[
P^*\rmD^2\Delta_k
=P^*(\rmD\varphi(x_k)-I)H_k + P^*B_k.
\]
Since $P^*(\rmD\varphi(x^*)-I)=0$, we have
\[
P^*(\rmD\varphi(x_k)-I)H_k
=P^*(\rmD\varphi(x_k)-\rmD\varphi(x^*))H_k.
\]
Using the Lipschitz continuity of $\rmD\varphi$ and the $r$-linear convergence
$\|x_k-x^*\|\le \widehat C\tau_K^k\|\eta\|$, together with the uniform boundedness
$\sup_k\|H_k\|\le C_{H,K}$ in Lemma~\ref{lem:Hk_uniform_bound}, we obtain
\[
\|P^*(\rmD\varphi(x_k)-I)H_k\|
\le C_{H,P}\,\tau_K^{\,k}\|\eta\|,
\]
for some $C_{H,P}>0.$
Moreover, by Lemma~\ref{lem:bound_of_Bk} we have
$\|P^*B_k\|\le C_{B,K}\,q_K^{\,k}\|\eta\|$.
Therefore,
\begin{equation}\label{eq:P_D2Delta_decay}
\|P^*\rmD^2\Delta_k\|\le C_{P,K}\,q_K^{\,k}\|\eta\|,
\end{equation}
for some $C_{P,K}>0.$

\textbf{Step 2.} Regarding the normal part, define
\[
h_k:=Q^*H_k,\qquad d_k:=h_{k+1}-h_k = Q^*(H_{k+1}-H_k)=Q^*\rmD^2\Delta_k.
\]
It follows from \eqref{eq:D2Delta_identity} that
\begin{equation}\label{eq:h_recursion_new}
h_{k+1}
= A_k h_k + f_k,
\qquad
A_k:=Q^*\rmD\varphi(x_k)Q^*,
\qquad
f_k:=Q^*\rmD\varphi(x_k)P^*H_k + Q^*B_k.
\end{equation}
We now derive a recursion for the increment $d_k$.
One has
\begin{equation}\label{eq:d_recursion_new}
\begin{aligned}
d_{k+1}
&=h_{k+2}-h_{k+1}\\
&=(A_{k+1}h_{k+1}+f_{k+1})-(A_k h_k+f_k)\\
&=A_{k+1}(h_{k+1}-h_k) + (A_{k+1}-A_k)h_k + (f_{k+1}-f_k)\\
&=A_{k+1}d_k + r_k,
\end{aligned}
\end{equation}
where
\begin{equation}\label{eq:r_k_def}
r_k:=(A_{k+1}-A_k)h_k + (f_{k+1}-f_k).
\end{equation}
We next bound $r_k$. Since $\rmD\varphi$ is Lipschitz on $\mathcal N_{\delta_K}(K)$ and $\|h_k\|\le \|H_k\|\le C_{H,K}$ on $\mathcal D_K$,
we get
\begin{equation}\label{eq:Ak_diff_term}
\|(A_{k+1}-A_k)h_k\|
\le C_{A,K}\,\|\Delta_k\|
\le C_{A,K}\,\tau_K^{\,k}\|\eta\|,
\end{equation}
for some constant $C_{A,K}>0$. Recall $f_k=Q^*\rmD\varphi(x_k)P^*H_k + Q^*B_k$.
By the triangle inequality,
\[
\begin{aligned}
\|f_{k+1}-f_k\|
 \le &\|Q^*\rmD\varphi(x_{k+1})P^*(H_{k+1}-H_k)\|\\
 &+\|Q^*(\rmD\varphi(x_{k+1})-\rmD\varphi(x_k))P^*H_k\|+\|Q^*(B_{k+1}-B_k)\|.
 \end{aligned}
\]
For the first term, combining \eqref{eq:block_perturbations_epsk} with
and $\|H_{k+1}-H_k\|\le 2C_{H,K}$ yields
\[
\|Q^*\rmD\varphi(x_{k+1})P^*(H_{k+1}-H_k)\|
\le C_{f,1}\,q_K^{\,k}\|\eta\|.
\]
For the second term, Lipschitz continuity of $\rmD\varphi$ implies
\[
\|Q^*(\rmD\varphi(x_{k+1})-\rmD\varphi(x_k))P^*H_k\|
\le C_{f,2}\,\tau_K^{\,k}\|\eta\|.
\]
For the third term, \eqref{eq:QBk_increment_decay} yields
$\|Q^*(B_{k+1}-B_k)\|\le C_{B,\Delta}\,q_K^{\,k}\|\eta\|$.
Combining these estimates, we obtain for some $C_{f,K}>0$
\begin{equation}\label{eq:fk_increment_bound}
\|f_{k+1}-f_k\|\le C_{f,K}\,q_K^{\,k}\|\eta\|.
\end{equation}

Therefore, combining \eqref{eq:Ak_diff_term} and \eqref{eq:fk_increment_bound} gives
\begin{equation}\label{eq:r_k_bound}
\|r_k\|\le C_{r,K}\,q_K^{\,k}\|\eta\|,\qquad \text{with}\ C_{r,K}=\max\{C_{A,K}, C_{f,K}\}.
\end{equation}
Moreover, it follows from \eqref{eq:normal_block_contr_along_iterates_Jk} that
there exists $\tilde c_K\in(0,1)$ such that
\begin{equation}\label{eq:Ak_contraction}
\|A_k\|=\|Q^*\rmD\varphi(x_k)Q^*\|\le \tilde c_K,
\qquad \forall k\ge0.
\end{equation}
Substituting \eqref{eq:Ak_contraction} and \eqref{eq:r_k_bound} into \eqref{eq:d_recursion_new},
we deduce
\[
\|d_{k+1}\|\le \tilde c_K\|d_k\| + C_{r,K}\,q_K^{\,k}\|\eta\|.
\]
Since $q_K = \max\{\tau_K,\tilde c_K\},$
a standard   argument yields that there exist $C_{Q,K}>0$ such that
\begin{equation}\label{ineq:bound_Q_DDelta}
\|d_k\|=\|Q^*\rmD^2\Delta_k\|\le C_{Q,K}\, q_K^{\,k}\|\eta\|,
\qquad \forall k\ge0.
\end{equation}

\textbf{Step 3.}
Combining \eqref{eq:P_D2Delta_decay} and \eqref{ineq:bound_Q_DDelta} yields
\[
\|\rmD^2\Delta_k\|
\le \|P^*\rmD^2\Delta_k\|+\|Q^*\rmD^2\Delta_k\|
\le C_{2,K}\,q_K^{\,k}\|\eta\|,
\]
which proves \eqref{eq:D2Delta_decay}.  
\end{proof}

 Hence, we conclude that $\psi$ is a second-order retraction.
 \begin{theorem}\label{thm:retr-2nd}
 	Under the same conditions of Theorem~\ref{thm:retr-1st} and further assume Assumption~\ref{assump:DP} with $p\ge 3$. Then,
 	   the limit map $\psi$
 	defines a second-order retraction on $\mathcal D_K$.
 \end{theorem}
 
 \begin{proof}
 It follows from Theorem~\ref{thm:retr-1st} that $\psi$ is a retraction.
 	By Lemma~\ref{lem:uniform_D2Delta_summable}, $\sum_k \sup_{\mathcal D_K}\|\rmD^2\Delta_k\|<\infty$.
 	Therefore, $\sum_k \Delta_k$ converges in $C^2(\mathcal D_K;\E)$.
 	Since $x+\eta$ is smooth on $\mathcal D_K$, we conclude that
 	$\psi=x+\eta+\sum_k\Delta_k$ is $C^2$ on $\mathcal D_K$ and the derivative identities follow.
 	Moreover, the second-order retraction properties~\eqref{eq:second-order-retraction} follows from Proposition~\ref{prop:tangent_error_o_eta2}.  
 \end{proof}

\subsection{Verification  of Assumptions for  APM}\label{subsection:very_assump_for_proj}
In this section, we show that the orthogonal projection $\Proj_{\M}$ satisfies
Assumption~\ref{assump:DP} for any
$C^{3,1}$ submanifold $\M$ embedded in $\mathcal E$.  Consequently, the APM mapping
defines a second-order retraction. 

Note that the Lipschitz continuity of $\rmD\Proj_{\M}$ and $\rmD^2\Proj_{\M}$ is guaranteed by Proposition~\ref{prop:well-define_proj}.
It therefore suffices to establish the residual estimates.
The next lemma provides a Taylor expansion of the orthogonal projection.
This expansion can also be derived from the tubular neighborhood theorem \cite[Theorem~6.24]{lee2012smooth}.
For completeness, we present an elementary proof, which refines the expansion for projective retractions in \cite[Prop.~5.55]{boumal2023introduction}.

 \begin{definition}\cite[Def.5.48]{boumal2023introduction}\label{def:second_and_weingarten}
	Let $\mathcal{M}$ be an embedded $C^p (p\ge 3)$ submanifold of a Euclidean space $\mathcal{E}$.   The second fundamental form at $x$ is the bilinear  map:
	\[
	\secondform_x : \T_x \mathcal{M} \times \T_x \mathcal{M} \to \rmN_x \mathcal{M}, \quad (u,v) \mapsto \secondform_x(u,v) = \mathcal{P}_u(v).
	\]
	The Weingarten map at $x$ is the map:
	\[
	\mathcal{W}_x : \T_x \mathcal{M} \times \rmN_x \mathcal{M} \to \T_x \mathcal{M}, \quad (u,w) \mapsto \mathcal{W}_x(u,w) = \mathcal{P}_u(w).
	\]
	For both, the map $\mathcal{P}_u : \mathcal{E} \to \mathcal{E}$ is defined as follows:
	\begin{equation}\label{def:derivat_proj}
		\mathcal{P}_u := \rmD(x\mapsto \Proj_{\T_x\M})(x)[u]= \frac{\mathrm{d}}{\mathrm{d}t}\Proj_{\T_{c(t)}\M}\bigg\rvert_{t=0},
	\end{equation}
	where $c:I\rightarrow \M$ is a $C^{p-1}$ curve with $c(0)=x$ and $c'(0)=u.$ 
\end{definition}

\begin{lemma}[Expansion of the Orthogonal Projection]\label{lem:expansion of orth_proj}
	Let \(\mathcal{M}\) be a $C^p(p\ge 2)$   submanifold of \(\mathcal E\) and fix \(x\in \mathcal{M}\). Assume that the orthogonal projection \(\mathcal{P}_{\mathcal{M}}\) is well-defined on a   neighborhood of \(\mathcal{M}\). 
	\begin{enumerate}
		\item For a constant $\delta_0>0$ such that $\Proj_\M$ is well defined in $\mathbb{B}_{\delta_0}(x)$,  it follows for  any  $w \in \mathcal{E}$ with $\|w\|\leq \delta_0$   that
		\begin{equation}\label{proj_M:decomposition-1}
			\Proj_\M(x+w)=x+\mathcal{P}_{\mathrm{T}_{x}\mathcal{M}}(w)+e_x(w), \tag{DP-I}
		\end{equation}
		where $e_x:\mathbb{B}(x,\delta_0)\rightarrow\mathcal E$ is $C^{p-1}$ and $e_x(w)=o(\|w\|)$  as $u\to 0$.
		\item  
		Let $p\ge 3$.  For any $u\in \mathcal E$ with $\|u\|\leq \delta_0$, define
		\[
		u_\T = \Proj_{\T_x\M}u, \qquad u_{\rmN} = \Proj_{\rmN_x\M}u.
		\]
		Then, it follows that
			\begin{equation}\label{proj_M:decomposition-2}
						\begin{aligned}
								\mathcal{P}_{\mathcal{M}}(x+u) 
								=x+u_\T+
								\mathcal{W}_x(u_\T,  u_\rmN)+ \frac{1}{2}\secondform_x(u_\T,u_\T)+r_x(u),
							\end{aligned}\tag{DP-II}
			\end{equation}
					where $r_x:\mathbb{B}(x,\delta_0)\rightarrow\mathcal E$ is $C^{p-1}$ and $r_x(u)=o(\|u\|^2)$  as $u\to 0$. 
	\end{enumerate}
	
\end{lemma}

\begin{proof}
	\begin{enumerate}
		\item  By Proposition~\ref{prop:well-define_proj}, for a $C^p$ submanifold the projection operator $\Proj_\M$ is $C^{p-1}$ and $\Proj_\M$ is uniquely defined near $x$, and $\rmD \Proj_\M = \Proj_{\T_x\M}.$ 
		Therefore, for any $\|w\|\leq \delta_0$ 
		\[
		\Proj_\M(x+w)= x+\Proj_{\T_x\M} w+ e_x(w),
		\]
		where $e_x(w)=o(\|w\|).$ Since $\Proj_{\T_x\M} w$ is $C^{p-1}$ related to $x$,  it follows that $e_x$ is $C^{p-1}$ related to $w.$ This proves \eqref{proj_M:decomposition-1}.
		\item \textbf{Step 1.} Let  $w:I\to \mathcal E$ is any smooth curve, where $I$ is an interval containing $0$ such that $\|w(t)\|\leq \delta_0, \forall t\in I $ and  $w(0)=0$.  Then    the  curve $c:I\rightarrow\M$,  $c(t):=\Proj_\M(x+w(t))$ is well-defined and $C^{p-1}$ by Proposition~\ref{prop:well-define_proj}.    We   define
				\[
		w_\T(t)=\Proj_{\T_x\M}w(t),\quad 	v = w'(0), \quad 	v_\rmT = \Proj_{\T_x\M}v \quad \text{and}\quad v_\rmN=\Proj_{\rmN_x\M}v.
				\]

		By the optimality condition of the projection, one has
		$x+w(t)-c(t)$ is orthogonal to $\T_{c(t)}\M$ for any $t\in I$. 
		Let $P(t)=\Proj_{\T_{c(t)}\M}$ denote the orthogonal projection onto the tangent space $\T_{c(t)}\M.$ Let
		\[
		\phi(t)=P(t)(x+w(t)-c(t)).
		\]
		It follows that $\phi(t)$ is at least $C^2$ on $I$ and 
		\[
		\phi(t)\equiv 0.
		\]
		Hence, the classical derivative of $\phi$ in $\mathcal E$ satisfies
		\begin{equation}\label{eq:phi_first_derivate}
		\phi'(t)=P'(t)(x+w(t)-c(t))+P(t)(w'(t)-c'(t))\equiv 0.
		\end{equation}
		At $t=0$, one has $c(0)=x$, $c'(0)\in \T_x\M$, $w(0)=0$. Combined with  \eqref{eq:phi_first_derivate}, we get   
		\begin{equation}\label{eq:c_first_derivative}
			c'(0)= \Proj_{\T_x\M} w'(0)=v_\T.
		\end{equation}
		Consider the second-order derivative, one has
		\[
	\ddot{\phi}(t)=\ddot{P}(t)(x+w(t)-c(t))+ 2P'(t)(w'(t)-c'(t)) +P(t) (\ddot{w}(t)-\ddot{c}(t))\equiv 0.
		\]

		At $t=0,$ we substitute $c(0)=x,$ $c'(0)=v_\T$, $w(0)=0$ into above equality to get 
		\[
	 \Proj_{\T_x\M} \ddot{c}(0) = 2P'(0)(v_\rmN) + \Proj_{\T_x\M} \ddot{w}(0)=2\mathcal{W}_x(v_\rmT,v_\rmN)+ \Proj_{\T_x\M} \ddot{w}(0),
		\]
				 where we use the Weingarten map in Definition \eqref{def:second_and_weingarten}.
			By \cite[Eq.(5.44)]{boumal2023introduction}, one has
		\begin{equation}\label{eq:decomp_derivate_c}
			\ddot{c}(t)=c''(t)+\secondform_x(c'(t),c'(t)),
		\end{equation}
		where $c''(t)\in \T_{c(t)}\M$ and $\secondform_x(c'(t),c'(t))\in\rmN_{c(t)}\M$.
Note that one also has \[ \Proj_{\T_x\M} \ddot{c}(0)=c''(0). \]	 Hence,
				\[
		\ddot{c}(0)=c''(0)+  \secondform_x(v_\T, v_\T) =\Proj_{\T_x\M} \ddot{w}(0)+2\mathcal{W}_x(v_\rmT,v_\rmN)+ \secondform_x(v_\T, v_\T).
		\]

	Since $c$ is $C^2$ in the ambient space,	it follows   for any $t\in I$ that
		\begin{equation}\label{eq:expansion_c_t}
		\begin{aligned}
			c(t)&=c(0)+c'(0)t+\frac{t^2}{2}\ddot{c}(0)+o(|t|^2)\\
			&=x+tv_\T+ \frac{t^2}{2} \Proj_{\T_x\M} \ddot{w}(0)+t^2\mathcal{W}_x(v_\rmT,v_\rmN) +\frac{t^2}{2}\secondform_x(v_\T, v_\T)+o(|t|^2 ).
		\end{aligned}
				\end{equation}
		It follows from the linearity of $\Proj_{\T_x\M}$  that
		\[
		\Proj_{\T_x\M}w(t) = tv_\T+\frac{t^2}{2}\Proj_{\T_x\M}\ddot{w}(0)+o(|t|^2).
		\]
Then,	by  the fact that  both the Weingarten map and the second fundamental form are bilinear, we get 
	\[
\mathcal{W}_x(\Proj_{\T_x\M}w(t),tv_\rmN)=t^2\mathcal{W}_x(v_\rmT,v_\rmN)+ o(|t|^2)
	\]
	and
	\[
\frac{1}{2}\secondform_x(\Proj_{\T_x\M}w(t), \Proj_{\T_x\M}w(t))=	\frac{t^2}{2}\secondform_x(v_\T, v_\T)+o(|t|^2).
	\]
	Substituting the above three equations into \eqref{eq:expansion_c_t} yields that 
			\begin{equation}\label{proj_M:decomposition-wt}
					\begin{aligned}
							\mathcal{P}_{\mathcal{M}}(x+w(t)) 
							=x+w_\T(t)+
							\mathcal{W}_x(w_\T(t),  t v_\rmN)+ \frac{1}{2}\secondform_x(w_\T(t),w_\T(t))+o(|t|^2). 
						\end{aligned} 
				\end{equation}
				Define
				\[
				r(t): = c(t) - \Big(x +w_\T(t)+
				\mathcal{W}_x(w_\T(t),  t v_\rmN)+ \frac{1}{2}\secondform_x(w_\T(t),w_\T(t))\Big).
				\]
		Since    $c(t)$ is $C^{p-1}$ and $\Proj_{\T_x\M}$, $\mathcal{W}_x$ and $\secondform_x$ are smooth with respect to $t$,		$r(t)$ is $C^{p-1}$ related to $t\in I$.

\textbf{Step 2.} Now, for any $u\in\mathcal E$ with $\|u\|\leq \delta_0$,  take $w(t)=tu, t\in[-1,1]$.  Define \[
P_x(u):=\mathcal P_{\M}(x+u),\qquad u\in \mathbb B(0,\delta_0).
\] One has
\[
w_\T(t) = \Proj_{\T_x\M}(tu) = tu_\T, \quad v=u, \quad v_\rmN=u_\rmN,\quad v_\T=u_\T.
\]

Substituting these into \eqref{proj_M:decomposition-wt}, we have
\[
P_x(tu) = x+tu_\T+t^2	\mathcal{W}_x(u_\T,   u_\rmN)+ \frac{t^2}{2}\secondform_x(u_\T,u_\T)+ o(t^2),\qquad \text{as } t\to 0.
\]
 Since $\Proj_{\M}$ is $C^{p-1}$,  the above equality gives that
 \[
 \frac12\,\mathrm D^2P_x(0)[u,u]
 =
 \mathcal W_x(u_\T,u_\rmN)+\frac12\,\secondform_x(u_\T,u_\T).
 \]
 Therefore, we obtain
 \[
P_x(u)
 =
 x+u_\T+\mathcal W_x(u_\T,u_\rmN)+\frac12\,\secondform_x(u_\T,u_\T)+o(\|u\|^2),\quad \text{as } u\to 0.
 \]
\textbf{Step 3.}
 Let \[
 r_x(u):= \Proj_{\M}(x+u) - \Big(x+u_\T+\mathcal W_x(u_\T,u_\rmN)+\frac12\,\secondform_x(u_\T,u_\T)\Big).
 \]
The mappings $u\mapsto u_\T$ and $u\mapsto u_\rmN$ are linear,
and $\mathcal W_x(\cdot,\cdot)$ and $\secondform_x(\cdot,\cdot)$ are bilinear at the fixed base point $x$, hence, $u\mapsto r_x(u)$ is $C^{p-1}$ and $r_x(u)=o(\|u\|^2)$ as $u\to 0$.
\end{enumerate}
	 
\end{proof}
\begin{remark}
	(i) The second-order expansion \eqref{proj_M:decomposition-2} can be understood as follows.
	The term $\frac12\,\secondform_x(u_\T,u_\T)\in \rmN_x\M$ is the intrinsic curvature contribution:
	it measures how the manifold bends in the normal direction when moving along tangent directions.
	The mixed term $\mathcal W_x(u_\T,u_{\rmN})\in \T_x\M$ is the tangential correction induced by
	normal perturbations: a normal displacement changes the tangent spaces. Hence $\mathcal W_x(\cdot,\cdot)$ describes the coupling between tangential motion
	and normal offsets. 	Finally, there is no $O(\|u_{\rmN}\|^2)$ term: in a tubular neighborhood the projection is constant
	along normal fibers, i.e., $\Proj_\M(x+t\nu)\equiv x$ for $\nu\in \rmN_x\M$, which forces the pure
	normal second derivative to vanish. 
	
	(ii) Moreover,   we note that  $e_x(w)=O(\|w\|^2)$ if $p\ge 3$,  and $r_x(w)=O(\|w\|^3)$ if $p\ge 4$. 
\end{remark}
Lemma~\ref{lem:expansion of orth_proj} decomposes the residual of the projection into its tangent and normal components. Note that \eqref{proj_M:decomposition-1} verifies that Assumption~\ref{assump:DP} holds for the orthogonal projector when $p=2$.

 Moreover, by \eqref{proj_M:decomposition-2},    the variation in the normal direction is of second order, then the residual of the linearization in the tangent direction becomes small second order. We summarize this observation in the following proposition. Therefore, Assumption~\ref{assump:DP} (ii) holds for $\phi_i=\Proj_{\M_i}$ by a simple uniform argument of the smoothness over a compact set $K$.

\begin{proposition}\label{coro:expansion}
	Under the same conditions of Lemma \ref{lem:expansion of orth_proj}.
For any $\|u\|\leq \delta_0$,	suppose   
	\[
	\Proj_{\rmN_{{x}}\M} u = o(\|u\|),
	\]
	and denote $u_\T=\Proj_{\T_x\M}u$.
	It follows that 
	\begin{equation}\label{eq:projection_M_residual_proof}\tag{DP-III}
	 \Proj_{\mathcal{M}}(x+u)=x+ u_\T+R_{\rmN}(u) + R_{\T}(u),
	\end{equation}
	where 
	\begin{itemize}
		\item in normal space:  $R_{\rmN}:\mathbb{B}(x,\delta_0)\to \rmN_{x}\M$ is smooth and \(R_{\rmN}(u)=\frac{1}{2}\secondform_x(u_\T, u_\T) +o(\|u\|^2)\).
		\item in tangent space: $R_{\T}:\mathbb{B}(x,\delta_0)\rightarrow \T_{x}\M$ is smooth and \(R_{\T}(u)= o(\|u\|^2).\)  
	\end{itemize} 
\end{proposition}
\begin{proof}
	Since $\Proj_{\rmN_x\M}u=o(\|u\|)$, we have \[
	\mathcal{W}_x(u_{\T},u_{\rmN})=o(\|u\|^2).
	\]
Therefore, we obtain \eqref{eq:projection_M_residual_proof} by  \eqref{proj_M:decomposition-2}.
  
\end{proof}

\subsection{Verification  of Assumptions for IAPM}\label{subsection:correctionness}
As shown in Section~\ref{section:revist-inexact-alt}, surrogate maps such as
\eqref{eq:projection_to_linearization} satisfy the
first-order approximation condition to $\Proj_{\M_i}$ in
\eqref{eq:phi_inexact_proj_first_order}. This condition is sufficient for the
first-order residual property in Assumption~\ref{assump:DP}. 
  
\begin{lemma}\label{lem:expansion_inexact_proj}
Let $\M_i$ $(i=1,2)$ be a $C^{p,1}$ $(p\ge2)$ submanifold of $\mathcal E$, and
assume that $\Proj_{\M_i}$ is well-defined on a neighborhood of
$x\in \M_i$. Let $\phi_i:\mathcal E\to\mathcal E$ be a $C^{p-1,1}$ operator
which is first-order consistent with $\Proj_{\M_i}$ in the sense of
\eqref{eq:phi_inexact_proj_first_order}. Then the first-order residual
condition in Assumption~\ref{assump:DP} holds for $\phi_i$. In particular,
\[
    \rmD\phi_i(x)=\Proj_{\T_x\M_i}.
\]
\end{lemma}
\begin{proof}
By \eqref{proj_M:decomposition-1} in Lemma~\ref{lem:expansion of orth_proj},
the orthogonal projection admits the expansion
\[
    \Proj_{\M_i}(x+u)
    =
    x+\Proj_{\T_x\M_i}u+e_x(u),
    \qquad
    \|e_x(u)\|=o(\|u\|).
\]
Since $x\in\M_i$, one has
\[
    d_{\M_i}(x+u)\le \|u\|.
\]
The first-order consistency condition \eqref{eq:phi_inexact_proj_first_order}
therefore gives
\[
    \|\phi_i(x+u)-\Proj_{\M_i}(x+u)\|
    =
    o(d_{\M_i}(x+u))
    =
    o(\|u\|).
\]
Combining the last two estimates yields
\[
    \phi_i(x+u)
    =
    x+\Proj_{\T_x\M_i}u+o(\|u\|).
\]
This is exactly the first-order residual condition in
Assumption~\ref{assump:DP}. It also implies
\[
    \rmD\phi_i(x)=\Proj_{\T_x\M_i}.
\]
The proof is complete.
\end{proof}

For the
second-order residual property required when $p\ge3$, we   verify it for the specific structure of the
surrogate map.
\begin{lemma}\label{lem:linearized_projection_satisfies_DP}
Let $\M_i=\{y\in\E:H(y)=0\}$ near $x\in\M_i$, where
$H:\E\to\mathbb R^m$ is $C^{3,1}$ and $\rmD H(x)$ is surjective. Define
\[
    \varphi_i(y)
    =
    y-\rmD H(y)^*
    \big(\rmD H(y)\rmD H(y)^*\big)^{-1}H(y),
\]
which is the linearized projection map in \eqref{eq:projection_to_linearization}.
Then $\varphi_i$ satisfies Assumption~\ref{assump:DP} at $x$.
\end{lemma}

\begin{proof}
Set
\[
    A:=\rmD H(x),\qquad
    G:=AA^*,\qquad
    B(y):=\rmD H(y)^*
    \big(\rmD H(y)\rmD H(y)^*\big)^{-1}.
\]
Since $A$ is surjective, $G$ is nonsingular. Moreover,
\[
    \T_x\M_i=\ker A,\qquad
    \N_x\M_i=\operatorname{range}(A^*),
\]
and
\[
    P_{\N_x\M_i}=A^*G^{-1}A.
\]
For $y=x+u$, Taylor expansion gives
\[
    H(x+u)
    =
    Au+\frac12\rmD^2H(x)[u,u]+O(\|u\|^3),
    \qquad
    B(x+u)=A^*G^{-1}+O(\|u\|).
\]
Hence
\[
\begin{aligned}
\varphi_i(x+u)
&=
x+u-B(x+u)H(x+u)\\
&=
x+u-A^*G^{-1}Au
-\frac12A^*G^{-1}\rmD^2H(x)[u,u]
+O(\|u\|\,\|Au\|)+O(\|u\|^3).
\end{aligned}
\]
Since $A^*G^{-1}A=P_{\N_x\M_i}$, we have
\[
    u-A^*G^{-1}Au=P_{\T_x\M_i}u.
\]
Therefore
\[
\varphi_i(x+u)
=
x+P_{\T_x\M_i}u
-\frac12A^*G^{-1}\rmD^2H(x)[u,u]
+O(\|u\|\,\|Au\|)+O(\|u\|^3).
\]
This already implies the first-order residual condition in
Assumption~\ref{assump:DP}, since the remainder after
$x+P_{\T_x\M_i}u$ is $O(\|u\|^2)=o(\|u\|)$.

Now suppose $p\ge3$ and $P_{\N_x\M_i}u=o(\|u\|)$. Since
$Au=A(P_{\N_x\M_i}u)$, we have $Au=o(\|u\|)$, and hence
\[
    O(\|u\|\,\|Au\|)=o(\|u\|^2).
\]
Moreover,
\[
    A^*G^{-1}\rmD^2H(x)[u,u]\in \operatorname{range}(A^*)=\N_x\M_i.
\]
Thus the second-order leading term is purely normal, while the tangential
component of the remaining error is $o(\|u\|^2)$. More precisely, we may define
\[
    R_{\rmN_i}(x,u)
    :=
    -\frac12A^*G^{-1}\rmD^2H(x)[u,u]
    +P_{\N_x\M_i}\mathcal R(x,u),
\]
and
\[
    R_{\rmT_i}(x,u)
    :=
    P_{\T_x\M_i}\mathcal R(x,u),
\]
where
\[
    \mathcal R(x,u)
    =
    O(\|u\|\,\|Au\|)+O(\|u\|^3).
\]
Then, whenever $P_{\N_x\M_i}u=o(\|u\|)$,
\[
    R_{\rmN_i}(x,u)=O(\|u\|^2),
    \qquad
    R_{\rmT_i}(x,u)=o(\|u\|^2).
\]
The required regularity of the residual maps follows from the $C^{3,1}$
regularity of $H$ and the local nonsingularity of
$\rmD H(y)\rmD H(y)^*$.
Therefore $\varphi_i$ satisfies Assumption~\ref{assump:DP}.
\end{proof}

\subsection{NewtonSLRA: second-order retraction property}\label{subsect:NewtonSLRA}
 In this subsection, we show that the standard NewtonSLRA iterates
\eqref{eq:NewtonSLRA}, \(x_{k+1}=\phi(x_k)\), induce a second-order retraction
on \(\M\). The proof combines two ingredients: (i) the pointwise expansion of
the orthogonal projection \(P_{\M}\) established earlier in \eqref{eq:projection_M_residual_proof}, and
 (ii) the fact that the
 NewtonSLRA single-step map is quadratically close to $\Proj_{\M}$, as proved in 
 \cite[Proposition 4.8 \& Theorems~4.1]{schost2016quadratically} and \cite[Proposition 4 \& Theorem 2]{nagasaka2021relaxed}.

\begin{fact} \label{fact:NewtonSLRA_quad_ProjM}
Let $\M=\M_1\cap \M_2$ and $\bar x\in \M$. Assume that $\M$ is a $C^3$ embedded submanifold
in a neighborhood of $\bar x$, and that $\M_1$ and $\M_2$ intersect transversally at $\bar x$.
Let $\varphi$ denote the NewtonSLRA mapping \eqref{eq:NewtonSLRA}.
Then there exist radii $\delta_N>0$, $\rho_N>0$ and constants $\mu,\mu',\gamma,\gamma'>0$ such that the following statements hold.

\smallskip
\noindent{(i) Quadratic approximation of $\Proj_{\M}$.}
For all $x\in \mathbb B(\bar x, \delta_N)$, one has $\varphi(x)\in \mathbb B(\bar x,\rho_N)$ and
\begin{align}
\|\varphi(x)-\Proj_{\M}(x)\|
&\le \mu\,\|x-\Proj_{\M}(x)\|^2,
\label{eq:NewtonSLRA_phi_minus_ProjM}\\
\|\Proj_{\M}(\varphi(x))-\Proj_{\M}(x)\|
&\le \mu'\,\|x-\Proj_{\M}(x)\|^2.
\label{eq:NewtonSLRA_ProjM_phi_minus_ProjM}
\end{align}

\smallskip
\noindent{(ii) Local quadratic convergence.}
For any $x_0\in \mathbb B_\delta(\bar x)$, the sequence $\{x_k\}_{k\ge0}$ defined by
$x_{k+1}=\varphi(x_k)$ is well-defined and converges to a point $x^*\in \M$ satisfying
\begin{align}
\|x_{k+1}-x^*\|
&\le \gamma\,\|x_k-x^*\|^2,
\qquad k\ge 0,
\label{eq:NewtonSLRA_quadratic_rate_xk}\\
\|x^*-\Proj_{\M}(x_0)\|
&\le \gamma'\,\|x_0-\Proj_{\M}(x_0)\|^2.
\label{eq:NewtonSLRA_limit_vs_ProjM}
\end{align}
\end{fact}
\begin{remark}[Transversality at the limit point]\label{rem:xstar_is_transversal}
Since transversality is stable under small perturbations (i.e., it holds in a neighborhood of any transversal intersection point), shrinking $\delta>0$ if necessary, the manifolds $\M_1$ and $\M_2$ intersect
transversally at every point of $\M\cap \mathbb B_\delta(\bar x)$.
In particular, the limit point $x^*\in \M\cap \mathbb B_\delta(\bar x)$ in
Fact~\ref{fact:NewtonSLRA_quad_ProjM} is   a transversal intersection point.
\end{remark}
We still use \eqref{def:xstar_series} to describe the limiting map:
\[
\psi(x,\eta)=x^*
= x_0+\sum_{k=0}^{\infty}\Delta_k(x,\eta),
\qquad \Delta_k:=x_{k+1}-x_k,
\]
and use $J_k=\rmD F_k(x,\eta)$ as defined in~\eqref{def:x_k_deriv_eta} and
$H_k=\rmD^2 F_k(x,\eta)$ as defined in~\eqref{eq:H_k}.

\paragraph{Second-order approximation.}
Let $x_0=\bar x+\eta$ with $\eta\in \T_{\bar x}\M$ and $\|\eta\|\le \delta_N$.
It follows from \eqref{eq:NewtonSLRA_limit_vs_ProjM} that
\[
\|\psi(\bar x,\eta)-\Proj_{\M}(\bar x+\eta)\|
= O\bigl(\|x_0-\Proj_{\M}(x_0)\|^2\bigr).
\]
Since $\Proj_{\M}$ is a second-order  retraction, we have
$\|x_0-\Proj_{\M}(x_0)\|=O(\|\eta\|^2)$ for $\eta\in \T_{\bar x}\M$,
and hence
\[
\|\psi(\bar x,\eta)-\Proj_{\M}(\bar x+\eta)\| = O(\|\eta\|^4).
\]
This verifies the second-order property \eqref{eq:second-order-retraction}
for the limiting map of NewtonSLRA.
    
\paragraph{$C^2$ regularity of $\psi$.}
The second-order approximation \eqref{eq:NewtonSLRA_ProjM_phi_minus_ProjM}
is stronger than the single-manifold approximation requirement in
\eqref{eq:second_order-inexact_M2}.
As a consequence, the NewtonSLRA mapping admits a stronger   splitting than
Lemma~\ref{lem:frozen_splitting_fixed_point}. 
\begin{lemma}
\label{lem:Dphi_equals_proj_TbarxM}
Under Fact~\ref{fact:NewtonSLRA_quad_ProjM}, the mapping $\varphi$ is differentiable at any transversal point
$ x\in \M$, and
\[
\rmD\varphi( x)=\rmD\Proj_{\M}( x)=\Proj_{\T_{ x}\M}.
\]
Then, it follows that
\begin{equation}\label{eq:PQ_identity_NewtonSLRA}
Q_x\,\rmD\varphi(x)\,P_x = P_x \rmD\varphi(x) Q_x=
Q_x\,\rmD\varphi(x)\,Q_x=0,\qquad
\rmD\varphi(x)P_x=P_x.
\end{equation}
\end{lemma}

\begin{proof}
Fix $h\to 0$.
By \eqref{eq:NewtonSLRA_phi_minus_ProjM},
\begin{equation}\label{eq:phi_minus_ProjM_Oh2}
\begin{aligned}
    \|\varphi( x+h)-\Proj_{\M}( x+h)\|
&\le \mu\,\| x+h-\Proj_{\M}( x+h)\|^2\\
& = \mu \, d_{\M}( x+h)^2\\
&\leq \mu \,\| x+h- x\|^2=\mu\,\|h\|^2.
\end{aligned}
\end{equation}
On the other hand, since $\M$ is $C^3$ near $ x$,
one has
$\rmD\Proj_{\M}( x)=\Proj_{\T_{ x}\M}$.
Consequently,
\[
\Proj_{\M}( x+h)= x+\Proj_{\T_{ x}\M}(h)+O(\|h\|^2).
\]
Combining the above two estimates, we obtain
\[
\varphi( x+h)
=\Proj_{\M}( x+h)+O(\|h\|^2)
= x+\Proj_{\T_{ x}\M}(h)+O(\|h\|^2),
\]
which proves $\rmD\varphi( x)=\Proj_{\T_{ x}\M}$. Thus, the identities \eqref{eq:PQ_identity_NewtonSLRA} hold directly.  
\end{proof}

Therefore, assuming in addition that $\rmD^2\varphi$ is locally Lipschitz continuous,
the summability of $\rmD\Delta_k$ and $\rmD^2\Delta_k$ follows by repeating the proofs of
Lemmas~\ref{lem:uniform_DDelta_summable} and \ref{lem:uniform_D2Delta_summable},
with Lemma~\ref{lem:frozen_splitting_fixed_point} replaced by
\eqref{eq:PQ_identity_NewtonSLRA}.
In particular, the recursions for the normal components become strictly easier,
since the ``contraction factor'' in the alternating projection analysis can be taken as $c_K=0$.
Note that $\rmD^2\varphi$ is Lipschitz continuous whenever $\varphi$ is
$C^{2,1}$. For the standard NewtonSLRA map, this follows from the
$C^{3,1}$ regularity of $\M_2$, the affine structure of $\M_1$, and the local
nonsingularity guaranteed by transversality. The key simplification is that \eqref{eq:PQ_identity_NewtonSLRA} implies
$Q^*\,\rmD\varphi(x^*)=0$ at the limit point $x^*$.
Hence both normal components of $J_k$ and $H_k$ are controlled by the summable perturbation
$\|\rmD^{j}\varphi(x_k)-\rmD^{j}\varphi(x^*)\|=O(\|x_k-x^*\|)$ with $j=1,2$, while the tangential components remain uniformly bounded. 
For simplicity, we omit the details.

\begin{theorem}
\label{thm:newtonslra_retraction}
Under Fact~\ref{fact:NewtonSLRA_quad_ProjM}, assume that $\M_1$ is affine, $\M_2$ is a $C^{3,1}$ embedded submanifold, and
$\M_1$ and $\M_2$ intersect transversally at $\bar x$. Let \(\phi\) be the standard NewtonSLRA map \eqref{eq:NewtonSLRA}.
Then there exists $r_N>0$ such that the limiting map
\[
\psi:\T\M\to \M,
\qquad
(x, \eta)\mapsto \psi( x,\eta),
\quad \text{with}\quad \|\eta\|< r_N,
\]
is a second-order retraction near $\bar x$ on $\M$.
\end{theorem}
\begin{remark}[Relaxed NewtonSLRA]
The relaxed NewtonSLRA update \eqref{eq:relax_SLRA} may still satisfy
pointwise second-order approximation estimates or local quadratic convergence
in the settings studied in \cite{nagasaka2021relaxed}. However, the relaxed tangent space
\[
    \T_{x,\tilde x}\M_2=\operatorname{span}\{\tilde x-x\}^{\perp}
\]
does not necessarily depend smoothly on \(x\) as \(x\to\M_2\). Hence the
\(C^2\)-regularity of the induced limiting map may not hold. 
\end{remark}

\section{A hybrid algorithm}\label{sec:new_alg}
Theorems~\ref{thm:retr-1st}, \ref{thm:retr-2nd} and \ref{thm:newtonslra_retraction} provide
three retraction constructions on the intersection manifold
$\M=\M_1\cap \M_2$ for AMP \cite{lewis2008alternating,drusvyatskiy2015transversality}, iAP \cite{drusvyatskiy2019local,andersson2013alternating} satisfying Assumptions~\ref{assump:alt_framework}--\ref{assump:DP}, NewtonSLRA \cite{schost2016quadratically}. Although they all yield (second-order) retractions,
their admissible radii of $K$ and $\eta\in\T_x\M$  are different.
In typical settings, one has
$\delta_A \gtrsim \delta_i \gtrsim \delta_H$ and $\rho_A \gtrsim \rho_i \gtrsim \rho_H$, 
where the subscripts ``$A$, $i$ and $H$'' correspond to APM, iAP and NewtonSLRA, respectively.
The reduction from APM to iAP is mainly due to the condition \eqref{eq:second_order-inexact_M2}.
The  NewtonSLRA usually requires additional second-order control, which further shrinks the
admissible neighborhood.  From an algorithmic perspective, this radius hierarchy interacts with the
 per-iteration cost.
 APM enjoys the largest convergence  radius and a robust decrease of the feasibility
 residual, but its local rate is at most linear.
 iAP replaces the expensive projection $\Proj_{\M_2}$ by a cheaper surrogate
 $\mathcal A_{\M_2}$, often improving the practical
 efficiency within a
 moderate neighborhood.
 NewtonSLRA  converges at a local quadratic
 convergence, yet it becomes reliable only when the iterate is sufficiently
 close to $\M$.
 
This motivates us to combine the three phases.
Specifically, we define
\begin{equation}\label{eq:limiting_TAPR_fixed}
\psi(x+\eta)
= x_0
+ \underbrace{\sum_{i=0}^{k_A} (x_{i+1}-x_i)}_{\hbox{{APM phase}}}
+ \underbrace{\sum_{i=k_A+1}^{k_I} (x_{i+1}-x_i)}_{\hbox{{iAP phase}}}
+ \underbrace{\sum_{i=k_I+1}^{\infty} (x_{i+1}-x_i)}_{\hbox{{NewtonSLRA}}},
\end{equation}
where $k_A$ and $k_I$ denote the iteration numbers of the APM phase and the iAP phase, respectively.
The following result follows by combining the finite-time estimate~\eqref{eq:tangent_error_o_eta2_finite}
with Lemmas~\ref{lem:uniform_DDelta_summable} and~\ref{lem:uniform_D2Delta_summable}, which hold uniformly for any finite iterate $k$.
Consequently, after a phase transition, only the $R$-linear contraction factor $\tau_K\in(0,1)$ may change, while the corresponding increment terms remain summable.

\begin{theorem}\label{thm:TAPR_second_order}
	Fix $\bar x\in\M=\M_1\cap\M_2$ and suppose that Assumptions~\ref{assumpt:M1andM2_intersect_cleanly},
\ref{assump:alt_framework}, and~\ref{assump:DP} hold in a neighborhood of
$\bar x$ with $p\ge3$. Assume in addition that the standard NewtonSLRA map
used in the final phase is well-defined and satisfies the assumptions of
Theorem~\ref{thm:newtonslra_retraction}; in particular, $\M_1$ is affine,
$\M_2$ is a $C^{3,1}$ embedded submanifold, and $\M_1$ and $\M_2$ intersect
transversally at $\bar x$.

Let $k_A,k_I\in\mathbb N$ be fixed constants, independent of the input
$(x,\eta)$. Then, after possibly shrinking the neighborhood of $\bar x$, the
fixed-phase hybrid map defined by \eqref{eq:limiting_TAPR_fixed} is well-defined for
all $(x,\eta)\in \T\M$ with $x$ near $\bar x$ and $\|\eta\|$ sufficiently small.
Moreover, the induced limiting map $\psi$ is a second-order retraction on
$\M$ near $\bar x$.
\end{theorem}

\subsection{A practical algorithm}
The iteration indices $k_A$ and $k_I$ in \eqref{eq:limiting_TAPR_fixed} are a priori unknown and can be difficult to determine.
Therefore, we seek an easy-to-compute surrogate measure (e.g., a feasibility residual) to drive the phase transitions.
 As such, we assume that $\M_2$ admits an equality-constraint representation
\[
\M_2=\{x\in\R^n \mid H(x)=0\},
\]
where $H:\R^n\to \R^m$ is $C^{p,1}$ and the linear independence constraint qualification (LICQ) holds at all feasible points,
i.e., $\mathrm{D}H(x)\in\R^{m\times n}$ has full row rank for all $x\in \M_2$.
In particular, $m=n-m_2$. Since $H$ is at least $C^2$ and $\rmD H(\bar x)$ is surjective,
 	the mapping $\Phi(x):=\{H(x)\}$ is metrically regular at $(\bar x,0)$
 	\cite[Theorem~9.43]{Rockafellar2009}. Therefore, there exist $r_H>0$ and $\kappa_H>0$ such that
 	\begin{equation}\label{eq:bound_Hx}
 		d_{\M_2}(x)\le r_H\,\|H(x)\|,
 		\qquad \forall x\in \mathbb B(\bar x,r_H).
 	\end{equation}
  Combining with Lemma~\ref{lem:EB_on_M} and taking $\bar\rho:=\min\{r_1,r_H\}$, $\bar\kappa:=\kappa\,\kappa_H$, we get
    	\begin{equation}\label{eq:bound_M_2_H}
 			d_{\M}(x)\ \le\  \kappa\ d_{\M_2}(x)\ \le\   \bar\kappa\|H(x)\|.
 	\end{equation}
On the other hand, since $H$ is at least $C^{2,1}$, it is $L_H$-Lipschitz continuous on
$\mathbb{B}(\bar x,r_H)$. Thus, for any $z\in \M_2\cap \mathbb{B}(\bar x,r_H)$, one has
$\|H(x)\|\le L_H\|x-z\|$ for all $x\in \mathbb{B}(\bar x,r_H)$, which implies
\begin{equation}\label{eq:bound_Hx_lower}
	\|H(x)\|\le L_H\, d_{\M_2}(x),
	\qquad \forall x\in \mathbb{B}(\bar x,r_H).
\end{equation}
Therefore, in view of Lemma~\ref{lem:uniform_fact2}, one may expect a $Q$-linear convergence rate
for $\|H(x_k)\|$ under APM and IAPM, and---by Fact~\ref{fact:NewtonSLRA_quad_ProjM}---a quadratic rate
for NewtonSLRA whenever the admissible radii are ensured.

\begin{algorithm}
	\caption{TAPR: A Three-phase Alternating-Projections based Retraction Method}
	\label{alg:TAPR} 
	\begin{algorithmic}[1]
		\Require Manifold $\M = {\M}_1\cap\M_2$, point $ x \in \M$, tangent vector $\eta \in \T_x \M$,  $\text{tol}>0$, $1>\mathfrak{a}_1>\mathfrak{a}_2>0$, $\mathfrak a_0>0$, $0<\mu_0<\mu_1\le \mu_2<1$.
		\Ensure  $y = \Retr_{x}(\eta) \in \M$
		\State Initialize $y = x + \eta$, phase $\leftarrow$ ``APM", $\text{maxiter}, i =1$
		\State Compute initial error: $\text{err}_0 = \|H(y)\|$.
		\If {$\text{err}_0>\mathfrak a_0$}
	\State	\Return  
		\EndIf
		\While{$i\leq \text{maxiter}$ and $\text{err}_{i-1}>\mathrm{tol}$}
		\If{phase = ``APM"}
		\State $y \leftarrow \text{APM}(y, \M)$,		 $\text{err}_i = \|H(y)\|$ \Comment{alternating projections method}
		\If{$\text{err}_i < \mathfrak{a}_1$}
		\State phase $\leftarrow$ ``iAP"
		\EndIf
		\ElsIf{phase =``iAP"}
		\State $y^+ \leftarrow \text{iAP}(y,\M)$,		  \Comment{iAP}
				\If{$\|H(y^+)\|^2\leq (1-\mu_1)\|H(y)\|^2$}
		\State $y=y^+$, $\text{err}_i = \|H(y)\|$
		\Else
		\State phase $\leftarrow$ ``APM"
		\EndIf
		\If{$\text{err}_i \leq \mathfrak{a}_2$ or slow convergence detected:  $\|H(y^+)\|^2 > (1-\mu_0)\|H(y)\|^2$}
		\State phase $\leftarrow$ ``second order"
		\EndIf
		\Else \Comment{phase = "second order"}
		\State $\text{NewtonSLRA}(y, \M)$ \Comment{   NewtonSLRA}
		\If{$\|H(y^+)\|^2\leq (1-\mu_2)\|H(y)\|^2$}
		\State $y=y^+$, $\text{err}_i = \|H(y)\|$
		\Else
		\State phase $\leftarrow$ ``iAP"
		\EndIf
		\EndIf
		\State $i=i+1$.
		\EndWhile
	\end{algorithmic}
\end{algorithm}
Motivated by these observations, we propose TAPR, a three-phase hybrid
alternating-projections based algorithm that switch among APM, iAP, and a second-order
regime based on the current feasibility residual   $\|H(y)\|$, together
with a simple stagnation criterion. Since the subroutines APM, iAP and NewtonSLRA are only guaranteed to be
well-defined and convergent within suitable local neighborhoods, TAPR is
designed as a {locally valid} retraction procedure. 
The constant $\mathfrak a_0>0$ serves as an initialization safeguard: we start the iteration only if the initial feasibility residual satisfies
$\|H(x+\eta)\|\le \mathfrak a_0$. The thresholds $1>\mathfrak a_1>\mathfrak a_2>0$ specify nested neighborhoods
of $\M$ and control the regime switching.  	Once $\|H(y)\|<\mathfrak a_1$, we enter the iAP regime, where the inexact mapping
$\phi_2$ is stable, and the method exhibits a predictable
$R$-linear decrease.

When the residual further drops below $\mathfrak a_2$, or when a stagnation
criterion is triggered, we activate a second-order method, NewtonSLRA,
to exploit faster local convergence.
Here, stagnation means that the iAP phase fails to deliver the expected linear
progress; for example, the empirical contraction factor
$\|H(y_{i})\|/\|H(y_{i-1})\|$ remains persistently close to one. To globalize the second-order updates, TAPR adopts a merit-function acceptance
safeguard with fallback.
Specifically, we use $\|H(y)\|^2$ as the merit function and accept a
second-order trial point $y^+$ only if it achieves a sufficient decrease,
$\|H(y^+)\|^2\le (1-\mu)\|H(y)\|^2$ with $\mu\in(0,1)$.   Otherwise, we reject the step
and return to the iAP regime.
Algorithm~\ref{alg:TAPR} summarizes the resulting procedure.

Note that the phase-transition rules of Algorithm~\ref{alg:TAPR} generally produce different numbers of inner iterations across stages for different initial pairs $(x,\eta)$.
As a result, TAPR does not necessarily define a smooth map on a neighborhood of $(\bar x,0)$, and hence may fail to induce a retraction in the strict sense.
Nevertheless, the method is numerically effective.
Moreover, varying the number of inner iterations only changes the associated summable error sequence through constant factors in the residual bounds; such variations are typically benign and do not materially affect the observed performance.

\section{Numerical experiment}\label{sec:numerical_exp}
In this section, we compare several alternative retraction strategies within the same
algorithmic framework. Our experiments are conducted on two representative testbeds from
\cite{hou2025low}: the quadratic assignment problem (QAP)   and quadratic knapsack problem (QKP). In all cases, we embed the compared retractions into the same Riemannian
based augmented Lagrangian method (RNNAL), keeping the outer ALM and the rank-adaptive
Riemannian gradient scheme unchanged and varying only the retraction routine used for the
manifold-intersection subproblem.  All experiments were conducted in MATLAB on a Linux server equipped with two
AMD EPYC 9754 128-Core Processor. Our implementation is based on the authors'
MATLAB package  
(\url{https://github.com/HouDiOpt/RiNNAL}), and we build our retraction variants by
modifying the corresponding retraction. Since the resulting low-rank formulations are highly
nonconvex, different retractions can lead to different iterates and sometimes different
final solutions. Therefore, our goal is not to claim that one retraction is uniformly
superior, but to illustrate that alternating-projection-type constructions provide viable
alternatives to the metric projection for manifold intersections.

Following \cite{hou2025low}, both test problems considered in this section admit an equivalent
doubly non-negative (DNN) relaxation and a smooth reformulation. For our purpose of comparing different retraction
routines, it suffices to recall the resulting low-rank Riemannian subproblem appearing in each
outer iteration; more details about  the DNN model are deferred to Appendix~\ref{app:DNN}.

After the smooth reformulation, we work with a low-rank factorization
\[
Y'=
\begin{pmatrix}
1 & (x')^\top\\
x' & X'
\end{pmatrix}
=
\begin{pmatrix}
e_1^\top\\
R
\end{pmatrix}
\begin{pmatrix}
e_1 & R^\top
\end{pmatrix},
\qquad
R\in\mathbb{R}^{N\times r},
\]
where $r$ is a rank upper bound updated by the rank-adaptive strategy in \cite{hou2025low},
and we denote by $r_{\mathrm{final}}$ the final rank bound at termination.
  The feasible set of the inner Riemannian subproblem is
the intersection manifold
\begin{equation}\label{eq:Mr_general_short}
\mathcal{M}_r
:=
\left\{
R\in\mathbb{R}^{N\times r}:
A'R=b'e_1^\top,\ \diag_B(RR^\top)=R_Be_1
\right\}.
\end{equation}
Here $B\subseteq[N]$ is the index set of binary variables (denote  $s:=|B|$), $\diag_B(RR^\top)\in\R^{s}$
collects the diagonal entries indexed by $B$, and $R_B\in\R^{s\times r}$ denotes the submatrix of $R$
consisting of the rows indexed by $B$. We write
\[
\mathcal{M}_r=\mathcal{M}_1\cap\mathcal{M}_2,
\quad
\mathcal{M}_1:=\{R\in\R^{N\times r}:A'R=b'e_1^\top\},
\
\mathcal{M}_2:=\{R\in\R^{N\times r}:\diag_B(RR^\top)=R_Be_1\}.
\]
Note that the transversality condition holds for $\M_r$ according to \cite[Lemma 4]{hou2025low}.

Accordingly, the inner low-rank subproblem takes the form
\begin{equation}\label{eq:rie_sub_ALM_short}
\min_{R\in\mathcal{M}_r} f(R),
\end{equation}
where $f:\R^{N\times r}\to\R$ is the smooth objective induced by the current outer subproblem.
Given $R\in\mathcal{M}_r$ and stepsize $t>0$, we form the trial point
\[
V=R-t\,\grad f(R),
\qquad
\grad f(R)=\Proj_{T_R\mathcal{M}_r}\nabla f(R),
\]
and compute a retracted point $\Retr_R(-t\,\grad f(R))$. In \cite{hou2025low}, the retraction is
taken as the metric projection
\begin{equation}\label{eq:retr_proj_short}
\Retr_R(-t\,\grad f(R))
=
\Proj_{\mathcal{M}_r}(V)
=
\arg\min_{X\in\mathcal{M}_r}\|X-V\|_F^2.
\end{equation}

In this section, we keep the outer framework and the rank-adaptive Riemannian gradient scheme
unchanged, and compare different solvers/retractions for \eqref{eq:retr_proj_short}, including the
baseline projection solver in \cite{hou2025low} (GWA / GWA-Newton) and several
alternating-projections-type retractions (APM, NewtonSLRA, relaxed NewtonSLRA, APHL, and TAPR).

For ease of comparison, Table~\ref{tab:retr_complexity} summarizes the per-iteration
computational cost of all retraction solvers under our current notation.  The APHL results are reported for numerical comparison only and are not claimed
to follow from the retraction theory established in this paper. Detailed derivations are deferred to  Appendix~\ref{app:complexity}.

\begin{table}[t]
\centering
\caption{Per-iteration complexity of different retraction solvers. }
\label{tab:retr_complexity}
\begin{tabular}{ll}
\toprule
Method & Per-iteration cost \\
\midrule
GWA
&
$\displaystyle
O\!\left(
mnr+msr+m^2r+\min\{m^3+m^2s,\ s^3+s^2r\}
\right)
$
\\[8pt]

APM
&
$\displaystyle
O\!\left(mnr+msr+m^2r\right)
$
\\[8pt]

GWA-Newton
&
$\displaystyle
O\!\left(
mnr+
\min\bigl\{msr+m^2r+s^2r+s^3,\ (mr)^3+s(mr)^2\bigr\}
\right)
$
\\[8pt]

NewtonSLRA
&
$\displaystyle
O\!\left(
mnr+
\min\bigl\{msr+m^2r+s^2r+s^3,\ (mr)^3+s(mr)^2\bigr\}
\right)
$
\\[8pt]
Relaxed NewtonSLRA
&
$\displaystyle
O\!\left(
mnr+msr+m^2r
\right)
$\\
APHL
&
$\displaystyle
O\!\left(
mnr+
\min\bigl\{msr+m^2r+s^2r+s^3,\ (mr)^3+s(mr)^2\bigr\}
\right)
$
\\
\bottomrule
\end{tabular}
\end{table}

 \paragraph{Performance metrics and stopping criteria}
 \begin{itemize}
    
 \item {KKT residuals and optimality gap.}
We use the relative KKT residuals $(R_p,R_d,R_c)$ (primal feasibility, dual feasibility, and
complementarity) and the relative primal--dual gap $\mathrm{pdgap}$ defined in
\cite{hou2025low} to measure solution accuracy. We report $
R_{\max}:=\max\{R_p,R_d,R_c\}$
as the overall KKT residual.

\item {Rank parameter.}
We report the final rank bound $r_{\mathrm{final}}$ at termination. Smaller $r_{\mathrm{final}}$ is preferable, since a feasible mixed-binary solution corresponds to a rank-one lifted matrix $Y'$ in the DNN model. The initial rank is set to
\(
r_0=\min\{200,\ \lceil n/5\rceil\},
\)
following \cite{hou2025low}.

 \item {Stopping criterion for RNNAL.}
All RNNAL variants (differing only in the retraction solver) terminate when
\(
R_{\max}<\mathrm{tol}
\)
or when the wall-clock time exceeds $\mathrm{TimeLimit}$.
Unless otherwise stated, we set $\mathrm{tol}=10^{-6}$ and $\mathrm{TimeLimit}=3600$ seconds.

 \item {Stopping criterion for the retraction subproblem.}
At each Riemannian gradient step, the retraction subproblem is solved inexactly.
We terminate the inner retraction iterations when the chosen retraction solver reaches the
adaptive tolerance
\begin{equation}\label{eq:ret_tol}
 \mathrm{retractol}
=
\max\!\left(
\min\!\left(\frac{\|\grad f(R)\|_F}{10^{2}},\ \frac{1}{i^{3}}\right),\ 10^{-9}
\right),   
\end{equation}
where $\|\grad f(R)\|_F$ denotes the Frobenius norm of the Riemannian gradient at the current
iterate and $i$ is the inner iteration counter of the retraction solver.

 \item {Iteration counters.}
We report $\mathrm{ALMite}$ as the total number of outer augmented Lagrangian iterations and
$\mathrm{BBite}$ as the total number of inner Riemannian gradient iterations with Barzilai--Borwein stepsize, as well as the average number of inner iterations $\overline{\mathrm{Rite}}$ used by the retraction subsolver.

Regarding TAPR, it   combines APM and NewtonSLRA and the transition error $\mathfrak{a}_2$ in Algorithm~\ref{alg:TAPR} is set to $\mathrm{retractol}\times 10^3.$
 \end{itemize}

\subsection{Quadratic assignment problem}

We use the quadratic assignment problem (QAP) as the first testbed. Let $\Pi$ denote the set
of $p\times p$ permutation matrices. Given $W,D\in\mathbb S^{p}$, the QAP is
\[
\min\bigl\{\ \langle Y, W Y D\rangle \ :\ Y\in \Pi\ \bigr\}.
\]
By vectorization, let $x=\operatorname{vec}(Y)\in\mathbb R^{n}$ with $n=p^2$ and
$Q=D\otimes W$. The permutation constraints are equivalent to the assignment constraints
\[
Ax=b,\qquad x\in\{0,1\}^{n},
\]
where
\[
A=
\begin{pmatrix}
e_p^\top\otimes I_p\\
I_p\otimes e_p^\top
\end{pmatrix}
\in\mathbb R^{(2p)\times p^2},
\qquad
b=e\in\mathbb R^{2p}.
\]
Hence, in the MBQP notation \eqref{eq:mbqp}, we have $m=2p$, $B=[n]$ and thus $s=|B|=n$.

In the smooth reformulation adopted in \eqref{eq:dnn_relax_Qprime}, the extended dimension is
\[
N=n+2m=p^2+4p,
\]
and the corresponding matrices are
\[
A'=
\begin{pmatrix}
A & I_m & 0\\
A & 0 & -I_m
\end{pmatrix}
\in \mathbb R^{(2m)\times (n+2m)},
\qquad
b'=
\begin{pmatrix}
b\\
b
\end{pmatrix}
\in\mathbb R^{2m}.
\]

We first empirically verify the second-order condition~\eqref{eq:second-order-retraction}.
For dataset \texttt{chr12a} from QAPLib~\cite{hahn2006qaplib}, we take the first iterate $R_0\in\mathcal M_r$
 and set the search direction $\eta=\grad f(R_0)$.
For a range of step sizes $t\in[10^{-7},10^{-5}]$, we define the local retraction error
\[
e(t):=\;\Retr_{R_0}(t\eta)-\bigl(R_0+t\eta\bigr),
\]
and measure its tangential component $\|\Proj_{\T_{R_0}\mathcal M_r}(e(t))\|_F$.
Figure~\ref{fig:retr_second_order} reports the resulting log--log curves. The observed slope \(\approx 3\) for APM, NewtonSLRA, and APHL suggests that
these procedures exhibit second-order-type tangential error behavior in this
test. For APHL, this is an empirical observation.
\begin{figure}[ht]
    \centering
    \includegraphics[width=0.6\linewidth]{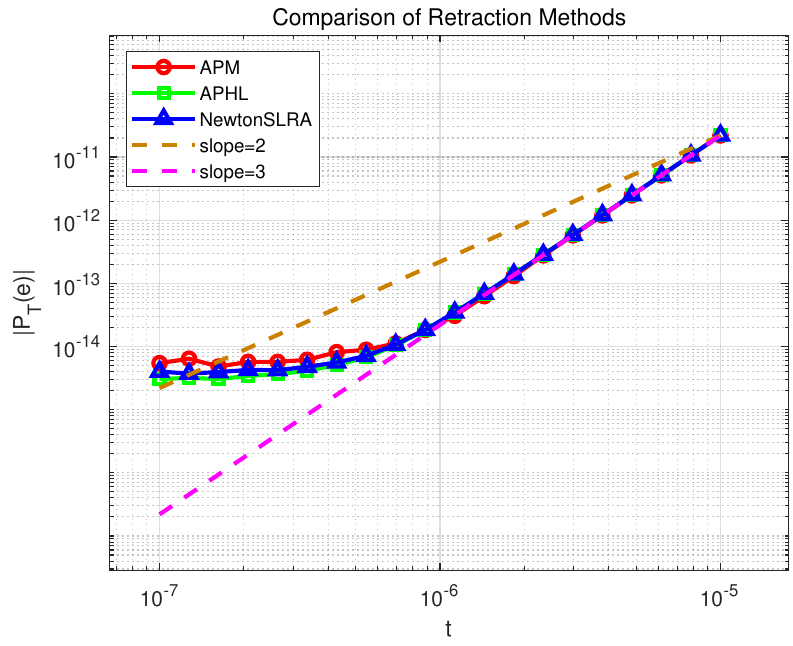}
    \caption{ Tangential retraction error
    $\|\Proj_{\T_{R_0}\mathcal M_r}\Big(\Retr_{R_0}(t\eta)-(R_0+t\eta)\Big)\|_F$ versus the step size $t$ in log--log scale. For very small $t$, the error curves plateau around $10^{-15}$ due to
    floating-point roundoff in double precision.}\label{fig:retr_second_order}
\end{figure}

The numerical results are reported in Table~\ref{tab:qap_main}. Overall, different
retractions lead to noticeably different behaviors across instances (in both
runtime and the final primal--dual metrics), and no single solver consistently dominates
the others. Additional results on all tested QAPLib instances are provided in
Table~\ref{tab:qap_all} in Appendix. The per-iteration complexity bounds in Table~\ref{tab:retr_complexity} help explain this
phenomenon. For QAP, all variables are binary, hence $B=[n]$ and $s=n=p^2$; moreover, the
assignment constraints satisfy $m=2p$. Substituting $s=n$ into Table~\ref{tab:retr_complexity} shows that (i)
first-order retractions (e.g., APM and GWA) are dominated by matrix-multiplication terms
of order $O(mnr)=O(p^3r)$, whereas (ii) Newton-type correction routines (e.g., NewtonSLRA/APHL/GWA-Newton) additionally require solving  linear systems, which in
turn have $O(s^3)=O(p^6)$   per inner iteration. Thus, there is
an inherent trade-off between cheap first-order steps and more expensive Newton-type
corrections, and the overall performance depends on instance-specific constants and the
number of ALM/BB/retraction iterations.  Similar sensitivity to the choice of retraction has also been observed in low-rank
Riemannian optimization; see \cite{absil2015low}.

\begin{table}[t]
\centering
\scriptsize
\setlength{\tabcolsep}{3pt}
\renewcommand{\arraystretch}{1.05}
\caption{QAPLib results on three representative instances. Here $pobj$ is the final primal objective value, $r$ is the final rank, $(R_p,R_d,R_c)$ are the relative KKT residuals defined in \cite{hou2025low}, and $\overline{\mathrm{Rite}}$ is the average number of inner iterations used by the retraction subsolver. 
Best values in each column are highlighted in bold within each instance.}
\label{tab:qap_main}
\resizebox{\linewidth}{!}{%
\begin{tabular}{lrrrrrrrrr}
\toprule
Method & $pobj$ & $r$ & $R_p$ & $R_d$ & $R_c$ & Time(s) & ALMite & BBite & $\overline{\mathrm{Rite}}$\\
\midrule
\multicolumn{10}{l}{\textbf{chr12b} ($n=144, p=12)$}\\
\midrule
GWA & 9.7422e+03 & 19 & 9.32e-07 & 2.45e-07 & 3.81e-07 & \textbf{8.53e+00} & 28 & \textbf{2315} & 2.67 \\
GWA-Newton & 9.7425e+03 & 11 & 8.98e-07 & 6.71e-07 & 1.90e-07 & 1.08e+01 & 28 & 2368 & \textbf{1.18} \\
APM & \textbf{9.7420e+03} & 15 & 4.21e-07 & 4.35e-10 & 2.17e-09 & 1.08e+01 & 26 & 2596 & 3.45 \\
NewtonSLRA & \textbf{9.7420e+03} & 25 & 7.11e-07 & 4.71e-10 & 4.22e-11 & 1.10e+01 & 25 & 2750 & 1.31 \\
APHL & \textbf{9.7420e+03} & \textbf{8} & \textbf{1.40e-07} & 2.58e-10 & \textbf{3.65e-11} & 1.03e+01 & 28 & 2692 & 1.25 \\
TAPR & \textbf{9.7420e+03} & 23 & 9.70e-07 & 9.35e-10 & 1.07e-10 & 1.08e+01 & \textbf{24} & 2486 & 2.29 \\
Relaxed NewtonSLRA & \textbf{9.7420e+03} & 19 & 6.96e-07 & \textbf{9.58e-17} & 5.54e-08 & 1.30e+01 & 47 & 2525 & 4.00 \\
\midrule
\multicolumn{10}{l}{\textbf{chr20b} ($n=400, p=20$)}\\
\midrule
GWA & \textbf{2.2980e+03} & 231 & \textbf{1.28e-07} & 4.15e-11 & 1.68e-10 & 3.17e+02 & 36 & 10582 & 4.24 \\
GWA-Newton & \textbf{2.2980e+03} & 233 & 2.17e-07 & 4.24e-11 & 3.46e-11 & 4.03e+02 & 36 & 11404 & \textbf{1.23} \\
APM & \textbf{2.2980e+03} & 180 & 8.63e-07 & 3.34e-10 & 5.36e-09 & \textbf{1.55e+02} & 40 & \textbf{4422} & 4.64 \\
NewtonSLRA & \textbf{2.2980e+03} & 232 & 6.51e-07 & \textbf{3.12e-11} & \textbf{9.32e-12} & 3.47e+02 & 37 & 10599 & 1.31 \\
APHL & \textbf{2.2980e+03} & 239 & 6.08e-07 & 1.36e-10 & 4.24e-09 & 3.66e+02 & \textbf{35} & 11815 & 1.36 \\
TAPR & \textbf{2.2980e+03} & 161 & 3.09e-07 & 3.03e-10 & 4.43e-10 & 2.61e+02 & 107 & 8203 & 2.33 \\
Relaxed NewtonSLRA & \textbf{2.2980e+03} & \textbf{150} & 6.94e-07 & 9.66e-10 & 4.01e-09 & 2.81e+02 & 99 & 7641 & 5.43 \\
\midrule
\multicolumn{10}{l}{\textbf{tai15b} ($n=225, p=15$)}\\
\midrule
GWA & 5.1782e+07 & 134 & 9.47e-07 & 6.99e-07 & 3.39e-07 & \textbf{4.40e+02} & 83 & \textbf{64380} & 1.21 \\
GWA-Newton & 5.1777e+07 & 130 & 8.14e-07 & 4.81e-07 & \textbf{3.14e-07} & 6.38e+02 & 83 & 67200 & \textbf{1.04} \\
APM & 5.1780e+07 & 133 & 9.79e-07 & 6.34e-07 & 3.41e-07 & 6.46e+02 & 87 & 89417 & 1.44 \\
NewtonSLRA & 5.1775e+07 & \textbf{129} & 7.84e-07 & 4.30e-07 & 4.03e-07 & 7.40e+02 & 88 & 89389 & 1.08 \\
APHL & 5.1784e+07 & 133 & 9.79e-07 & 7.23e-07 & 9.83e-07 & 6.43e+02 & \textbf{78} & 79776 & 1.10 \\
TAPR & \textbf{5.1774e+07} & 130 & \textbf{7.44e-07} & \textbf{3.75e-07} & 3.89e-07 & 8.78e+02 & 94 & 93555 & 2.07 \\
Relaxed NewtonSLRA & 5.1787e+07 & 139 & 8.12e-07 & 9.43e-07 & 3.93e-07 & 7.70e+02 & 84 & 94670 & 1.41 \\
\bottomrule
\end{tabular}%
}
\end{table}

    \section{Conclusion}
 We develop a unified framework showing that several alternating-projection-type schemes on manifold intersections induce well-defined retractions. 
Our analysis operates under the clean intersection geometry and a set of mild, verifiable assumptions on the underlying alternating mappings, which cover three representative algorithmic classes in the literature. 
The key technical ingredient is a tangent/normal decomposition derived from local projection expansions: the tangential component admits a summability estimate, while the normal component contracts under an angle-type condition. 
Finally, motivated by the trade-off between admissible neighborhood, local convergence rate, and per-iteration cost, we propose a hybrid multi-phase algorithm that adaptively combines first- and second-order alternating  updates.
\section*{Acknowledgements}

The authors are grateful to Dr. Bin Gao and  Dr. Nachuan Xiao for helpful
discussions on clean intersections and the APHL method.
	\bibliographystyle{unsrt}
	\bibliography{manifold}

@article{lewis2008alternating,
  title={Alternating projections on manifolds},
  author={Lewis, Adrian S and Malick, J{\'e}r{\^o}me},
  journal={Mathematics of Operations Research},
  volume={33},
  number={1},
  pages={216--234},
  year={2008},
  publisher={INFORMS}
}

@article{yang2025variational,
  title={Variational analysis of determinantal varieties},
  author={Yang, Yan and Gao, Bin and Yuan, Ya-xiang},
  journal={arXiv preprint arXiv:2511.22613},
  year={2025}
}

@article{bauschke2013restricted,
  title={Restricted normal cones and the method of alternating projections: applications},
  author={Bauschke, Heinz H and Luke, D Russell and Phan, Hung M and Wang, Xianfu},
  journal={Set-Valued and Variational Analysis},
  volume={21},
  number={3},
  pages={475--501},
  year={2013},
  publisher={Springer}
}

@article{schwarz1869uber,
  title={Über einige Abbildungsaufgaben},
  author={Schwarz, HA},
  journal={Gesammelte Mathematische Abhandlungen 11},
  pages={65--83},
  year={1869}
}

@article{bauschke2013restricted-theory,
  title={Restricted normal cones and the method of alternating projections: theory},
  author={Bauschke, Heinz H and Luke, D Russell and Phan, Hung M and Wang, Xianfu},
  journal={Set-Valued and Variational Analysis},
  volume={21},
  number={3},
  pages={431--473},
  year={2013},
  publisher={Springer}
}

@article{bauschke1996projection,
  title={On projection algorithms for solving convex feasibility problems},
  author={Bauschke, Heinz H and Borwein, Jonathan M},
  journal={SIAM review},
  volume={38},
  number={3},
  pages={367--426},
  year={1996},
  publisher={SIAM}
}

@article{kruger2016regularity,
	title={Regularity of Collections of Sets and Convergence of Inexact Alternating Projections},
	author={Kruger, Alexander Y and Thao, Nguyen H},
	journal={Journal of Convex Analysis},
	volume={23},
	number={3},
	pages={823--847},
	year={2016}
}

@article{smith1994optimization,
  title={Optimization Techniques on Riemannian Manifolds},
  author={Smith, Steven T},
  journal={Fields Institute Communications},
  volume={3},
  year={1994}
}

@article{gubin1967method,
  title={The method of projections for finding the common point of convex sets},
  author={Gubin, LG and Polyak, Boris T and Raik, EV},
  journal={USSR Computational Mathematics and Mathematical Physics},
  volume={7},
  number={6},
  pages={1--24},
  year={1967},
  publisher={Elsevier}
}

@article{bauschke1993convergence,
  title={On the convergence of von Neumann's alternating projection algorithm for two sets},
  author={Bauschke, Heinz H and Borwein, Jonathan M},
  journal={Set-Valued Analysis},
  volume={1},
  number={2},
  pages={185--212},
  year={1993},
  publisher={Springer}
}

@inproceedings{bregman1965method,
  title={The method of successive projections for finding a common point of convex sets},
  author={Bregman, Lev M},
  booktitle={Soviet Math. Dokl.},
  volume={6},
  pages={688--692},
  year={1965}
}

@book{von1950geometry,
  author    = {von Neumann, John},
  title     = {Functional Operators. {II}. The Geometry of Orthogonal Spaces},
  series    = {Annals of Mathematics Studies, 22},
  publisher = {Princeton University Press},
  address   = {Princeton, NJ},
  year      = {1950}
}

@article{davis2025stochastic,
  title={Stochastic Optimization Over Proximally Smooth Sets},
  author={Davis, Damek and Drusvyatskiy, Dmitriy and Shi, Zhan},
  journal={SIAM Journal on Optimization},
  volume={35},
  number={1},
  pages={157--179},
  year={2025},
  publisher={SIAM}
}

@article{xiao2025quadratically,
	title={A Quadratically Convergent Alternating Projection Method for Nonconvex Sets},
	author={Xiao, Nachuan and Wang, Shiwei and Tang, Tianyun and Toh, Kim-Chuan},
	journal={arXiv preprint arXiv:2511.22916},
	year={2025}
}

@article{andersson2013alternating,
	title={Alternating projections on nontangential manifolds},
	author={Andersson, Fredrik and Carlsson, Marcus},
	journal={Constructive approximation},
	volume={38},
	number={3},
	pages={489--525},
	year={2013},
	publisher={Springer}
}

@article{hou2025low,
  title={A low-rank augmented Lagrangian method for doubly nonnegative relaxations of mixed-binary quadratic programs},
  author={Hou, Di and Tang, Tianyun and Toh, Kim-Chuan},
  journal={Operations Research},
  year={2025},
  publisher={INFORMS}
}

@article{hahn2006qaplib,
  title={QAPLIB--a quadratic assignment problem library},
  author={Hahn, Peter and Anjos, Miguel and Burkard, RE and Karisch, SE and Rendl, F},
  journal={Avialable at http://www. seas. upenn. edu/qaplib},
  year={2006}
}

@article{wang2024adaptive,
  title={An adaptive regularized proximal Newton-type methods for composite optimization over the Stiefel manifold},
  author={Wang, Qinsi and Yang, Wei Hong},
  journal={Computational Optimization and Applications},
  volume={89},
  number={2},
  pages={419--457},
  year={2024},
  publisher={Springer}
}

@inproceedings{nagasaka2021relaxed,
  title={Relaxed NewtonSLRA for approximate GCD},
  author={Nagasaka, Kosaku},
  booktitle={International Workshop on Computer Algebra in Scientific Computing},
  pages={272--292},
  year={2021},
  organization={Springer}
}

@article{schost2016quadratically,
  title={A quadratically convergent algorithm for structured low-rank approximation},
  author={Schost, {\'E}ric and Spaenlehauer, Pierre-Jean},
  journal={Foundations of Computational Mathematics},
  volume={16},
  number={2},
  pages={457--492},
  year={2016},
  publisher={Springer}
}

@book{boumal2023introduction,
  title={An introduction to optimization on smooth manifolds},
  author={Boumal, Nicolas},
  year={2023},
  publisher={Cambridge University Press}
}

@article{noll2016local,
  title={On local convergence of the method of alternating projections},
  author={Noll, Dominikus and Rondepierre, Aude},
  journal={Foundations of Computational Mathematics},
  volume={16},
  pages={425--455},
  year={2016},
  publisher={Springer}
}

@book{deutsch2001best,
  title={Best approximation in inner product spaces},
  author={Deutsch, Frank and Deutsch, F},
  volume={7},
  year={2001},
  publisher={Springer}
}

@book{lee2012smooth,
  title={Introduction to smooth manifolds},
  author={Lee, John M},
  journal={Graduate texts in mathematics},
  year={2012},
  publisher={New York: Springer}
}

@article{absil2007trust,
  title={Trust-region methods on Riemannian manifolds},
  author={Absil, P.-A. and Baker, Christopher G and Gallivan, Kyle A},
  journal={Foundations of Computational Mathematics},
  volume={7},
  pages={303--330},
  year={2007},
  publisher={Springer}
}

@Book{Absil2009,
  title     = {Optimization algorithms on matrix manifolds},
  publisher = {Princeton University Press},
  year      = {2009},
  author    = {P.-A. Absil and R. Mahony and R. Sepulchre},
}

@Article{absil2015low,
  author  = {Absil, P.-A. and Oseledets, I. V},
  title   = {Low-rank retractions: a survey and new results},
  journal = {Computational Optimization and Applications},
  year    = {2015},
  volume  = {62},
  number  = {1},
  pages   = {5--29},
}

@article{lewis2009local,
  title={Local linear convergence for alternating and averaged nonconvex projections},
  author={Lewis, Adrian S and Luke, D Russell and Malick, J{\'e}r{\^o}me},
  journal={Foundations of Computational Mathematics},
  volume={9},
  number={4},
  pages={485--513},
  year={2009},
  publisher={Springer}
}

@article{drusvyatskiy2015transversality,
  title={Transversality and alternating projections for nonconvex sets},
  author={Drusvyatskiy, Dmitriy and Ioffe, Alexander D and Lewis, Adrian S},
  journal={Foundations of Computational Mathematics},
  volume={15},
  number={6},
  pages={1637--1651},
  year={2015},
  publisher={Springer}
}

@article{poliquin2000local,
  title={Local differentiability of distance functions},
  author={Poliquin, Ren{\'e} and Rockafellar, R and Thibault, Lionel},
  journal={Transactions of the American mathematical Society},
  volume={352},
  number={11},
  pages={5231--5249},
  year={2000}
}

@article{drusvyatskiy2019local,
  title={Local linear convergence for inexact alternating projections on nonconvex sets},
  author={Drusvyatskiy, Dmitriy and Lewis, Adrian S},
  journal={Vietnam Journal of Mathematics},
  volume={47},
  pages={669--681},
  year={2019},
  publisher={Springer}
}

@Book{Rockafellar2009,
  author    = {Rockafellar, R.T. and Wets, R. J-B},
  publisher = {Springer Science \& Business Media},
  title     = {Variational analysis},
  year      = {2009},
  volume    = {317},
}

@Article{boumal2019global,
  author    = {Boumal, Nicolas and Absil, Pierre-Antoine and Cartis, Coralia},
  title     = {Global rates of convergence for nonconvex optimization on manifolds},
  journal   = {IMA Journal of Numerical Analysis},
  year      = {2019},
  volume    = {39},
  number    = {1},
  pages     = {1--33},
  publisher = {Oxford University Press},
}

@Article{chen2020proximal,
  author    = {Chen, Shixiang and Ma, Shiqian and Man-Cho So, Anthony and Zhang, Tong},
  title     = {Proximal gradient method for nonsmooth optimization over the {Stiefel} manifold},
  journal   = {SIAM Journal on Optimization},
  year      = {2020},
  volume    = {30},
  number    = {1},
  pages     = {210--239},
  publisher = {SIAM},
}

@Article{absil2012projection,
  author    = {Absil, P.-A. and Malick, J{\'e}r{\^o}me},
  title     = {Projection-like retractions on matrix manifolds},
  journal   = {SIAM Journal on Optimization},
  year      = {2012},
  volume    = {22},
  number    = {1},
  pages     = {135--158},
  publisher = {SIAM},
}
 	 \appendix
 \section*{Appendix}
\section{Technical lemma}   

\begin{lemma} \label{lem:summable_coupled_recursion}
Let $\{t_k\}_{k\ge0}$ and $\{a_k\}_{k\ge0}$ be two nonnegative sequences satisfying
\begin{align}
t_{k+1}
&\le (1+\varepsilon_k)\,t_k + \alpha_k\,a_k + p_k,
\qquad k\ge0, 
\label{eq:summable_coupled_t}\\
a_{k+1}
&\le c\,a_k + \beta_k\,t_k + b_k,
\qquad k\ge0,
\label{eq:summable_coupled_a}
\end{align}
where $c\in[0,1)$ and $\varepsilon_k,\alpha_k,\beta_k,p_k,b_k\ge0$.
Assume that
\[
\sum_{k\ge0}\varepsilon_k<\infty,\quad
\sum_{k\ge0}\alpha_k<\infty,\quad
\sum_{k\ge0}p_k<\infty,\quad
\sum_{k\ge0}\beta_k<\infty,
\quad
\sup_{k\ge0} b_k <\infty.
\]
Then $\sup_{k\ge0} t_k<\infty$ and $\sup_{k\ge0} a_k<\infty$.
\end{lemma}

\begin{proof}
 By \eqref{eq:summable_coupled_a}, we have for all $k\ge1$,
\[
a_k
\le c^{\,k}a_0+\sum_{j=0}^{k-1}c^{\,k-1-j}\bigl(\beta_j t_j+b_j\bigr).
\]
Using $\sum_{i=0}^{k-1}c^{\,i}\le \frac{1}{1-c}$ and $\sup_{j\ge0}b_j<\infty$, we obtain
\begin{equation}\label{eq:a_bound_by_t_simple_final}
a_k
\le \frac{a_0}{1-c}
+\frac{1}{1-c}\sum_{j=0}^{k-1}\beta_j t_j
+\frac{\sup_{\ell\ge0}b_\ell}{1-c},
\qquad \forall k\ge0.
\end{equation}

 Substituting \eqref{eq:a_bound_by_t_simple_final} into \eqref{eq:summable_coupled_t} gives
\begin{equation}\label{eq:t_recursion_with_history_final}
t_{k+1}
\le (1+\varepsilon_k)t_k
+\frac{\alpha_k}{1-c}\sum_{j=0}^{k-1}\beta_j t_j
+\underbrace{\Big(\frac{\alpha_k}{1-c}\bigl(a_0+\sup_{\ell}b_\ell\bigr)+p_k\Big)}_{=:r_k},
\qquad k\ge0.
\end{equation}
Define  
\[
T_k:=\max_{0\le j\le k} t_j,\qquad k\ge0.
\]
Since $t_j\le T_k$ for all $0\le j\le k$, we have
\[
\sum_{j=0}^{k-1}\beta_j t_j
\le \Big(\sum_{j=0}^{k-1}\beta_j\Big)T_k.
\]
Hence, \eqref{eq:t_recursion_with_history_final} implies
\[
t_{k+1}
\le (1+\varepsilon_k)T_k
+\frac{\alpha_k}{1-c}\Big(\sum_{j=0}^{k-1}\beta_j\Big)T_k
+r_k
=(1+\delta_k)T_k+r_k,
\]
where
\[
\delta_k:=\varepsilon_k+\frac{\alpha_k}{1-c}\sum_{j=0}^{k-1}\beta_j\ge0.
\]
By definition, $T_{k+1}=\max\{T_k,t_{k+1}\}$, and since $\delta_k,r_k\ge0$, we obtain
\begin{equation}\label{eq:Tk_summable_recursion_final}
T_{k+1}\le (1+\delta_k)T_k+r_k,\qquad k\ge0.
\end{equation}

Moreover,
\[
\sum_{k\ge0}\delta_k
=\sum_{k\ge0}\varepsilon_k
+\frac{1}{1-c}\sum_{k\ge0}\alpha_k\sum_{j=0}^{k-1}\beta_j
\le \sum_{k\ge0}\varepsilon_k
+\frac{1}{1-c}\Big(\sum_{k\ge0}\alpha_k\Big)\Big(\sum_{j\ge0}\beta_j\Big)
<\infty,
\]
and $\sum_{k\ge0}r_k<\infty$ because $\sum_{k\ge0}\alpha_k<\infty$ and $\sum_{k\ge0}p_k<\infty$.
Then, from \eqref{eq:Tk_summable_recursion_final} we get
\[
T_k
\le \Big(\prod_{j=0}^{k-1}(1+\delta_j)\Big)T_0
+\sum_{i=0}^{k-1}\Big(\prod_{j=i+1}^{k-1}(1+\delta_j)\Big)r_i,
\qquad k\ge1.
\]
Using $\prod_{j=0}^{k-1}(1+\delta_j)\le \exp(\sum_{j=0}^{k-1}\delta_j)<\infty$, we conclude that
\[
T_k\le \exp\!\Big(\sum_{j=0}^{k-1}\delta_j\Big)\Big(T_0+\sum_{j=0}^{k-1}r_j\Big),
\qquad k\ge1,
\]
which implies $\sup_{k\ge0}T_k<\infty$, and hence $\sup_{k\ge0}t_k<\infty$.

Finally, the bound $\sup_{k\ge0}a_k<\infty$ follows from \eqref{eq:a_bound_by_t_simple_final},
$\sum_{k\ge0}\beta_k<\infty$, and $\sup_{k\ge0}t_k<\infty$. 
\end{proof}


\section{\texorpdfstring{Convergence of iAP in \cite{kruger2016regularity}}{Convergence of iAP in Kruger (2016)}}\label{sec:append-iap} 

We recall the inexact alternating projection framework of \cite{kruger2016regularity}.
Given a point $x$ and a closed set $A$ in $\E$, and parameters
$\tau\in(0,1]$ and $\sigma\in[0,1)$, the $(\tau,\sigma)$-projection of $x$ on $A$ is
\[
\Proj_A^{\tau,\sigma}(x)
:=\Big\{a\in A:\ \tau\|x-a\|\le d_A(x),\ d_{\N_x A}(x-a)\le \sigma\|x-a\|\Big\}.
\]
The $(\tau,\sigma)$-alternating projections $\{z_n\}$ for a pair of closed sets
$\{A,B\}$ iterates as follows
\[
z_{2n+1}\in \Proj^{\tau,\sigma}_B(x_{2n}),\qquad
z_{2n+2}\in \Proj^{\tau,\sigma}_A(x_{2n+1}),\qquad n=0,1,\dots,
\]
where $\Proj^{\tau,\sigma}_A(\cdot)$ is the $(\tau,\sigma)$-projection mapping
(see \cite[(27)]{kruger2016regularity}). In particular, $\Proj^{1,\sigma}_A(x)=\Proj_A(x)$
for any $\sigma\in[0,1)$.

\medskip
\paragraph{Theorem~31 in \cite{kruger2016regularity}.}
Assume that $\{A,B\}$ is intrinsic transverse at $\bar x$, and let
$0\le \sigma < \hat\theta_4[A,B](\bar x)$ and $0<\tau\le 1$, where $\hat\theta_4[A,B](\bar x)$ is some constant given by \cite[Def.18]{kruger2016regularity}.
Then, for any $\gamma<\hat\theta_4[A,B](\bar x)$ satisfying
$0<\gamma-\sigma\le \tau$ and
\[
c:=\tau^{-1}\bigl(1-\gamma^2+\gamma\sigma\bigr)<1,
\]
any sequence of $(\tau,\sigma)$-alternating projections for $\{A,B\}$ with initial
point sufficiently close to $\bar x$ converges to a point in $A\cap B$ with
$R$-linear rate $c$.

\medskip
\paragraph{A key step used in the proof.}
In the proof of Theorem~31, for $a\in A$ and $b\in \Proj^{\tau,\sigma}_B(a)$, one shows
\[
d_A(b)\le (1-\gamma^2+\gamma\sigma)\,\|b-a\|,
\]
and consequently for any $a'\in \Proj^{\tau,\sigma}_A(b)$,
\begin{equation}\label{eq:KT32}
\|a'-b\|\le \tau^{-1}d(b,A)\le c\,\|a-b\|.
\end{equation}

\medskip
\paragraph{Modification of parameters to ensure Fact~\ref{fact:alt_linear}.}
Let $A=\M_1$ and $B=\M_2$. Define   $x_k:=z_{2k}\in \M_1$
and $y_k:=z_{2k+1}\in \M_2$. Then \eqref{eq:KT32} yields
\[
d_{\M_2}(x_{k+1}) \le \|x_{k+1}-y_k\|
\le c\,\|x_k-y_k\|.
\]
Moreover, since $y_k\in \Proj^{\tau,\sigma}_{\M_2}(x_k)$, it follow from $\|x_k-y_k\|\leq \frac{1}{\tau}d_{\M_2}(x_k)$ that
\begin{equation}\label{eq:dist-contract}
d_{\M_2}(x_{k+1})
\le \frac{c}{\tau}\, d_{\M_2}(x_k)
\quad\text{with}\quad
\tau_K:=\frac{c}{\tau}=\tau^{-2}\bigl(1-\gamma^2+\gamma\sigma\bigr).
\end{equation}

Consequently, in addition to the original assumptions of Theorem~31,
if we impose the stronger requirement, we have
\begin{equation}\label{eq:tauK<1}
\tau_K=\tau^{-2}(1-\gamma^2+\gamma\sigma)<1
\quad\Longleftrightarrow\quad
\tau^2>1-\gamma^2+\gamma\sigma.
\end{equation}

\section{First-order property \eqref{eq:first-order-retraction} of APHL}
\label{app:APHL_first_order}

Let $\M_1,\M_2\subset\E$ be $C^2$ embedded submanifolds and let
\[
    \M=\M_1\cap\M_2,
    \qquad
    \M_2=\{z\in\E:H(z)=0\},
\]
where $H:\E\to\mathbb R^m$ is a $C^2$ local defining map with
$\rmD H$ surjective near $\bar x\in\M$. Suppose that the assumptions of
\cite[Assumptions~1.1--1.2]{xiao2025quadratically} hold near $\bar x$, with
$X=\M_1$ and $c=H$, and that the nondegeneracy condition
\[
    \rmD H(x)\operatorname{lin}(\T_{\M_1}(x))=\mathbb R^m
\]
holds for all $x\in\M$ near $\bar x$.

For $x\in\M$ near $\bar x$ and $\eta\in\T_x\M$ sufficiently small, set
\[
    z_0:=P_{\M_1}(x+\eta).
\]
Starting from $z_0$, let the APHL iterates be generated by
\[
    z_{k+1}\in \Pi_{\M_1}(A(z_k)+r_k),
    \qquad
    \|r_k\|\le \xi(z_k)\|H(z_k)\|^2,
\]
as in \cite[(1.6)--(1.7)]{xiao2025quadratically}. Let $z^*$ be the limit point
generated by this APHL sequence and define
\[
    \Phi_x^{\rm APHL}(\eta):=z^* .
\]

\begin{lemma}[First-order expansion of the APHL limiting map]
\label{lem:APHL_first_order_expansion}
Under the above assumptions, one has
\[
    \Phi_x^{\rm APHL}(\eta)
    =
    x+\eta+O(\|\eta\|^2).
\]
In particular, the APHL limiting point satisfies the first-order expansion
condition \eqref{eq:first-order-retraction}.
\end{lemma}

\begin{proof}
Since $\eta\in\T_x\M\subset\T_x\M_1$, the projection expansion onto the
$C^2$ manifold $\M_1$ gives
\[
    z_0=P_{\M_1}(x+\eta)=x+\eta+O(\|\eta\|^2).
\]
Moreover, since $\eta\in\T_x\M\subset\T_x\M_2$, the projection expansion onto
$\M_2$ gives
\[
    d_{\M_2}(x+\eta)=O(\|\eta\|^2).
\]
Combining this estimate with $z_0=x+\eta+O(\|\eta\|^2)$ yields
\[
    d_{\M_2}(z_0)=O(\|\eta\|^2).
\]

For $\|\eta\|$ sufficiently small, $z_0\in\M_1$ and $z_0$ belongs to the
neighborhood $\Xi_x$ in \cite[Proposition~3.8]{xiao2025quadratically}.
By \cite[Proposition~3.8]{xiao2025quadratically}, the APHL iterates remain in
the local neighborhood and satisfy
\[
    \sum_{k=0}^{N}\|z_{k+1}-z_k\|
    \le
    C_x\, d_{\M_2}(z_0),
    \qquad N\ge0,
\]
where $C_x>0$ depends on the local constants at $x$. By
\cite[Theorem~3.9]{xiao2025quadratically}, the sequence converges to a point
$z^*\in\M$. Letting $N\to\infty$ gives
\[
    \|z^*-z_0\|
    \le
    C_x\, d_{\M_2}(z_0)
    =
    O(\|\eta\|^2).
\]
Therefore,
\[
    \Phi_x^{\rm APHL}(\eta)
    =
    z^*
    =
    z_0+O(\|\eta\|^2)
    =
    x+\eta+O(\|\eta\|^2),
\]
which proves the claim.
\end{proof}

\section{Numerical experiments: supplementary materials}
\subsection{DNN relaxation and smooth reformulation}\label{app:DNN}

We consider the mixed-binary quadratic program (MBQP)
\begin{equation}
\min \left\{ x^\top Qx + 2c^\top x :
Ax=b,\ x_i\in\{0,1\}\ \forall i\in B,\ x_ix_j=0\ \forall (i,j)\in E,\ x\in\mathbb{R}_+^n \right\},
\label{eq:mbqp}
\end{equation}
where $Q\in\mathbb{S}^n$ is a symmetric matrix, $c\in\mathbb{R}^n$, $A\in\mathbb{R}^{m\times n}$, $b\in\mathbb{R}^m$,
$B\subseteq [n] := \{1,2,\ldots, n\}$ is the index set of binary variables, and
$E\subseteq \{(i,j)\mid 1\le i<j\le n\}$ is the index set of incompatible pairs.

Following \cite{hou2025low}, we consider the DNN relaxation and its smooth
reformulation with
\[
Q' :=
\begin{pmatrix}
Q & 0_{n\times 2m}\\
0_{2m\times n} & 0_{2m\times 2m}
\end{pmatrix},
\qquad
c' :=
\begin{pmatrix}
c\\
0_{2m}
\end{pmatrix},
\qquad
A' :=
\begin{pmatrix}
A & I_m & 0_{m\times m}\\
A & 0_{m\times m} & -I_m
\end{pmatrix},
\qquad
b' :=
\begin{pmatrix}
b\\
b
\end{pmatrix},
\]
and set
\[
N:=n+2m, \qquad s:=|B|.
\]
Then, one obtains the following DNN problem:
\begin{equation}\label{eq:dnn_relax_Qprime}
\begin{aligned}
\min_{x'\in\mathbb R^{N},\,X'\in\mathbb S^{N}}\quad
& \langle Q',X'\rangle + 2(c')^\top x' \\
\text{s.t.}\quad
& A'x' = b', \quad  A'X' = b'(x')^\top,\\
& x'_i = X'_{ii},\qquad \forall\, i\in B,\\
&
\begin{pmatrix}
1 & (x')^\top\\
x' & X'
\end{pmatrix}\in \mathbb S^{N+1}_+\cap \mathcal P,
\end{aligned}
\end{equation}
where $\mathcal P$ denotes certain polyhedral constraints (e.g., entrywise
nonnegativity and possible zero-pattern constraints, which will be not used in later discussion).
Let $e_1\in\mathbb{R}^r$ be the first canonical basis vector. Here and below, $e$ denotes the all-ones vector of compatible dimension. To further reduce the computational cost, we consider the low-rank factorization 
\[
Y'=
\begin{pmatrix}
1 & (x')^\top\\
x' & X'
\end{pmatrix}
=
\begin{pmatrix}
e_1^\top\\
R
\end{pmatrix}
\begin{pmatrix}
e_1 & R^\top
\end{pmatrix},
\qquad
R\in\mathbb{R}^{N\times r},
\]
where $r\ll n$ is a rank upper bound parameter. The RNNAL framework uses the rank-adaption technique to reduce it during update.
Then, the feasible set of the inner Riemannian subproblem of RNNAL is
\begin{equation}
\mathcal{M}_r
:=
\left\{
R\in\mathbb{R}^{N\times r}:
A'R=b'e_1^\top,\ \diag_B(RR^\top)=R_Be_1
\right\},
\label{eq:Mr_general}
\end{equation}
where $\diag_B(RR^\top)\in\mathbb{R}^{s}$ denotes the vector formed by the diagonal entries
of $RR^\top$ indexed by $B$, and $R_B\in\mathbb{R}^{s\times r}$ denotes the submatrix of $R$
consisting of the rows indexed by $B$.
We further split it as
\[
\mathcal{M}_{r}=\mathcal{M}_{1}\cap\mathcal{M}_{2},
\quad \mathcal{M}_{1}:=\{R\in\R^{N\times r}:A'R=b'e_1^\top\},
\quad
\mathcal{M}_{2}:=\{R\in\R^{N\times r}:\diag_{B}(RR^\top)=R_{B}e_1\},
\]

Accordingly, the low-rank subproblem in RNNAL takes the form
\begin{equation} 
\min_{R\in\mathcal{M}_r} f(R),
\label{eq:rie_sub_ALM}
\end{equation}
where $f:\mathbb{R}^{N\times r}\to\mathbb{R}$ denotes the smooth objective induced by the
current augmented Lagrangian subproblem. Solving \eqref{eq:rie_sub_ALM} is the key step. Let $\grad f(R)$ denote the Riemannian gradient on $\mathcal{M}_r$. Under the Euclidean metric, one has
\[
\grad f(R) = \Proj_{\T_R \M_r} \nabla f(R).
\]
Using the Sherman–Morrison–Woodbury formula, the  computational cost of  projection onto the tangent space   is 
\[
O\Big( \min\{ s^3+ m^2r+mrs, (mr)^2s+(mr)^3 \} \Big).
\]
In \cite{hou2025low}, at each Riemannian gradient step with stepsize $t>0$, given the trial point
\[
V = R - t\,\grad f(R),
\]
the retraction is computed through the metric projection
\begin{equation}
\Retr_R(-t\,\grad f(R))
=
\Proj_{\mathcal{M}_r}(V)
=
\arg\min_{X\in\mathcal{M}_r}\|X-V\|_F^2.
\label{eq:retr_proj}
\end{equation}

In the original RNNAL implementation, the nonconvex projection
\eqref{eq:retr_proj} is transformed into the convex problem
\begin{equation}
\min_{\Theta\in\mathbb{R}^{2m\times r}}
G(\Theta)
:=
\sum_{i\in B}\|(Y(\Theta))_{i,:}\|_2
+
\sum_{i\in [N]\setminus B}\|(Y(\Theta))_{i,:}\|_2^2
+
\langle \gamma e_1^\top,\Theta\rangle,
\label{eq:ret_sub}
\end{equation}
where
\[
Y(\Theta):=V'+A'^\top\Theta \in \R^{N\times r},
\qquad
V' := V-\frac12 ee_1^\top,
\qquad
\gamma := A'e-2b'\in \R^{2m}.
\]
 By
\cite[Theorem~5]{hou2025low}, if $\Theta^\star$ is an optimal solution of
\eqref{eq:ret_sub}, then
\[
\Proj_{\mathcal{M}_r}(V)
=
\Proj_{\mathcal{M}_2}(V+A'^\top\Theta^\star).
\]
The problem \eqref{eq:ret_sub} can  be solved by the generalized Weiszfeld algorithm (GWA) or the GWA-Newton method to accelerate the local convergence.
Next, we compare the orthogonal-projection approach in \cite{hou2025low} with the
alternating-projection-type retractions studied in this paper.  Throughout the experiments,
we keep the outer ALM and the rank-adaptive Riemannian gradient framework unchanged, and
vary only the retraction for the  subproblem \eqref{eq:rie_sub_ALM}.  

\subsection{Implementation details for DNN problem}\label{app:complexity}
\paragraph{ GWA/GWA-Newton in \cite{hou2025low}.}
Let
\[
Y^k:=Y(\Theta^k)=V'+A'^\top\Theta^k.
\]
Define the weight vector $v^k\in\mathbb{R}^N$ by
\begin{equation} 
v_i^k=
\begin{cases}
\|Y^k_{i,:}\|_2^{-1}, & i\in B,\\[3pt]
2, & i\in [N]\setminus B.
\end{cases}
\label{eq:gwa_weights}
\end{equation}
Then the generalized Weiszfeld algorithm updates as
\begin{equation}
\Theta^{k+1}
=
-\Bigl(A'\Diag(v^k)A'^\top\Bigr)^{-1}
\Bigl(\gamma e_1^\top + A'\Diag(v^k)V'\Bigr).
\label{eq:gwa}
\end{equation}

Let $A'_B\in\mathbb{R}^{m\times s}$ denote the submatrix of $A'$ formed by the columns
indexed by $B$. Since only the $s$ entries in $v_B^k$ vary, one has
\begin{equation}
A'\Diag(v^k)A'^\top
=
2A'A'^\top + A'_B\Diag(v_B^k-2e)A_B'{}^\top .
\label{eq:gwa_matrix}
\end{equation}
Hence, after precomputing $A'V'$ and $2A'A'^\top$, one GWA iteration requires:
\begin{itemize}
\item forming $Y^k=V'+A'^\top\Theta^k$, costing $O(mNr)$;
\item evaluating $v_B^k$, costing $O(sr)$;
\item forming the right-hand side in \eqref{eq:gwa}, costing $O(msr)$;
\item solving the linear system with coefficient matrix \eqref{eq:gwa_matrix}.
\end{itemize}
For the last step, there are two practical implementations:

(i) \emph{Direct solve.} Form \eqref{eq:gwa_matrix} explicitly and factorize the resulting
$(2m)\times (2m)$ matrix. This costs
\[
O(m^2s+m^3+m^2r).
\]

(ii) \emph{Sherman-Morrison-Woodbury.} View \eqref{eq:gwa_matrix} as a diagonal-rank-$s$ update of
the precomputed matrix $2A'A'^\top$, and apply the Sherman--Morrison--Woodbury formula.
Then, the cost is
\[
O(s^3+s^2r+msr+m^2r).
\]

Therefore, the total cost per GWA iteration is estimated by
\begin{equation}
O\!\left(
mnr+m^2r+msr+
\min\Bigl\{
m^2s+m^3,\;
s^3+s^2r
\Bigr\}
\right).
\label{eq:gwa_cost_new}
\end{equation}

For local acceleration, we may switch to a Newton step as suggested in
\cite[Remark~8]{hou2025low}. Define the Hessian operator
$P_k:\mathbb{R}^{N\times r}\to\mathbb{R}^{N\times r}$ by
\begin{equation}
\bigl(P_k[H]\bigr)_{i,:}
=
\begin{cases}
\displaystyle
H_{i,:}
\left(
\frac{I_r}{\|Y^k_{i,:}\|_2}
-
\frac{(Y^k_{i,:})^\top Y^k_{i,:}}{\|Y^k_{i,:}\|_2^3}
\right), & i\in B,\\[10pt]
2H_{i,:}, & i\in [N]\setminus B,
\end{cases}
\qquad
H\in\mathbb{R}^{N\times r}.
\label{eq:newton_operator}
\end{equation}
Then the Newton direction $\Delta\Theta^k\in\mathbb{R}^{m\times r}$ is defined by
\begin{equation}
\bigl(A'P_kA'^\top\bigr)[\Delta\Theta^k]
=
\gamma e_1^\top + A'\bigl(\Diag(v^k)Y^k\bigr),
\label{eq:newton_system}
\end{equation}
and the update is
\begin{equation}
\Theta^{k+1}=\Theta^k-\Delta\Theta^k.
\label{eq:newton_baseline}
\end{equation}

As in the projection step of Section~5.3 in \cite{hou2025low}, the active rows indexed by
$B$ contribute the only nontrivial row blocks, while the remaining rows contribute constant
blocks. Therefore, one Newton step admits two practical implementations:\\
(i) \emph{Direct vectorized solve.} Vectorize \eqref{eq:newton_system} into an
$(mr)\times(mr)$ linear system. Since only the $s$ active rows contribute nonconstant
Hessian blocks, the matrix assembly and factorization cost
\[
O((mr)^2s+(mr)^3).
\]
(ii) \emph{Sherman-Morrison-Woodbury.} Eliminate the affine part first and solve the
resulting dense system of size $s\times s$, then recover the $m\times r$ direction. This
costs
\[
O(s^3+m^2r + s^2r+msr).
\]
Including the cost of forming $Y^k$ and the right-hand side, the total cost per
GWA-Newton step is
\begin{equation}
O\!\left(
mnr+
\min\Bigl\{
s^3+ s^2r+m^2r+mrs,\;
(mr)^2s+(mr)^3
\Bigr\}
\right).
\label{eq:newton_cost_new}
\end{equation}
In practice, we use the cheaper of the two implementations in
\eqref{eq:gwa_cost_new} and \eqref{eq:newton_cost_new}. This is almost the same as the projection onto the tangent space.

Following \cite{hou2025low}, we consider the reformulated lifted problem with
\[
N:=n+4p,\qquad m:=4p,
\]
and
\[
Q' :=
\begin{pmatrix}
Q & 0_{n\times 4p}\\
0_{4p\times n} & 0_{4p\times 4p}
\end{pmatrix},
\qquad
A' :=
\begin{pmatrix}
A & I_{2p} & 0_{2p\times 2p}\\
A & 0_{2p\times 2p} & -I_{2p}
\end{pmatrix}\in\R^{m\times N}.
\]
Let   $[n]:=\{1,\dots,n\}$ and   $e\in\mathbb{R}^{m}$ denote the all-ones vector. Then the manifold in the inner Riemannian subproblem of RNNAL is given by 
\begin{equation}
\mathcal{M}_{r}
:=
\left\{
R\in\mathbb{R}^{N\times r}:
A'R = e e_1^\top,\ 
\diag_{[n]}(RR^\top)=R_{[n]} e_1
\right\},
\label{eq:Mr}
\end{equation}
where $\diag_{[n]}(RR^\top)$ denotes the vector formed by the first $n$ diagonal entries of
$RR^\top$, and $R_{[n]}$ denotes the first $n$ rows of $R$.

\paragraph{APM.}
One has
\begin{equation}
R^{k+\frac12} = \Proj_{\mathcal{M}_2}(R^k),
\qquad
R^{k+1} = \Proj_{\mathcal{M}_1}(R^{k+\frac12}).
\label{eq:apm}
\end{equation}
 For any $R\in\mathbb R^{N\times r}$, it follows that
\begin{equation}
\Proj_{\mathcal M_1}(R)
=
R-A'^\top(A'A'^\top)^{-1}(A'R-b'e_1^\top),
\label{eq:proj_M2}
\end{equation}
and
\begin{equation}
\bigl(\Proj_{\mathcal M_2}(R)\bigr)_{i,:}
=
\begin{cases}
\displaystyle
\frac12\left(
v_i(R)\,(2R_{i,:}-e_1^\top)+e_1^\top
\right), & i\in B,\\[8pt]
R_{i,:}, & i\in [N]\setminus B,
\end{cases}
\label{eq:proj_M1}
\end{equation}
where $v$ is given by
\[
v_i(R):=\|2R_{i,:}-e_1^\top\|_2^{-1},\qquad i\in B.
\]
Note that $(A'A'^\top)^{-1}$ can be precomputed.
The per-iteration cost is given by
\[
O(mnr+m^2r+msr).
\]

\paragraph{NewtonSLRA.}
NewtonSLRA first projects onto $\mathcal{M}_2$, and then performs a Newton-type
correction toward $\mathcal{M}_1$ along the affine approximation of $\mathcal{M}_2$.
Starting from the trial point $R^0=V$, for each inner iteration we define
\[
\widetilde R_k:=\Proj_{\mathcal M_2}(R^k)\in\mathbb R^{N\times r}.
\]
Then $R^{k+1}$ is obtained by solving
\begin{equation}
R^{k+1}
=
\arg\min_{R}\frac12\|R-R^k\|_F^2
\quad\text{s.t.}\quad
R\in \mathcal M_1\cap\bigl(\widetilde R_k+T_{\widetilde R_k}\mathcal M_2\bigr).
\label{eq:newtonslra_main}
\end{equation}
Writing $R=R^k+\Delta$, the above problem becomes
\begin{equation}
\min_{\Delta}\ \frac12\|\Delta\|_F^2
\quad\text{s.t.}\quad
A'\Delta=0,
\qquad
c_i^k(\Delta_{i,:})^\top=-g_i^k,\quad i\in B,
\label{eq:newtonslra_sub}
\end{equation}
where
\[
c_i^k:=2\widetilde R_{k,i,:}-e_1^\top,
\qquad
g_i^k:=\bigl\langle R^k_{i,:}-\widetilde R_{k,i,:},\,c_i^k\bigr\rangle,
\qquad i\in B.
\]
Let $C_k\in\mathbb R^{s\times r}$ be the matrix whose rows are $c_i^k$ for $i\in B$, and
let $g^k\in\mathbb R^s$ be the vector whose entries are $g_i^k$. Moreover, recall that  
$A'_B\in\mathbb R^{m\times s}$ is the submatrix of $A'$ formed by the columns indexed by $B$,
and define
\[
S:=A_B'{}^\top(A'A'^\top)^{-1}A'_B\in\mathbb R^{s\times s}.
\]
Then the KKT system of \eqref{eq:newtonslra_sub} reduces to the Schur complement system
\begin{equation}
\Bigl(I_s-(C_kC_k^\top)\odot S\Bigr)\mu^k=g^k,
\label{eq:newtonslra_schur}
\end{equation}
where $\mu^k\in\mathbb R^s$ and $\odot$ denotes the Hadamard product. After solving
\eqref{eq:newtonslra_schur}, we recover
\begin{equation}
\Lambda^k
=
-(A'A'^\top)^{-1}A'_B\bigl(\Diag(\mu^k)C_k\bigr),
\label{eq:newtonslra_lambda}
\end{equation}
and
\begin{equation}
\Delta^k=-A'^\top\Lambda^k-T_B^\ast(\mu^k),
\label{eq:newtonslra_delta}
\end{equation}
where $T_B^\ast(\mu^k)\in\mathbb R^{N\times r}$ is the matrix whose rows indexed by $B$
are $\mu_i^k c_i^k$ and whose remaining rows are zero. The next iterate is then
\[
R^{k+1}=R^k+\Delta^k.
\]

For the linear system \eqref{eq:newtonslra_schur}, there are two practical
implementations:\\
(i) \emph{Direct  solve.}
If the   matrix in \eqref{eq:newtonslra_schur} is formed explicitly, then one
iteration costs
\[
O\!\left(mnr+m^2r+msr+s^2r+s^3\right).
\]
(ii) \emph{Sherman--Morrison--Woodbury.}
Since
$S=A_B'{}^\top(A'A'^\top)^{-1}A'_B$
has rank at most $m$, write
\[
S=UU^\top,\qquad U\in\mathbb R^{s\times m}.
\]
Then
\[
(C_kC_k^\top)\odot S=W_kW_k^\top,
\qquad
W_k:=
\bigl[\Diag(C_k(:,1))U,\dots,\Diag(C_k(:,r))U\bigr]
\in\mathbb R^{s\times mr}.
\]
Hence \eqref{eq:newtonslra_schur} becomes
\[
(I_s-W_kW_k^\top)\mu^k=g^k,
\]
and by the Sherman--Morrison--Woodbury formula,
\begin{equation}
\mu^k
=
g^k+W_k\bigl(I_{mr}-W_k^\top W_k\bigr)^{-1}W_k^\top g^k.
\label{eq:newtonslra_smw}
\end{equation}
Therefore, the cost per iteration is
\[
O\!\left(mnr+m^2r^2 s+m^3r^3\right).
\]

In practice, we use the cheaper of the above two implementations, and thus the
per-iteration cost of NewtonSLRA is
\[
O\!\left(
mnr+
\min\bigl\{\,msr+m^2r+s^2r+s^3,\ (mr)^2s+(mr)^3\,\bigr\}
\right),
\]
which is the same as GWA-Newton. 

Regarding the relaxed NewtonSLRA algorithm, the linear system in \eqref{eq:newtonslra_schur} has size $1\times 1$ since the constraint $c_i^k(\Delta_{i,:})^\top=-g_i^k,\quad i\in B$ in \eqref{eq:newtonslra_sub} is replaced by  the single relaxed-tangent constraint
\[
\langle R^k-\widetilde R_k,\ R-\widetilde R_k\rangle_F = 0,
\qquad \widetilde R_k:=\Proj_{\M_2}(R^k).
\]
 Therefore, the per-iteration cost of relaxed NewtonSLRA is
$$O\!\left(
mnr+msr+m^2r
\right).$$
\paragraph{APHL.}
 Starting from
\[
R^0=\Proj_{\mathcal M_2}(V)\in\mathbb R^{N\times r},
\]
for each inner iteration define
\[
E_k:=A'R^k-b'e_1^\top\in\mathbb R^{2m\times r}.
\]

Since $\mathcal M_2$ is a smooth embedded submanifold, we take $Q(R)=\Proj_{\T_{R}\mathcal M_2}$.
Then $\ker(Q(R))=\mathrm N_R\mathcal M_2$, and the second-order requirement follows from the
standard second-order approximation property of smooth manifolds.
For $i\in B$, let
\[
c_i^k:=2R^k_{i,:}-e_1^\top,
\]
and let $C_k\in\mathbb R^{s\times r}$ be the matrix whose rows are $c_i^k$, $i\in B$.
Then $Q_k$ acts row-wise as
\begin{equation}
\bigl(Q_k[H]\bigr)_{i,:}
=
\begin{cases}
H_{i,:}-\langle H_{i,:},c_i^k\rangle\,c_i^k, & i\in B,\\[3pt]
H_{i,:}, & i\in [N]\setminus B,
\end{cases}
\qquad H\in\mathbb R^{N\times r}.
\label{eq:aphl_Qk}
\end{equation}
We compute $\Lambda^k\in\mathbb R^{2m\times r}$ from
\begin{equation}
A'\!\left(Q_k[A'^\top\Lambda^k]\right)=E_k,
\label{eq:aphl_linear_system}
\end{equation}
and then update
\begin{equation}
\widetilde R_{k+1}
=
R^k-Q_k[A'^\top\Lambda^k],
\qquad
R^{k+1}=\Proj_{\mathcal M_1}(\widetilde R_{k+1}).
\label{eq:aphl_update}
\end{equation}

Let $A'_B\in\mathbb R^{2m\times s}$ be the submatrix of $A'$ formed by the columns
indexed by $B$, and define
\[
S:=A_B'{}^\top(A'A'^\top)^{-1}A'_B\in\mathbb R^{s\times s}.
\]
Moreover, define $h^k\in\mathbb R^s$ by
\[
h_i^k
:=
\left\langle
\bigl(A_B'{}^\top(A'A'^\top)^{-1}E_k\bigr)_{i,:},\,c_i^k
\right\rangle,
\qquad i\in B.
\]
Then \eqref{eq:aphl_linear_system} reduces to the Schur complement system
\begin{equation}
\Bigl(I_s-(C_kC_k^\top)\odot S\Bigr)\mu^k=h^k,
\label{eq:aphl_schur}
\end{equation}
where $\mu^k\in\mathbb R^s$ and $\odot$ denotes the Hadamard product. After solving
\eqref{eq:aphl_schur}, we recover
\begin{equation}
\Lambda^k
=
(A'A'^\top)^{-1}\Bigl(E_k+A'_B\bigl(\Diag(\mu^k)C_k\bigr)\Bigr),
\label{eq:aphl_lambda}
\end{equation}
and
\begin{equation}
Q_k[A'^\top\Lambda^k]
=
A'^\top\Lambda^k-T_B^\ast(\mu^k),
\label{eq:aphl_Qrecover}
\end{equation}
where $T_B^\ast(\mu^k)\in\mathbb R^{N\times r}$ is the matrix whose rows indexed
by $B$ are $\mu_i^k c_i^k$ and whose remaining rows are zero. Hence
\[
\widetilde R_{k+1}
=
R^k-A'^\top\Lambda^k+T_B^\ast(\mu^k).
\]

Similar to NewtonSLRA, for the linear system \eqref{eq:aphl_schur}, there are two practical implementations:\\
(i)\emph{Direct Schur solve.}
If the Schur matrix is formed explicitly, then one iteration costs
\[
O\!\left(m(n+m)r+msr+s^2r+s^3\right).
\]
(ii)\emph{Sherman--Morrison--Woodbury.}
 The corresponding cost is
\[
O\!\left(m(n+m)r+msr+sm^2r^2+m^3r^3\right).
\]

Therefore, the per-iteration cost of APHL is
\[
O\!\left(
mnr+
\min\bigl\{\,msr+s^2r+m^2r+s^3,\ s(mr)^2+(mr)^3\,\bigr\}
\right),
\]
which is the same as NewtonSLRA and GWA-Newton.

 \subsection{Numerical results for QAP problem}
  {\scriptsize
\begin{longtable}{lrrrrrrrrr}
\caption{QAPLib instances: results of different retraction solvers embedded into the same RNNAL outer framework.
Here $pobj$ is the final primal objective value, $r$ is the final rank, $(R_p,R_d,R_c)$ are the relative KKT residuals defined in \cite{hou2025low}, and $\overline{\mathrm{Rite}}$ is the average number of inner iterations used by the retraction subsolver.  Best values in each column \emph{except} $pobj$ are highlighted in bold within each instance.}
\label{tab:qap_all}\\
\toprule
Method & $pobj$ & $r$ & $R_p$ & $R_d$ & $R_c$ & Time(s) & ALMite & BBite & $\overline{\mathrm{Rite}}$\\
\midrule
\endfirsthead

\toprule
Method & $pobj$ & $r$ & $R_p$ & $R_d$ & $R_c$ & Time(s) & ALMite & BBite & $\overline{\mathrm{Rite}}$\\
\midrule
\endhead

\midrule
\multicolumn{10}{r}{\small Continued on next page.}\\
\endfoot

\bottomrule
\endlastfoot

\midrule
\multicolumn{10}{l}{\textbf{chr12a} ($n=144, p=12$)}\\
\midrule
GWA & 9.5524e+03 & 13 & 7.04e-07 & 6.24e-07 & 1.29e-07 & 1.90e+01 & 28 & 5321 & 3.31 \\
GWA-Newton & 9.5521e+03 & 16 & 2.99e-07 & 2.14e-07 & 2.11e-08 & \textbf{1.72e+01} & 29 & \textbf{3760} & \textbf{1.20} \\
APM & 9.5521e+03 & 49 & \textbf{2.43e-07} & 1.06e-07 & 8.41e-08 & 2.26e+01 & 28 & 5085 & 4.04 \\
NewtonSLRA & 9.5520e+03 & \textbf{2} & 5.02e-07 & 4.87e-08 & 3.19e-09 & 2.02e+01 & 28 & 5043 & 1.29 \\
APHL & 9.5522e+03 & 54 & 7.73e-07 & 3.12e-07 & 8.01e-08 & 2.08e+01 & \textbf{26} & 4980 & 1.32 \\
TAPR & 9.5520e+03 & 14 & 4.18e-07 & \textbf{1.35e-09} & \textbf{8.68e-12} & 2.17e+01 & 29 & 5168 & 2.33 \\
Relaxed NewtonSLRA & 9.5520e+03 & 61 & 5.65e-07 & 6.19e-08 & 7.01e-08 & 2.43e+01 & 28 & 4706 & 4.79 \\

\midrule
\multicolumn{10}{l}{\textbf{chr12b} ($n=144, p=12$)}\\
\midrule
GWA & 9.7422e+03 & 19 & 9.32e-07 & 2.45e-07 & 3.81e-07 & \textbf{8.53e+00} & 28 & \textbf{2315} & 2.67 \\
GWA-Newton & 9.7425e+03 & 11 & 8.98e-07 & 6.71e-07 & 1.90e-07 & 1.08e+01 & 28 & 2368 & \textbf{1.18} \\
APM & 9.7420e+03 & 15 & 4.21e-07 & 4.35e-10 & 2.17e-09 & 1.08e+01 & 26 & 2596 & 3.45 \\
NewtonSLRA & 9.7420e+03 & 25 & 7.11e-07 & 4.71e-10 & 4.22e-11 & 1.10e+01 & 25 & 2750 & 1.31 \\
APHL & 9.7420e+03 & \textbf{8} & \textbf{1.40e-07} & 2.58e-10 & \textbf{3.65e-11} & 1.03e+01 & 28 & 2692 & 1.25 \\
TAPR & 9.7420e+03 & 23 & 9.70e-07 & 9.35e-10 & 1.07e-10 & 1.08e+01 & \textbf{24} & 2486 & 2.29 \\
Relaxed NewtonSLRA & 9.7420e+03 & 19 & 6.96e-07 & \textbf{9.58e-17} & 5.54e-08 & 1.30e+01 & 47 & 2525 & 4.00 \\

\midrule
\multicolumn{10}{l}{\textbf{chr12c} ($n=144, p=12$)}\\
\midrule
GWA & 1.1156e+04 & \textbf{3} & \textbf{6.80e-08} & 5.50e-10 & 1.51e-10 & \textbf{2.91e+01} & 36 & 11119 & 2.51 \\
GWA-Newton & 1.1156e+04 & \textbf{3} & 1.92e-07 & 2.40e-09 & \textbf{2.51e-11} & 3.95e+01 & 35 & 9877 & \textbf{1.16} \\
APM & 1.1156e+04 & \textbf{3} & 7.63e-07 & 1.93e-09 & 1.00e-08 & 3.26e+01 & \textbf{32} & 11446 & 2.92 \\
NewtonSLRA & 1.1156e+04 & \textbf{3} & 8.41e-07 & 1.60e-09 & 7.15e-10 & 3.13e+01 & 34 & \textbf{9721} & 1.23 \\
APHL & 1.1156e+04 & 8 & 7.71e-07 & 1.46e-08 & 4.23e-09 & 3.78e+01 & 49 & 11362 & 1.24 \\
TAPR & 1.1156e+04 & 4 & 2.92e-07 & 5.51e-10 & 6.90e-11 & 3.53e+01 & 35 & 10279 & 2.22 \\
Relaxed NewtonSLRA & 1.1156e+04 & \textbf{3} & 7.63e-07 & \textbf{2.76e-10} & 1.21e-09 & 4.12e+01 & 35 & 12030 & 2.79 \\

\midrule
\multicolumn{10}{l}{\textbf{chr20a} ($n=400, p=20$)}\\
\midrule
GWA & 2.1921e+03 & 7 & 8.77e-07 & 5.95e-07 & 1.03e-07 & 6.64e+02 & \textbf{43} & 29765 & 2.43 \\
GWA-Newton & 2.1921e+03 & 9 & 5.69e-07 & 3.03e-07 & 2.64e-07 & 8.57e+02 & 46 & 26658 & \textbf{1.11} \\
APM & 2.1920e+03 & 45 & 2.44e-07 & 1.94e-08 & \textbf{5.45e-08} & \textbf{3.45e+02} & 66 & \textbf{11329} & 2.64 \\
NewtonSLRA & 2.1920e+03 & \textbf{4} & \textbf{2.02e-07} & 1.16e-08 & 9.25e-08 & 9.30e+02 & 50 & 34444 & 1.16 \\
APHL & 2.1920e+03 & 8 & 7.16e-07 & 1.28e-07 & 6.37e-07 & 7.07e+02 & 45 & 25408 & 1.21 \\
TAPR & 2.1920e+03 & 116 & 2.42e-07 & \textbf{8.34e-09} & 4.49e-07 & 4.11e+02 & 115 & 13485 & 2.27 \\
Relaxed NewtonSLRA & 2.1920e+03 & 189 & 7.40e-07 & 2.95e-08 & 6.59e-08 & 5.83e+02 & 135 & 16716 & 5.43 \\

\midrule
\multicolumn{10}{l}{\textbf{chr20b} ($n=400, p=20$)}\\
\midrule
GWA & 2.2980e+03 & 231 & \textbf{1.28e-07} & 4.15e-11 & 1.68e-10 & 3.17e+02 & 36 & 10582 & 4.24 \\
GWA-Newton & 2.2980e+03 & 233 & 2.17e-07 & 4.24e-11 & 3.46e-11 & 4.03e+02 & 36 & 11404 & \textbf{1.23} \\
APM & 2.2980e+03 & 180 & 8.63e-07 & 3.34e-10 & 5.36e-09 & \textbf{1.55e+02} & 40 & \textbf{4422} & 4.64 \\
NewtonSLRA & 2.2980e+03 & 232 & 6.51e-07 & \textbf{3.12e-11} & \textbf{9.32e-12} & 3.47e+02 & 37 & 10599 & 1.31 \\
APHL & 2.2980e+03 & 239 & 6.08e-07 & 1.36e-10 & 4.24e-09 & 3.66e+02 & \textbf{35} & 11815 & 1.36 \\
TAPR & 2.2980e+03 & 161 & 3.09e-07 & 3.03e-10 & 4.43e-10 & 2.61e+02 & 107 & 8203 & 2.33 \\
Relaxed NewtonSLRA & 2.2980e+03 & \textbf{150} & 6.94e-07 & 9.66e-10 & 4.01e-09 & 2.81e+02 & 99 & 7641 & 5.43 \\

\midrule
\multicolumn{10}{l}{\textbf{had12} ($n=144, p=12$)}\\
\midrule
GWA & 1.6520e+03 & \textbf{21} & 8.55e-07 & 3.24e-09 & 7.46e-09 & 1.08e+03 & 2858 & 73647 & 3.02 \\
GWA-Newton & 1.6520e+03 & 46 & 8.75e-07 & 5.03e-08 & 4.79e-07 & 1.48e+01 & 56 & \textbf{3212} & \textbf{1.06} \\
APM & 1.6520e+03 & 22 & \textbf{5.65e-07} & \textbf{4.18e-10} & 6.02e-07 & 6.38e+02 & 2794 & 69133 & 3.53 \\
NewtonSLRA & 1.6520e+03 & 46 & 7.10e-07 & 3.75e-08 & 2.85e-07 & 1.44e+01 & 35 & 3966 & 1.09 \\
APHL & 1.6520e+03 & 44 & 7.91e-07 & 8.76e-09 & \textbf{2.76e-09} & \textbf{1.27e+01} & 37 & 3335 & 1.11 \\
TAPR & 1.6520e+03 & 43 & 8.38e-07 & 8.34e-10 & 6.42e-08 & 1.51e+01 & \textbf{32} & 3660 & 2.07 \\
Relaxed NewtonSLRA & 1.6520e+03 & 43 & 6.60e-07 & 3.70e-09 & 1.22e-08 & 2.02e+01 & 35 & 5350 & 1.56 \\

\midrule
\multicolumn{10}{l}{\textbf{nug12} ($n=144, p=12$)}\\
\midrule
GWA & 5.6799e+02 & 88 & 8.86e-07 & \textbf{7.89e-07} & 3.79e-08 & 8.83e+01 & 73 & 23313 & 1.02 \\
GWA-Newton & 5.6799e+02 & \textbf{87} & 9.61e-07 & 9.97e-07 & 4.27e-08 & 8.22e+01 & \textbf{72} & \textbf{16281} & \textbf{1.01} \\
APM & 5.6799e+02 & 88 & \textbf{8.21e-07} & 9.71e-07 & 3.91e-09 & 9.88e+01 & 75 & 26117 & 1.05 \\
NewtonSLRA & 5.6799e+02 & \textbf{87} & 8.49e-07 & 9.17e-07 & 7.40e-09 & 7.30e+01 & 76 & 16923 & 1.02 \\
APHL & 5.6799e+02 & \textbf{87} & 8.60e-07 & 9.38e-07 & \textbf{2.83e-09} & 7.41e+01 & 76 & 17517 & 1.02 \\
TAPR & 5.6799e+02 & \textbf{87} & 9.06e-07 & 9.63e-07 & 1.16e-08 & 8.67e+01 & 74 & 17629 & 2.01 \\
Relaxed NewtonSLRA & 5.6799e+02 & \textbf{87} & 8.30e-07 & 9.83e-07 & 1.90e-08 & \textbf{7.19e+01} & 77 & 17216 & 1.08 \\

\midrule
\multicolumn{10}{l}{\textbf{nug15} ($n=225, p=15$)}\\
\midrule
GWA & 1.1406e+03 & 143 & \textbf{7.19e-07} & 8.72e-07 & \textbf{2.83e-09} & 2.33e+02 & 79 & 32528 & 1.02 \\
GWA-Newton & 1.1406e+03 & \textbf{142} & 7.81e-07 & 9.33e-07 & 2.17e-08 & 2.33e+02 & 77 & 23575 & \textbf{1.01} \\
APM & 1.1406e+03 & \textbf{142} & 7.50e-07 & 9.38e-07 & 4.50e-09 & 2.44e+02 & 78 & 32595 & 1.06 \\
NewtonSLRA & 1.1406e+03 & \textbf{142} & 8.68e-07 & \textbf{8.49e-07} & 2.67e-08 & 1.88e+02 & \textbf{74} & 22297 & \textbf{1.01} \\
APHL & 1.1406e+03 & 143 & 9.40e-07 & 8.75e-07 & 1.60e-08 & \textbf{1.81e+02} & \textbf{74} & \textbf{22290} & 1.02 \\
TAPR & 1.1406e+03 & \textbf{142} & 8.34e-07 & 8.58e-07 & 5.82e-09 & 2.34e+02 & \textbf{74} & 23942 & 2.01 \\
Relaxed NewtonSLRA & 1.1406e+03 & 143 & 8.00e-07 & 9.56e-07 & 2.27e-08 & 2.00e+02 & 78 & 23872 & 1.07 \\

\midrule
\multicolumn{10}{l}{\textbf{tai15a} ($n=225, p=15$)}\\
\midrule
GWA & 3.7710e+05 & 126 & 9.70e-07 & 5.69e-07 & 3.12e-07 & 6.68e+01 & 50 & 10286 & 1.04 \\
GWA-Newton & 3.7710e+05 & \textbf{125} & \textbf{8.60e-07} & \textbf{2.95e-07} & \textbf{7.49e-10} & 6.48e+01 & \textbf{49} & \textbf{6669} & \textbf{1.02} \\
APM & 3.7710e+05 & \textbf{125} & 9.74e-07 & 5.32e-07 & 1.83e-09 & 6.34e+01 & 57 & 8905 & 1.18 \\
NewtonSLRA & 3.7710e+05 & \textbf{125} & 9.79e-07 & 5.64e-07 & 5.51e-08 & \textbf{5.72e+01} & \textbf{49} & 7095 & 1.03 \\
APHL & 3.7710e+05 & \textbf{125} & 9.10e-07 & 4.40e-07 & 2.06e-09 & 7.32e+01 & 52 & 9518 & 1.03 \\
TAPR & 3.7710e+05 & \textbf{125} & 9.46e-07 & 3.32e-07 & 6.53e-09 & 7.16e+01 & 51 & 7670 & 2.04 \\
Relaxed NewtonSLRA & 3.7710e+05 & \textbf{125} & 9.80e-07 & 4.38e-07 & 6.88e-08 & 6.91e+01 & 51 & 9183 & 1.12 \\

\midrule
\multicolumn{10}{l}{\textbf{tai15b} ($n=225, p=15$)}\\
\midrule
GWA & 5.1782e+07 & 134 & 9.47e-07 & 6.99e-07 & 3.39e-07 & \textbf{4.40e+02} & 83 & \textbf{64380} & 1.21 \\
GWA-Newton & 5.1777e+07 & 130 & 8.14e-07 & 4.81e-07 & \textbf{3.14e-07} & 6.38e+02 & 83 & 67200 & \textbf{1.04} \\
APM & 5.1780e+07 & 133 & 9.79e-07 & 6.34e-07 & 3.41e-07 & 6.46e+02 & 87 & 89417 & 1.44 \\
NewtonSLRA & 5.1775e+07 & \textbf{129} & 7.84e-07 & 4.30e-07 & 4.03e-07 & 7.40e+02 & 88 & 89389 & 1.08 \\
APHL & 5.1784e+07 & 133 & 9.79e-07 & 7.23e-07 & 9.83e-07 & 6.43e+02 & \textbf{78} & 79776 & 1.10 \\
TAPR & 5.1774e+07 & 130 & \textbf{7.44e-07} & \textbf{3.75e-07} & 3.89e-07 & 8.78e+02 & 94 & 93555 & 2.07 \\
Relaxed NewtonSLRA & 5.1787e+07 & 139 & 8.12e-07 & 9.43e-07 & 3.93e-07 & 7.70e+02 & 84 & 94670 & 1.41 \\
\end{longtable}
}
\subsection{Quadratic knapsack problem}

We also test the binary quadratic knapsack problem (QKP) considered in
\cite{hou2025low}. Given a profit matrix $Q\in\mathbb S^n$, a weight vector
$a\in\mathbb R_{++}^n$ and a capacity $\tau>0$, QKP is
\begin{equation}\label{eq:qkp}
\max\left\{x^\top Qx:\ a^\top x\le \tau,\ x\in\{0,1\}^n\right\}.
\end{equation}
Following \cite{hou2025low}, we convert the inequality to an equality constraint and
apply the DNN relaxation within the same RNNAL framework. In the MBQP notation \eqref{eq:mbqp}, QKP has a single linear constraint, i.e.,
\[
A=a^\top\in\mathbb R^{1\times n},\qquad b=\tau,\qquad B=[n],\qquad s=|B|=n,\qquad m=1.
\]
We generate QKP instances by sampling $Q_{ij}=Q_{ji}$ uniformly from $\{1,\dots,100\}$ with probability $p$ (and $0$ otherwise), sampling $a_i$ uniformly from $\{1,\dots,50\}$, and setting the capacity $\tau=0.9\,e^\top a$.
Under the constraint-relaxation reformulation in \eqref{eq:dnn_relax_Qprime}, the extended
dimension is
\[
N=n+2m=n+2,
\]
and the corresponding reformulated constraint matrix and right-hand side are
\[
A'=
\begin{pmatrix}
A & I_m & 0\\
A & 0 & -I_m
\end{pmatrix}
=
\begin{pmatrix}
a^\top & 1 & 0\\
a^\top & 0 & -1
\end{pmatrix}
\in\mathbb R^{2\times (n+2)},
\qquad
b'=
\begin{pmatrix}
\tau\\
\tau
\end{pmatrix}\in\mathbb R^{2}.
\]
 The full results can be found in Table~\ref{tab:qkp_all_results}.
In QKP, we have $m=1$ and $B=[n]$ (hence $s=n$). From
Table~\ref{tab:retr_complexity}, we know that APM and GWA have a per-iteration cost of order
$O(nr)$, whereas Newton-type retractions (GWA-Newton, NewtonSLRA, and APHL) have a
per-iteration cost of order $O(nr^2+r^3)$. Therefore, although Newton-type steps are
more expensive per inner iteration, they typically require substantially fewer inner
retraction iterations and fewer outer ALM/BB iterations, and thus can be faster overall
in our QKP experiments. A plausible explanation is that the adaptive retraction tolerance in \eqref{eq:ret_tol}
is relatively loose, and second-order retractions typically reduce the constraint
violation much more aggressively, yielding a sufficiently accurate retracted point in
only a few inner steps. In particular, for relaxed NewtonSLRA, we observe a  linear decrease of the
retraction residual, rather than a clear
quadratic rate. This can be attributed to the fact that the asymptotic quadratic
neighborhood may be very small and the inner iterations are terminated early by the
inexact tolerance \eqref{eq:ret_tol}.
 
{\scriptsize
\begin{longtable}{lrrrrrrrrr}
\caption{QKP instances (ordered by $n$ from small to large): results of different retraction solvers embedded into the same RNNAL outer framework. Here $pobj$ is the final primal objective value, $r$ is the final rank, $(R_p,R_d,R_c)$ are the relative KKT residuals defined in \cite{hou2025low}, and $\overline{\mathrm{Rite}}$ is the average number of inner iterations used by the retraction subsolver. Instances are generated as in \cite{hou2025low}: the density parameter is $p$ and the capacity is fixed to $\tau=0.9\,e^\top a$. Best values in each column except $pobj$ are highlighted in bold within each instance. }
\label{tab:qkp_all_results}\\
\toprule
Method & $pobj$ & $r$ & $R_p$ & $R_d$ & $R_c$ & Time(s) & ALMite & BBite & $\overline{\mathrm{Rite}}$\\
\midrule
\endfirsthead

\toprule
Method & $pobj$ & $r$ & $R_p$ & $R_d$ & $R_c$ & Time(s) & ALMite & BBite & $\overline{\mathrm{Rite}}$\\
\midrule
\endhead

\midrule
\multicolumn{10}{r}{\small Continued on next page.}\\
\endfoot

\bottomrule
\endlastfoot

\midrule
\multicolumn{10}{l}{($n=500$, $p=0.1$)}\\
\midrule
GWA & -1.1422e+06 & 14 & 7.87e-07 & 7.03e-07 & 1.07e-08 & 5.36e+01 & 13 & 917 & \textbf{1.00} \\
GWA-Newton & -1.1422e+06 & \textbf{11} & 6.78e-07 & \textbf{3.83e-07} & 1.88e-09 & 4.72e+00 & 8 & 305 & 1.02 \\
APM & -1.1422e+06 & 12 & \textbf{3.92e-07} & 7.05e-07 & 1.05e-09 & 4.64e+01 & 14 & 857 & 1.14 \\
NewtonSLRA & -1.1422e+06 & 12 & 9.16e-07 & 8.40e-07 & 2.96e-09 & 4.36e+00 & \textbf{7} & \textbf{243} & 1.39 \\
APHL & -1.1422e+06 & \textbf{11} & 8.96e-07 & 8.87e-07 & 4.86e-09 & \textbf{4.30e+00} & \textbf{7} & 263 & 2.05 \\
TAPR & -1.1422e+06 & 12 & 7.63e-07 & 5.90e-07 & 3.85e-09 & 4.66e+00 & 8 & 274 & 2.25 \\
Relaxed NewtonSLRA & -1.1422e+06 & \textbf{11} & 7.33e-07 & 5.64e-07 & \textbf{3.67e-10} & 4.91e+00 & 8 & 329 & 1.78 \\

\midrule
\multicolumn{10}{l}{($n=500$, $p=0.5$)}\\
\midrule
GWA & -5.6677e+06 & 18 & 7.03e-07 & 9.93e-07 & 5.23e-08 & 2.46e+01 & 22 & 892 & 1.06 \\
GWA-Newton & -5.6677e+06 & \textbf{11} & 8.23e-07 & 4.34e-07 & 4.26e-09 & 3.89e+00 & \textbf{7} & 199 & \textbf{1.03} \\
APM & -5.6677e+06 & 15 & \textbf{5.24e-07} & 4.91e-07 & \textbf{4.86e-10} & 1.36e+01 & 10 & 479 & 3.24 \\
NewtonSLRA & -5.6677e+06 & 13 & 8.80e-07 & 2.91e-07 & 3.00e-09 & 4.68e+00 & \textbf{7} & 227 & 1.38 \\
APHL & -5.6677e+06 & 15 & 7.99e-07 & \textbf{1.31e-07} & 8.81e-09 & 4.21e+00 & \textbf{7} & 235 & 1.96 \\
TAPR & -5.6677e+06 & 26 & 7.31e-07 & 3.54e-07 & 6.17e-09 & 3.73e+00 & 8 & \textbf{182} & 2.28 \\
Relaxed NewtonSLRA & -5.6677e+06 & 13 & 7.42e-07 & 3.43e-07 & 2.51e-09 & \textbf{3.63e+00} & \textbf{7} & 222 & 2.01 \\

\midrule
\multicolumn{10}{l}{ ($n=500$, $p=0.9$)}\\
\midrule
GWA & -1.0261e+07 & 18 & 8.13e-07 & 8.59e-07 & 5.40e-08 & 1.44e+01 & 24 & 693 & \textbf{1.01} \\
GWA-Newton & -1.0261e+07 & 21 & \textbf{3.39e-07} & 2.70e-07 & 6.30e-10 & 3.50e+00 & \textbf{4} & 169 & 1.02 \\
APM & -1.0261e+07 & \textbf{14} & 6.13e-07 & 6.45e-07 & 4.16e-08 & 1.14e+01 & 9 & 381 & 13.57 \\
NewtonSLRA & -1.0261e+07 & 27 & 7.81e-07 & 6.14e-07 & \textbf{1.09e-10} & 4.37e+00 & \textbf{4} & 156 & 1.54 \\
APHL & -1.0261e+07 & \textbf{14} & 3.68e-07 & \textbf{2.40e-07} & 7.26e-10 & 4.49e+00 & \textbf{4} & 194 & 2.04 \\
TAPR & -1.0261e+07 & 21 & 8.22e-07 & 2.66e-07 & 2.90e-10 & 3.48e+00 & \textbf{4} & 158 & 2.45 \\
Relaxed NewtonSLRA & -1.0261e+07 & 29 & 5.32e-07 & 5.75e-07 & 3.69e-10 & \textbf{2.98e+00} & \textbf{4} & \textbf{148} & 2.77 \\

\midrule
\multicolumn{10}{l}{($n=1000$, $p=0.1$)}\\
\midrule
GWA & -4.5995e+06 & 29 & 6.50e-07 & \textbf{3.25e-07} & 3.63e-09 & 5.84e+01 & 11 & 893 & \textbf{1.00} \\
GWA-Newton & -4.5995e+06 & 25 & 8.68e-07 & 5.32e-07 & 1.17e-08 & 2.94e+01 & \textbf{7} & 386 & 1.01 \\
APM & -4.5995e+06 & 27 & 9.06e-07 & 3.88e-07 & 4.17e-09 & 5.32e+01 & 8 & 659 & 1.40 \\
NewtonSLRA & -4.5995e+06 & 25 & 9.06e-07 & 5.21e-07 & 2.21e-09 & \textbf{2.26e+01} & \textbf{7} & \textbf{367} & 1.30 \\
APHL & -4.5995e+06 & 25 & 9.59e-07 & 5.76e-07 & 6.02e-09 & 2.81e+01 & \textbf{7} & 434 & 1.83 \\
TAPR & -4.5995e+06 & 25 & \textbf{6.14e-07} & 6.81e-07 & \textbf{1.56e-09} & 3.96e+01 & 9 & 714 & 2.11 \\
Relaxed NewtonSLRA & -4.5995e+06 & \textbf{22} & 7.85e-07 & 3.30e-07 & 5.46e-09 & 2.63e+01 & \textbf{7} & 492 & 1.51 \\

\midrule
\multicolumn{10}{l}{ ($n=1000$, $p=0.5$)}\\
\midrule
GWA & -2.2747e+07 & 33 & \textbf{4.69e-07} & 7.26e-07 & 2.95e-08 & 1.87e+02 & 22 & 1968 & 1.07 \\
GWA-Newton & -2.2747e+07 & \textbf{18} & 5.10e-07 & \textbf{2.52e-07} & \textbf{2.13e-10} & 2.13e+01 & 6 & 301 & \textbf{1.02} \\
APM & -2.2747e+07 & 25 & 9.67e-07 & 7.12e-07 & 5.45e-08 & 7.38e+01 & 8 & 500 & 4.50 \\
NewtonSLRA & -2.2747e+07 & 25 & 8.34e-07 & 2.81e-07 & 2.57e-09 & 1.54e+01 & 6 & 254 & 1.40 \\
APHL & -2.2747e+07 & 19 & 8.01e-07 & 4.33e-07 & 2.75e-09 & 1.52e+01 & \textbf{5} & 235 & 2.09 \\
TAPR & -2.2747e+07 & 23 & 9.09e-07 & 3.94e-07 & 2.58e-09 & \textbf{1.36e+01} & \textbf{5} & 217 & 2.33 \\
Relaxed NewtonSLRA & -2.2747e+07 & 37 & 8.69e-07 & 6.73e-07 & 4.02e-09 & 1.40e+01 & \textbf{5} & \textbf{191} & 2.31 \\

\midrule
\multicolumn{10}{l}{($n=1000$, $p=0.9$)}\\
\midrule
GWA & -4.0935e+07 & 27 & 5.60e-07 & 4.00e-07 & 3.07e-08 & 3.94e+02 & 28 & 1720 & \textbf{1.01} \\
GWA-Newton & -4.0935e+07 & 24 & 8.07e-07 & 3.41e-07 & 8.63e-10 & 1.66e+01 & \textbf{3} & 197 & 1.03 \\
APM & -4.0935e+07 & 25 & 5.04e-07 & 3.10e-07 & 2.54e-08 & 5.66e+01 & 17 & 213 & 4.21 \\
NewtonSLRA & -4.0935e+07 & \textbf{14} & \textbf{2.94e-07} & 1.58e-07 & 6.60e-10 & \textbf{1.14e+01} & 4 & 220 & 1.51 \\
APHL & -4.0935e+07 & 40 & 8.00e-07 & 2.77e-07 & 1.76e-09 & 1.18e+01 & \textbf{3} & \textbf{164} & 2.00 \\
TAPR & -4.0935e+07 & 18 & 8.16e-07 & \textbf{4.39e-08} & \textbf{1.71e-10} & 1.88e+01 & 4 & 253 & 2.36 \\
Relaxed NewtonSLRA & -4.0935e+07 & 17 & 8.17e-07 & 1.32e-07 & 2.84e-10 & 1.20e+01 & 4 & 213 & 3.82 \\

\midrule
\multicolumn{10}{l}{($n=5000$, $p=0.1$)}\\
\midrule
GWA & -1.1387e+08 & 119 & 8.12e-07 & 9.41e-07 & 1.91e-08 & 2.34e+03 & 18 & 868 & \textbf{1.00} \\
GWA-Newton & -1.1387e+08 & 123 & 7.35e-07 & \textbf{3.51e-07} & 8.57e-10 & 5.07e+02 & \textbf{3} & 368 & 1.01 \\
APM & -1.1387e+08 & 102 & 7.62e-07 & 4.23e-07 & 1.86e-08 & 7.11e+02 & 4 & 471 & 2.26 \\
NewtonSLRA & -1.1387e+08 & 115 & \textbf{6.60e-07} & 4.04e-07 & 5.89e-10 & \textbf{2.67e+02} & \textbf{3} & \textbf{291} & 1.54 \\
APHL & -1.1387e+08 & 107 & 6.97e-07 & 4.04e-07 & 1.29e-09 & 2.87e+02 & \textbf{3} & 294 & 2.08 \\
TAPR & -1.1387e+08 & \textbf{101} & 6.95e-07 & 4.01e-07 & 2.09e-10 & 2.84e+02 & \textbf{3} & 313 & 2.43 \\
Relaxed NewtonSLRA & -1.1387e+08 & 118 & 6.80e-07 & 4.27e-07 & \textbf{3.33e-11} & 2.86e+02 & \textbf{3} & 318 & 2.53 \\

\midrule
\multicolumn{10}{l}{($n=5000$, $p=0.5$)}\\
\midrule
GWA & -5.6825e+08 & \textbf{51} & 1.17e-05 & 2.08e-05 & 2.57e-07 & 3.78e+03 & 21 & 1050 & 1.04 \\
GWA-Newton & -5.6825e+08 & 76 & 4.68e-07 & 1.96e-07 & 2.86e-09 & 2.66e+02 & \textbf{2} & 206 & \textbf{1.03} \\
APM & -5.6825e+08 & 83 & \textbf{2.92e-07} & 5.80e-07 & 5.80e-09 & 9.61e+02 & 6 & 298 & 8.70 \\
NewtonSLRA & -5.6825e+08 & 112 & 4.71e-07 & 2.68e-07 & 3.04e-09 & \textbf{1.87e+02} & \textbf{2} & \textbf{205} & 1.74 \\
APHL & -5.6825e+08 & 66 & 4.49e-07 & \textbf{1.39e-07} & 3.06e-09 & 1.89e+02 & \textbf{2} & 206 & 2.06 \\
TAPR & -5.6825e+08 & 73 & 4.55e-07 & 1.58e-07 & 2.90e-09 & 1.92e+02 & \textbf{2} & 227 & 2.52 \\
Relaxed NewtonSLRA & -5.6825e+08 & 143 & 4.54e-07 & 3.77e-07 & \textbf{2.69e-09} & 2.03e+02 & \textbf{2} & 212 & 3.49 \\

\midrule
\multicolumn{10}{l}{ ($n=5000$, $p=0.9$)}\\
\midrule
GWA & -1.0200e+09 & {45} & 1.09e-05 & 1.91e-05 & 2.53e-07 & 4.76e+03 & 24 & 1147 & \textbf{1.09} \\
GWA-Newton & -1.0200e+09 & 143 & 6.11e-07 & 4.19e-08 & 9.88e-10 & 2.56e+02 & \textbf{1} & 165 & 1.16 \\
APM & -1.0200e+09 & \textbf{37} & \textbf{1.23e-07} & \textbf{3.33e-08} & 1.06e-09 & 5.95e+02 & 3 & 278 & 11.19 \\
NewtonSLRA & -1.0200e+09 & 143 & 6.15e-07 & 4.44e-08 & 9.75e-10 & \textbf{1.39e+02} & \textbf{1} & \textbf{158} & 1.95 \\
APHL & -1.0200e+09 & 143 & 6.34e-07 & 4.53e-08 & 1.10e-09 & 1.54e+02 & \textbf{1} & 166 & 2.18 \\
TAPR & -1.0200e+09 & 143 & 5.78e-07 & 5.53e-07 & \textbf{9.53e-10} & 1.54e+02 & \textbf{1} & 171 & 2.89 \\
Relaxed NewtonSLRA & -1.0200e+09 & 143 & 6.06e-07 & 3.94e-08 & 1.05e-09 & 1.66e+02 & \textbf{1} & 164 & 8.38 \\

\end{longtable}
}

\end{document}